\documentclass[12pt]{amsart}
\usepackage[utf8]{inputenc}
\usepackage{enumerate,amssymb,xcolor,comment}
\usepackage[a4paper, footskip=0.5in, headheight = 0.5in, top=1.5in, bottom=1.5in,  right=1.5in,  left=1.5in]{geometry}
\usepackage[colorlinks,linkcolor=blue,citecolor=red,hypertexnames=false]{hyperref}
\usepackage{todonotes}

\newtheorem{question}{Question}[section]
\newtheorem{lemma}[question]{Lemma}
\newtheorem{theorem}[question]{Theorem}
\newtheorem{conjecture}[question]{Conjecture}

\usepackage{xcolor}

\usepackage[foot]{amsaddr}

\usepackage[T1]{fontenc}

\usepackage[leqno]{amsmath}

\makeatletter
\newcommand{\leqnomode}{\tagsleft@true}
\newcommand{\reqnomode}{\tagsleft@false}
\makeatother

\usepackage{latexsym}
\usepackage{amsfonts}

\usepackage{tikz}
\usepackage{float}
\usepackage{lmodern}

\def\dd{\hbox{-}}

\usepackage{marvosym}               

\DeclareMathOperator{\tw}{tw}
\DeclareMathOperator{\width}{width}

\DeclareMathOperator{\Hub}{Hub}
\DeclareMathOperator{\hdim}{hdim}

\DeclareMathOperator{\adh}{adh}
\DeclareMathOperator{\Core}{Core}
\DeclareMathOperator{\soff}{end}
\DeclareMathOperator{\Conn}{Conn}

\newcounter{tbox}

\newcommand{\sta}[1]{\vspace*{0.3cm} \refstepcounter{tbox}
  \noindent{ \parbox{\textwidth}{(\thetbox) \emph{#1}}}\vspace*{0.3cm}}

\newcommand{\mylongtitle}[1]{%
  \ifodd\value{page}%
    \protect\parbox{0.97\linewidth}{#1}\hfill%
  \else%
    \hfill\protect\parbox{0.97\linewidth}{#1}%
  \fi%
}

\makeatletter
\newcommand{\otherlabel}[2]{\protected@edef\@currentlabel{#2}\label{#1}}
\makeatother

\mathchardef\mh="2D

\title[Induced subgraphs and tree decompositions XV.]{Induced subgraphs and tree decompositions\\
XV. Even-hole-free graphs with bounded clique number have logarithmic treewidth}

\author{Maria Chudnovsky$^{\dagger}$}
\author{Peter Gartland$^{\ddagger}$}
\author{Sepehr Hajebi $^{\mathsection}$}
\author{Daniel Lokshtanov$^{\ddagger}$}
\author{Sophie Spirkl$^{\mathsection \parallel}$}

\thanks{$^{\ddagger}$ Department of Computer Science, University of California Santa Barbara, Santa Barbara, CA, USA. Supported by NSF grant CCF-2008838.}
\thanks{$^{\mathsection}$Department of Combinatorics and Optimization, University of Waterloo, Waterloo, Ontario, Canada}
\thanks{$^{\dagger}$Princeton University, Princeton, NJ, USA.  Supported by NSF-EPSRC Grant DMS-2120644 and by AFOSR grant FA9550-22-1-0083.}
\thanks{$^{\parallel}$ We acknowledge the support of the Natural Sciences and Engineering Research Council of Canada (NSERC), [funding reference number RGPIN-2020-03912].
Cette recherche a \'et\'e financ\'ee par le Conseil de recherches en sciences naturelles et en g\'enie du Canada (CRSNG), [num\'ero de r\'ef\'erence RGPIN-2020-03912]. This project was funded in part by the Government of Ontario. This research was completed while Spirkl was an Alfred P. Sloan fellow.}

\date {\today}

\begin{document}
\raggedbottom

\begin{abstract}
  We prove that for every integer $t\geq 1$ there exists an integer $c_t\geq 1$ such
  that
  every $n$-vertex even-hole-free graph with no clique of size $t$ 
has treewidth at most
 $c_t\log{n}$. This resolves a conjecture of Sintiari and Trotignon, who also proved that the logarithmic bound is asymptotically best possible. 
It follows that several \textsf{NP}-hard problems such as \textsc{Stable Set}, \textsc{Vertex Cover}, \textsc{Dominating Set} and \textsc{Coloring} admit polynomial-time algorithms on this class of graphs. As a consequence, for every positive integer $r$, $r$-{\sc Coloring} can be solved in polynomial time on even-hole-free graphs without any assumptions on clique size.

As part of the proof,  we show that there is an integer $d$ such that
every even-hole-free graph has a balanced separator which is contained in the (closed) neighborhood of at most $d$ vertices. This is of independent interest; for instance, it implies the existence of efficient approximation algorithms for certain \textsf{NP}-hard problems while restricted to the class of all even-hole-free graphs.
\end{abstract}

\maketitle

\section{Introduction} \label{intro}
All graphs in this paper are finite and simple.  For a graph $G$, we denote by $V(G)$ the vertex set of $G$, and by $E(G)$ the edge set of $G$.
For graphs $H,G$, we say that $H$ is an
{\em induced subgraph} of $G$ if $V(H) \subseteq V(G)$, and $x,y \in V(H)$ are adjacent in 
$H$ if and only if $xy \in E(G)$.
A {\em hole} in a graph is an induced cycle with at least four vertices. The
{\em length} of a hole is the number of vertices in it. A hole is {\em even} it it has even length, and {\em odd} otherwise.  
The class of even-hole-free graphs has attracted much attention in the past
due to its somewhat tractable, yet very rich structure (see the survey \cite{Kristina}).
In addition to stuctural results, 
much effort was put into designing efficient algorithms for
even-hole-free graphs (to solve problems that are NP-hard in general).
This is discussed in  \cite{Kristina}, while
\cite{AAKST, CdSHV, CGP, Le}
provide examples of more recent work. Nevertheless, 
many open questions remain. Among them is the complexity of several algorithmic problems: \textsc{Stable Set}, \textsc{Vertex Cover}, \textsc{Dominating Set}, $k$-\textsc{Coloring} and \textsc{Coloring}.

For a graph $G = (V(G),E(G))$, a \emph{tree decomposition} $(T, \chi)$ of $G$ consists of a tree $T$ and a map $\chi: V(T) \to 2^{V(G)}$ with the following properties: 
\begin{enumerate}[(i)]
\itemsep -.2em
    \item For every vertex $v \in V(G)$, there exists $t \in V(T)$ such that $v \in \chi(t)$. 
    
    \item For every edge $v_1v_2 \in E(G)$, there exists $t \in V(T)$ such that $v_1, v_2 \in \chi(t)$.
    
    \item For every vertex $v \in V(G)$, the subgraph of $T$ induced by $\{t \in V(T) \mid v \in \chi(t)\}$ is connected.
\end{enumerate}
 
For each $t\in V(T)$, we refer to $\chi(t)$ as a \textit{bag of} $(T, \chi)$.  The \emph{width} of a tree decomposition $(T, \chi)$, denoted by $\width(T, \chi)$, is $\max_{t \in V(T)} |\chi(t)|-1$. The \emph{treewidth} of $G$, denoted by $\tw(G)$, is the minimum width of a tree decomposition of $G$. 

Treewidth, first introduced by Robertson and Seymour in the context of  graph minors, is an extensively studied graph parameter, mostly due to the fact that graphs of bounded treewidth exhibit interesting structural
\cite{RS-GMXVI} and algorithmic \cite{Bodlaender1988DynamicTreewidth} properties.
It is easy to see that large complete graphs and large complete bipartite graphs have large treewidth,
but there are others (in particular a subdivided $k \times k$-wall, which is a planar graph with maximum degree three, and which we will not define  here).
A theorem of  \cite{RS-GMXVI}  characterizes precisely excluding which graphs as minors (and in fact as subgraphs) results in a class of  bounded treewidth. 

Bringing tree decompositions into the world of induced subgraphs is a relatively recent endeavor.
It began in \cite{ST}, where the authors observed yet another evidence of the  structural complexity of the class of even-hole-free graphs: the fact that there exist
even-hole-free graphs of arbitrarily large 
treewidth  (even when large complete subgraphs are excluded).  Closer examination of their  constructions  led the authors of \cite{ST} to make the following two conjectures (the {\em diamond} is the unique simple graph with four vertices and five edges):

\begin{conjecture}[Sintiari and Trotignon \cite{ST}]
\label{diamond}
For every integer $t\geq 1$, there exists a constant $c_t$ such that every
even-hole-free graph $G$ with no induced diamond and no clique of size $t$ satisfies
$\tw(G)\leq c_t$. 
\end{conjecture}

\begin{conjecture}[Sintiari and Trotignon \cite{ST}]
\label{logconj}
For every integer $t\geq 1$, there exists a constant $C_t$ such that every
even-hole-free graph $G$ with no clique of size $t$ satisfies
$\tw(G)\leq C_t \log |V(G)|$. 
\end{conjecture}

Conjecture~\ref{diamond} was resolved in \cite{TWXI}.
Here we prove Conjecture~\ref{logconj}, thus closing the first line of research on induced subgraphs and tree decompositions.

We remark that the construction of \cite{ST} shows that the logarithmic bound of Conjecture~\ref{logconj} is asymptotically best possible. Furthermore, our main result implies that many algorithmic problems that are NP-hard in general (among them 
that  \textsc{Stable Set}, \textsc{Vertex Cover}, \textsc{Dominating Set} and \textsc{Coloring}) admit polynomial-time algorithms in the class of even-hole-free graphs with bounded clique number. As a consequence, for every positive integer $r$, $r$-{\sc Coloring} can be solved in polynomial time on even-hole-free graphs without any assumptions on clique size.

Before we proceed, we introduce some notation and definitions.
Let $G = (V(G),E(G))$ be a graph. For a set $X \subseteq V(G)$, we denote by $G[X]$ the subgraph of $G$ induced by $X$. For $X \subseteq V(G)$, we denote by $G \setminus X$ the subgraph induced by $V(G) \setminus X$. In this paper, we use induced subgraphs and their vertex sets interchangeably.

For graphs $G$ and $H$, we say that $G$ {\em contains} $H$ if some induced subgraph of $G$ is isomorphic to $H$. For a family $\mathcal{H}$ of graphs, $G$
{\em contains}
$\mathcal{H}$ if $G$ contains a member of $\mathcal{H}$, and we say that 
$G$ is {\em $\mathcal{H}$-free} if $G$ does not contain $\mathcal{H}$.
A {\em clique} in a graph is a set of pairwise adjacent vertices, and
a {\em stable set} is a set of vertices no two of which are adjacent.
Let $v \in V(G)$. The \emph{open neighborhood of $v$}, denoted by $N(v)$, is the set of all vertices in $V(G)$ adjacent to $v$. The \emph{closed neighborhood of $v$}, denoted by $N[v]$, is $N(v) \cup \{v\}$. Let $X \subseteq V(G)$. The \emph{open neighborhood of $X$}, denoted by $N(X)$, is the set of all vertices in $V(G) \setminus X$ with at least one neighbor in $X$. The \emph{closed neighborhood of $X$}, denoted by $N[X]$, is $N(X) \cup X$. If $H$ is an induced subgraph of $G$ and $X \subseteq V(G)$, then $N_H(X)=N(X) \cap H$ and $N_H[X]=N_H(X) \cup (X \cap H)$. Let $Y \subseteq V(G)$ be disjoint from $X$. We say $X$ is \textit{complete} to $Y$ if all possible edges with an end in $X$ and an end in $Y$ are present in $G$, and $X$ is \emph{anticomplete}
to $Y$ if there are no edges between $X$ and $Y$.

Given a graph $G$, a {\em path in $G$} is an induced subgraph of $G$ that is a path. If $P$ is a path in $G$, we write $P = p_1 \dd \cdots \dd p_k$ to mean that $V(P) = \{p_1, \dots, p_k\}$, and $p_i$ is adjacent to $p_j$ if and only if $|i-j| = 1$. We call the vertices $p_1$ and $p_k$ the \emph{ends of $P$}, and say that $P$ is a path \emph{from $p_1$ to $p_k$}. The \emph{interior of $P$}, denoted by $P^*$, is the set $V(P) \setminus \{p_1, p_k\}$. The \emph{length} of a path $P$ is the number of edges in $P$.  We denote by $C_k$ a cycle with $k$ vertices.

Next, we describe a few types of graphs that we will need (see Figure \ref{fig:forbidden_isgs}).
\begin{figure}[t!]
\begin{center}
\includegraphics[scale=0.7]{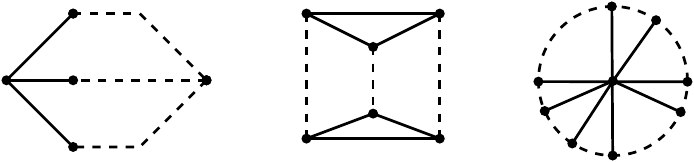}
\end{center}
\caption{Theta, prism and an even wheel. Dashed lines represent paths of length at least one.}
\label{fig:forbidden_isgs}
\end{figure}

A \emph{theta} is a graph consisting of three internally vertex-disjoint paths $P_1 = a \dd \cdots \dd b$, $P_2 = a \dd \cdots \dd b$, and
$P_3 = a \dd \cdots \dd b$, each of length at least 2, such that $P_1^*, P_2^*, P_3^*$ are pairwise anticomplete. In this case we call $a$ and $b$ the
{\em ends} of the theta.

A \emph{prism} is a graph consisting of three vertex-disjoint paths
$P_1 = a_1 \dd \cdots \dd b_1$, $P_2 = a_2 \dd \cdots \dd b_2$, and $P_3 = a_3 \dd \cdots \dd b_3$, each of
length at least 1, such that $a_1a_2a_3$ and $b_1b_2b_3$ are triangles, and no edges exist between the paths except those of the two triangles.

A \emph{pyramid} is a graph consisting of three  paths
$P_1 = a \dd \cdots \dd b_1$, $P_2 = a \dd \cdots \dd b_2$, and $P_3 = a \dd \cdots \dd b_3$ such that $P_1 \setminus \{a\}, P_2 \setminus \{a\}, P_3 \setminus \{a\}$ are pairwise disjoint, at least two of the three paths $P_1, P_2, P_3$ have length at least two, $b_1b_2b_3$ is triangle (called the \emph{base} of the pyramid), and no edges exist between $P_1 \setminus \{a\}, P_2 \setminus \{a\}, P_3 \setminus \{a\}$ except those of the triangle $b_1b_2b_3$. The vertex $a$ is called the \emph{apex} of the pyramid. 

A \emph{wheel} $(H, x)$ in $G$ is a pair where $H$ is a hole and $x$ is a vertex with at least three neighbors in $H$.
A wheel $(H,x)$ is {\em even} if $x$ has an even number of neighbors on $H$.

Let $\mathcal{C}$ be the
class of ($C_4$, theta, prism, even wheel)-free graphs (sometimes called ``$C_4$-free odd-signable'' graphs). For every integer $t \geq 1$, let
$\mathcal{C}_t$ be the class of all graphs in $\mathcal{C}$ with no clique of size $t$. It is easy to see that every even-hole-free graph is in $\mathcal{C}$.

The reader may be familiar with \cite{TWX} where a special case of Conjecture
~\ref{logconj} was proved;
moreover, only one Lemma of \cite{TWX} uses the fact that the set-up there is
not the most general case. At the time, the authors of \cite{TWX}  thought that the
full proof of Conjecture
~\ref{logconj} would follow the general outline of \cite{TWX}, fixing the one missing lemma. That is not what happened. The proof in the present paper takes a
different path: while it  relies  on  insights and a general understanding of the class of even-hole-free graphs gained in \cite{TWX}, and uses several theorems proved there,  a significant
detour is needed.

The first part of the paper is not concerned with treewidth at all. Instead,
we focus on the following question: given two non-adjacent vertices in
a graph in $\mathcal{C}$ of bounded clique number, how many internally
vertex-disjoint paths can there be between them? Given that if instead of
``internally vertex-disjoint'' we say ``with pairwise anticomplete interiors'', then the answer is ``two'', this is somewhat related to the recent work on
the induced Menger theorem \cite{GKL, HNST, NSS}.
Let $k\geq 1$ be an integer and let $a,b \in V(G)$. We say that
$ab$ is a {\em $k$-banana} if
$a$ is non-adjacent to $b$, and there exist $k$ pairwise internally-vertex-disjoint paths from $a$ to $b$. Note that if $ab$ is a $k$-banana in $G$, then
$ab$ is an $l$-banana in $G$ for every $l \leq k$.
We prove:

\begin{theorem}
\label{banana}
For every integer $t\geq 1$, there exists a constant $c_t$ such that 
if $G \in \mathcal{C}_{t}$, then $G$ contains no $c_t \log |V(G)|$-banana.
\end{theorem}

The next step in the proof of Conjecture~\ref{logconj} is the following.
Let $c \in [0, 1]$ and let $w$ be a normal weight function on $G$, that is, $w : V(G) \rightarrow \mathbb{R}_{\geq 0}$ satisfies $\sum_{v \in V(G)} w(v) = 1$. A set $X \subseteq V(G)$ is a {\em $(w,c)$-balanced separator} if
$w(D) \leq  c$ for every component $D$ of $G \setminus X$. The set $X$ is a {\em $w$-balanced separator} if $X$ is a $(w,\frac{1}{2})$-balanced separator.
We show:

\begin{theorem}
  \label{balancedsep}
There is an integer  $d$ with the following property.
Let $G \in \mathcal {C}$ and let $w$ be a normal weight function on $G$.
Then there exists $Y \subseteq V(G)$ such that
\begin{itemize}
\item $|Y| \leq d$, and
\item $N[Y]$ is a $w$-balanced separator in $G$.
\end{itemize}
\end{theorem}

We then use Theorem~\ref{banana} and Theorem~\ref{balancedsep} to prove our main result:

\begin{theorem}
\label{main}
For every integer $t\geq 1$, there exists a constant $c_t$ such that every
 $G \in \mathcal{C}_{t}$ satisfies
$\tw(G)\leq c_t \log |V(G)|$. 
\end{theorem}

\subsection{Proof outline and organization}

We only include a few  definitions in this section; all the necessary definitions appear in later parts of the paper.
Our first goal is to prove Theorem~\ref{banana}. Let $a,b \in V(G)$
be non-adjacent.
Recall that a {\em separation} of  a graph  $G$ is a triple $(X,Y,Z)$ of pairwise disjoint
  subsets of $G$ with $X \cup Y \cup Z = G$ such that $X$ is anticomplete to
  $Z$.
Similarly to \cite{TWI}, we use the fact that graphs in $\mathcal{C}_t$
admit a natural family of separations that correspond to special vertices of
the graph called
``hubs'' and are discussed in Section~\ref{cutsets}.
Unfortunately, the interactions between these separations may be complex, and,
similarly to  \cite{TWIII}, we use degeneracy to partition the set
of all hubs other than $a,b$
(which yields a partition of all the natural separations)
of an $n$-vertex  graph $G$ in $\mathcal{C}_t$  into collections $S_1, \dots, S_p$, where each $S_i$ is ``non-crossing''
(this property is captured in Lemma \ref{looselylaminar}), $p \leq C(t) \log n$ (where $C(t)$ only depends on $t$ and works for all $G \in \mathcal{C}_t$)
and each vertex of  $S_i$ has at most $d$ (where $d$ depends on  $t$)
neighbors in
$\bigcup_{j=i}^p S_j$. We prove a strengthening of Theorem~\ref{banana}, which
asserts that the size of the largest $ab$-banana is bounded by a linear function of $p$.

We  observe that a result of \cite{TWX} implies that there  exists a
an integer  $k$ (that depends on $t$) such
that if $G \in \mathcal{C}_t$ has no  hubs
other than $a,b$, then $ab$ is not a $k$-banana in $G$.
More precisely, the result of \cite{TWX} states that if $a$ and $b$
are joined by $k$ internally-vertex-disjoint paths $P_1, \dots, P_k$, then for at least one $i \in \{1, \dots, k\}$, the neighbor of
$a$ in $P_i$ is a hub. 

We now proceed as follows. Let $m=2d+k$ (where $d$ comes from the degeneracy partition and $k$ from the previous paragraph);
suppose that $ab$ is a $(4m+2)(m-1)$-banana in $G$ and let $\mathcal{P}$ be the set
of all paths of $G$ with ends $a,b$.

Let $S_1, \dots, S_p$ denote the partition of hubs as decribed above. We proceed by induction on $p$ and describe a process below that finds an induced subgraph $H$ of $G$ in which: 
\begin{itemize}
    \item Vertices in $S_1$ are no longer hubs; 
    \item If there are not many internally disjoint $ab$-paths in $H$, we can ``lift'' this to $G$.
\end{itemize}

We first consider a
so-called $m$-lean tree decomposition $(T, \chi)$ of $G$ (discussed in Section \ref{linkedandlean}). By examining the intersection graphs of subtrees of a tree we deduce that there exist   two (not necessarily distinct) vertices
$t_1,t_2 \in V(T)$
such that for every $ab$ path $P$, the interior of $P$
meets $\chi(t_1) \cup \chi (t_2)$. We also argue that $a,b \in \chi(t_1) \cup \chi(t_2)$. A vertex $u$ of $S_1$ is {\em bad} if $u$ has large
degree  (at least $D = 2m(m-1)$)
in the torso of $\chi(t_1)$, or
$u$ has large degree in 
  the torso of $\chi(t_2)$, or $u$ is adjacent to both $a$ and $b$.
  We show that there are at most three bad vertices in $S_1$.

We would like to bound the size of the largest
$ab$-banana in the union of the two torsos.
Unfortunately, the torso of $\chi(t_i)$ may not be a graph in $\mathcal{C}_t$.
Instead, we find an induced subgraph of $G$, which we call $\beta$, that
consists of $\chi(t_1) \cup \chi(t_2)$ together with a collection of disjoint vertex sets
$\Conn(t_i,t)$ for $t \in V(T) \setminus \{t_1,t_2\}$ and $i \in \{1,2\}$ except vertices $t$ on the $t_1\dd t_2$-path in $T$, where each $\Conn(t_i,t)$ ``remembers'' 
 the component of $G \setminus (\chi(t_1) \cup \chi(t_2))$ that
 meets $\chi(t)$.  Importantly, no vertex of
 $\beta \setminus (\chi(t_1) \cup \chi(t_2))$ is a hub of $\beta$, 
 and all but at most three  vertices  of $S_1$ have bounded degree in
 the union of the torso of $\chi(t_1)$ and the torso of $\chi(t_2)$.

We next  decompose
$\beta$, simultaneously, by all the separations corresponding to the hubs in $S_1$ that are not bad, and delete the set of (at most three) bad vertices of $S_1$. We denote the resulting graph  by $\beta^A(S_1)$ and call it 
    the ``central bag'' for $S_1$. The parameter $p$
  is smaller for $\beta^A(S_1)$ than it is for $G$, and so we can use induction
  to obtain a bound on the largest size of an $ab$-banana in $\beta^A(S_1)$.
  Since our goal is to bound the size of an $ab$-banana in $G$
  by a linear function of $p$, we now need to show that the
  size of the largest $ab$-banana does not grow by more than an additive constant when we move from $\beta^A(S_1)$ to $G$.
  
  We start with  a smallest subset $X$    of $\beta^A(S_1)$ that separates $a$ from $b$ in $\beta^A(S_1)$ (and whose size is bounded by Menger's theorem)
  and show how to transform
  in into a set  $Y$ separating $a$ from $b$ in $\beta$, while increasing the size of its intersection with $\chi(t_1) \cup \chi(t_2)$ by at most
an additive constant, and ensuring  
  the number of
  sets $\Conn(t_i,t)$ that $Y$ meets is bounded by a constant.
    Then we repeat a similar procedure to obtain a set $Z$ of vertices
  separating
  $a$ from $b$ in $G$, making sure that the increase in size is 
  again bounded by  an additive constant.

  Let us now discuss how we obtain the bound on the growth of the separator.
  In the first step, to obtain $Y$, we add to $X$   
  the neighbor sets of the vertices of $S_1  \cap X$.
  Since while constructing $\beta^A(S_1)$ we have deleted all the bad vertices of $S_1$, the number of vertices of $\chi(t_1) \cup \chi(t_2)$  added  to $X$ is at most
  $2|S_1 \cap X|D$, and $Y \setminus X$ meets at most
  $2|S_1 \cap X|D$ of the sets $\Conn(t_i,t)$.
  Note that the bound on the size of $X$ that we have depends on $p$,
  which may be close to $\log |V(G)|$, so another argument is
  needed to obtain a constant bound on $|S_1 \cap X|$ and on the
  number of sets $\Conn(t_i,t)$ that meet $Y$.
This is a consequence of 
    Theorem \ref{bound}, because no vertex of $S_1$ is a hub in
    $\beta^A(S_1)$  (this is proved in Theorem \ref{A_centralbag}),
    and no vertex of $\Conn(t_i,t)$ is a hub in $\beta$.
    (The proof of Theorem~\ref{bound} analyzes the structure of minimal
    separators in graphs in $\mathcal{C}$ using tools developed in Section~\ref{sec:connectifier}.)

    In the second growing step, we start with the set $Y$ obtained in
    the previous paragraph.
    Then we add to $Y$ the following subsets (here we describe the cases
    when $t_1 \neq t_2$; if $t_1=t_2$ the argument is similar).
    \begin{enumerate}
    \item $\chi(t_1) \cap \chi(t_1')$ where $t_1'$ is the unique
      neighbor of $t_1$ in the path in $T$ from $t_1$ to $t_2$.
 \item $\chi(t_2) \cap \chi(t_2')$ where $t_2'$ is the unique
   neighbor of $t_2$ in the path in $T$ from $t_1$ to $t_2$.
 \item $\chi(t_1) \cap \chi(t)$ for every $t \in N(t_1) \setminus \{t_1'\}$
   such that $Y \cap \Conn(t_1,t) \neq \emptyset$.
   \item  $\chi(t_2) \cap \chi(t)$ for every $t \in N(t_2) \setminus \{t_2'\}$
   such that $Y \cap \Conn(t_2,t) \neq \emptyset$.
 \item The set of all bad vertices of $S_1$. 
    \end{enumerate}
    One of the properties
    of $m$-lean tree decompositions is that the size of each adhesion (intersection of ``neighboring'' bags)
    is less than $m$.
    The number of adhesions added to $Y$
    is again bounded since the number of distinct sets
    $\Conn(t_i,t)$ that meet $Y$ is bounded.
        This completes the proof of Theorem~\ref{banana}.

        The next key ingredient in the proof of Theorem~\ref{main} is  Theorem~\ref{balancedsep}, asserting that         there is an integer $C$
    such that for every graph $G \in \mathcal{C}$ and every normal weight function
    $w$ on $G$, there is a $w$-balanced separator $X$ in $G$  such that
    $X$ is contained in the union of the neighborhoods of at most $C$ vertices of $G$. To prove that, we first prove a decomposition theorem, similar to the
    one giving us the natural separations associated with hubs, but this time
    the separations come from pyramids in $G$. We then use the two kinds of
    separations (the kind that come from hubs and the kind that come from pyramids) as follows.

    For $X \subseteq V(G)$ we say that a set $Y \subseteq V(G)$ {\em dominates} $X$
    if $X \subseteq N[Y]$.
    By Lemma \ref{lem:gyarfas} (from  \cite{QPTAS}) there is a path $P=p_1 \dd \cdots \dd p_k$
    in $G$ such that
    $w(D) \leq \frac 1 2$ for every component $D$ of $G \setminus N[P]$.
    We now use a ``sliding window'' argument: we divide $P$ into subpaths, each with the property that its neighborhood can be dominated by a small, but not too small, set of vertices. Using the decomposition theorems above, we find a bounded-size set $X(W)$ such that, except for our window $W$, the path $P$ is disjoint from $N[X(W)]$, and $N[X(W)]$ separates the subpath of $P$ before our current window from that after our current window. This means (unless $N[X]$ is a balanced separator) that the big component of $G \setminus N[X]$ does not contain either the subpath of $P$ before our window, or the subpath after our window. By looking at the point in the path where this answer changes from ``before'' to ''after'' (and showing that such a point exists), and by combining the separators for the two windows before and after this point as well as small dominating sets for the neighborhood of those two windows, we are able to find a $w$-balanced separator with bounded domination number. 
     This completes the proof of Theorem~\ref{balancedsep}. We remark
     that Theorem~\ref{balancedsep} applies to all graphs in $\mathcal{C}$ and does not assume a bound on the clique number. 

   The next step in the proof of Theorem~\ref{main} is to prove
     Theorem~\ref{bettersep}, asserting that for every integer $L$,
     if a graph  $G$ contains no $L$-banana, then for
     every normal weight function $w$ on $G$, if $G$ has a $w$-balanced separator $N[X]$, then $G$ has an $w$-balanced separator
     $Y$ and a clique $T \subseteq X$ such that 
     $|Y \setminus T| \leq 3|X|L$. 
          This step uses $3L$-lean
          tree decompositions of $G$  and works for all graphs $G$.

     Now Theorem~\ref{bettersep} and Theorem~\ref{banana} together
     imply that for every $t$, there exists an integer  $q$,
     depending on $t$, such that for every $G \in \mathcal{C}_t$,
     $G$ has a balanced separator of size at most $q \log |V(G)|+t$;
     that is Theorem~\ref{smallsep}.
     By Lemma~\ref{lemma:bs-to-tw} this immediately implies the desired bound
     on the treewidth of $G$.

    The paper is organized as follows. In Section~\ref{linkedandlean},
    we discuss lean tree decompositions and their properties,
    along with other classical results in graph theory. For some of them
    we prove variations tailored specifically to our needs.
        In Section~\ref{cutsets} we summarize 
        results guaranteeing the existence of separations associated with hubs.
        We also establish a stronger version of Theorem~\ref{banana} for the case when the set of hubs in $G$ is very restricted.
  In Section~\ref{sec:centralbag_banana} we discuss the construction of
  the graphs $\beta$ and $\beta^A(S_1)$, and how to use
$ab$-separators there to obtains $ab$-separators in $G$.
  In Section~\ref{sec:connectifier}  we analyze the structure of minimal separators in graphs of $\mathcal{C}$.
  In Section~\ref{boundnonhubs}  we use the results of Section~\ref{sec:connectifier}
  to obtain a bound on the size of a stable set of non-hubs in
  an $ab$-separator of smallest size.
  Section~\ref{sec:bananaproof} puts together the results of all the previous sections
  to prove Theorem \ref{banana}.
  The goal of  Section~\ref{sec:domsep} is to prove  Theorem~\ref{balancedsep}.
  We start with Lemmas~\ref{lemma:proper_wheel_forcer} and \ref{apexnbrs}
  to establish the existence of certain decompositions in 
  graphs of $\mathcal{C}$, and then proceed with the sliding window argument.
  Section~\ref{sec:smallsep} is devoted to the proof of Theorem~\ref{smallsep}.
The proof of Theorem~\ref{main} is completed in Section~\ref{sec:proof}.  
  Finally, Section~\ref{algsec}
  discusses algorithmic consequences of 
  Theorem~\ref{balancedsep} and Theorem~\ref{main}.

\section{Special tree decompositions and connectivity} \label{linkedandlean}

In this section we have collected several results from the literature that we
need; for some of them we have also proved our own versions, tailored
specifically to our needs. Along the way we also introduce some notation.

\subsection{Connectivity}
We start with a result on connectivity. For two vertices $u,v \in G$ and a set $X \subseteq G \setminus \{u,v\}$
we say that $X$ {\em separates} $u$ from $v$ if $P^* \cap X \neq \emptyset$
for every path $P$ of $G$ with ends $u$ and $v$. The following is a well-known variant of a classical theorem due to Menger \cite{Menger}: 

\begin{theorem}[Menger \cite{Menger}]\label{Menger_vertex}
   Let $k\geq 1$ be an integer, let $G$ be a graph and let $u,v \in G$ be distinct and non-adjacent. Then either there exists a set $M\subseteq G\setminus \{u,v\}$ with $|M|<k$ such that $M$ separates $u$ and $v$ in $G$, or $uv$ is a $k$-banana in $G$.
\end{theorem}

\subsection{Lean tree decompositions}

We adopt notation from \cite{TWX}: For a tree $T$ and vertices $t,t' \in V(T)$,
we denote by $tTt'$ the unique path in $T$ from $t$ to $t'$.
Let $(T, \chi)$ be a tree decomposition of a graph $G$.
For every $uv \in E(T)$, the {\em adhesion at $uv$}, denoted by $\adh(u,v)$, is the set $\chi(u) \cap \chi (v)$. For $u,v \in T$ such that $uv \not \in E(T)$
(in particular, if $u=v$), we set $\adh(u,v)=\emptyset$.
We define $\adh(T, \chi)=\max_{uv \in E(T)} |\adh(u,v)|$. For every $x\in V(T)$,  the \textit{torso at $x$}, denoted by $\hat{\chi}(x)$, is the graph obtained from the bag $\chi(x)$ by, for each $y\in N_T(x)$, adding an edge between every two non-adjacent vertices $u,v\in \adh(x,y)$.

In the proof of Theorem \ref{banana} and Theorem \ref{main}, we will use
a special kind of a tree decomposition that we present next.
Let $k>0$ be an integer. A tree decomposition $(T, \chi)$ is called {\em $k$-lean} if the following hold: 
\begin{itemize}
    \item $\adh(T,\chi) < k$; and 
    \item for all $t,t' \in V(T)$ and sets $Z \subseteq \chi(t)$ 
and $Z' \subseteq \chi(t')$ with $|Z|=|Z'| \leq k$, either $G$ contains $|Z|$ disjoint paths from $Z$ to $Z'$, or some edge $ss'$ of $tTt'$ satisfies $|\adh(s,s')| <|Z|$.
\end{itemize}

For a tree $T$ and an edge $tt'$ of $T$, we  denote by
$T_{t \rightarrow t'}$ the
 component of $T \setminus t$ containing $t'$. Let
 $G_{t \rightarrow t'}=G[\bigcup_{v \in T_{t \rightarrow t'}} \chi(v)]$.
    A tree decomposition
$(T,\chi)$ is {\em tight} if for every edge $tt' \in E(T)$  there is a component
$D$ of $G_{t \rightarrow t'} \setminus \chi(t)$ such that
$\chi(t) \cap \chi(t') \subseteq N(D)$ (and therefore $\chi(t) \cap \chi(t') = N(D)$). 

   Next, we need a definition from  \cite{BD}. Given a tree decomposition $(T, \chi)$ of an $n$-vertex graph $G$, its \emph{fatness} is the vector $(a_n, \dots, a_0)$ where $a_i$ denotes the number of bags of $T$ of size $i$. A tree decomposition $(T, \chi)$ of $G$ is \emph{$k$-atomic} if $\adh(T, \chi) < k$ and the fatness of $(T, \chi)$ is lexicographically minimum among all tree decompositions of $G$ with adhesion less than $k$.
 
It was observed in \cite{kblockpaper} that 
\cite{BD} contains a proof of the following:

\begin{theorem} [Bellenbaum and Diestel \cite{BD}, see 
 Carmesin, Diestel, Hamann, Hundertmark \cite{kblockpaper}, see also Wei{\ss}auer \cite{Weissauer}]
  \label{atomictolean}
  Every $k$-atomic tree decomposition is $k$-lean.
\end{theorem}

We also need:
\begin{theorem} [Wei{\ss}auer \cite{Weissauer}]
  \label{atomictotight}
  Every $k$-atomic tree decomposition is tight. 
\end{theorem}

Using ideas similar to those of Wei\ss auer \cite{Weissauer} and using Theorems \ref{Menger_vertex} and \ref{atomictolean}, we prove: 

\begin{theorem}
\label{smallsepnonsep}
  Let $G$ be a graph, let $k \geq 1$ and let $(T, \chi)$ be
  $3k$-atomic tree decomposition of $G$. Let $t_0 \in V(T)$ and let 
  $u,v \in G$. Assume that $N[u]$ is not separated from $\chi(t_0) \setminus u$
  by a set of size less than $3k$, and that $N[v]$ is not separated from $\chi(t_0) \setminus v$
  by a set of size less than $3k$. Then $u$ is not separated from $v$ by a set of
  size less than $k$, and consequently $uv$ is a $k$-banana. 
\end{theorem}

\begin{proof}
  By Theorems \ref{atomictolean} and \ref{atomictotight}, we have that $(T,\chi)$ is  tight and $3k$-lean.
  Suppose that $u$ is separated from $v$ by a set of size less than $k$. 
By Theorem~\ref{Menger_vertex}, there exists a set $\mathcal{P}_u$ of pairwise vertex-disjoint (except $u$) paths, each with one end $u$ and the other end in $\chi(t_0) \setminus u$, and such that
 $|\mathcal{P}_u|=3k$. Let $Z_u$ be the set of the ends of members of $\mathcal{P}_u$ in $\chi(t_0)$. Similarly, 
there exists a collection  $\mathcal{P}_v$ of pairwise vertex-disjoint (except $v$) paths, each with one end $v$ and the other end in $\chi(t_0) \setminus v$, and such that
$|\mathcal{P}_v|=3k$. Let $Z_v$ be the set of the ends of members of $\mathcal{P}_v$ in $\chi(t_0)$. Since $(T,\chi)$ is $3k$-lean,
there is a collection $\mathcal{Q}$ of pairwise vertex-disjoint paths
from $Z_u$ to $Z_v$. 
Let $X$ be a set with  $|X| <k$ separating $u$ from $v$.
Then,  for every $u$-$v$ path $P$  in $G$ we have
$P^* \cap X \neq \emptyset$.
But since $|Z_u|=|Z_v|=|\mathcal{Q}|=3k$,
there is a path $Q \in \mathcal{Q}$ with ends $z_u \in Z_u$ and $z_v \in Z_v$, a path
$P_u \in \mathcal{P}_u$ from $u$ to $z_u$, and a path $P_v \in \mathcal{P}_v$
from $v$ to $Z_v$ such that $X \cap (Q \cup (P_u \setminus u)  \cup (P_v\setminus v))=\emptyset$, contrary to the fact that $u \dd P_u \dd z_u \dd Q \dd z_v \dd P_v \dd v$ is a path from $u$ to $v$ in $G$.
This proves the first statement of the theorem. The second statement
follows immediately by Theorem~\ref{Menger_vertex}.
\end{proof}

We finish this subsection with  a theorem about tight tree decompositions in theta-free graphs that was proved in \cite{TWX}. Note that by Theorem \ref{atomictotight}, the following result applies in particular to $k$-atomic tree decompositions for every $k$.

 A \emph{cutset} $C \subseteq V(G)$ of $G$ is a (possibly empty) set of vertices such that $G \setminus C$ is disconnected. A {\em clique cutset} is a cutset that is a clique.
\begin{theorem}[Abrishami, Alecu, Chudnovsky, Hajebi, Spirkl \cite{TWX}]
  \label{connectedbranches}
          Let $G$ be a theta-free graph and assume that $G$ does not admit a clique
  cutset.
  Let $(T, \chi)$ be a tight tree decomposition of $G$.
  Then for every  edge $t_1t_2$ of $T$ the graph
  $G_{t_1\rightarrow t_2} \setminus \chi(t_1)$ is connected and
   $N(G_{t_1\rightarrow t_2} \setminus \chi(t_1))=\chi(t_1) \cap \chi(t_2)$.
  Moreover, if $t_0,t_1,t_2 \in V(T)$ and $t_1,t_2 \in N_T(t_0)$, then
  $\chi(t_0) \cap \chi(t_1) \neq \chi(t_0) \cap \chi(t_2)$.
\end{theorem}

\subsection{Catching a banana}
In this subsection
we discuss another important feature of tree decompositions that is needed in the proof of Theorem~\ref{banana}.

\begin{theorem}
  \label{basket}
  Let $G$ be a theta-free graph  and  let $a,b \in V(G)$. Let $(T, \chi)$ be a
  tree decomposition of $G$. Then there exist $t_1,t_2 \in V(T)$ (not
  necessarily distinct) such that
  for every path $P$ of $G$ with ends $a,b$ we have that
  $(\chi(t_1) \cup \chi(t_2)) \cap P^* \neq \emptyset.$ 
  Moreover, if $D$ is a component of $G \setminus (\chi(t_1) \cup \chi(t_2))$,
  then $|N(D) \cap \{a,b\}| \leq 1$.
\end{theorem}

\begin{proof}
  Let $\mathcal{P}$ be the set of all paths of $G$ with ends $a,b$. For every
  $P \in \mathcal{P}$ let
  $T(P)=\{t \in T\ \; : \; \chi(t) \cap P^* \neq \emptyset\}  $.
  Since $P^*$ is a connected subgraph of $G$, it follows that $T[T(P)]$
  is a subtree of $T$.

  Let $H$ be a graph with vertex set $\mathcal{P}$ and such that
  $P_1P_2 \in E(H)$ if and only if $T(P_1) \cap T(P_2) \neq \emptyset$.
  Then, by the main result of \cite{subtrees}, $H$ is a chordal graph.

  \sta{If $P_1, P_2 \in H$ are non-adjacent in $H$, then
    $P_1^*$, $P_2^*$ are disjoint and anticomplete to each other in $G$.
  \label{disjoint}}

  Suppose that there exists   $v \in P_i^* \cap P_j^*$. Then
  for every $t \in T$ such that $v \in \chi(T)$, we have $t \in T(P_1) \cap T(P_2)$, contrary to the fact that $P_1,P_2$ are non-adjacent in $H$.
  Next suppose that there is an edge $v_1v_2$ of $G$ such that
  $v_i \in P_i^*$. 
  Let $t \in T$ be such that $v_1,v_2 \in \chi(t)$. Then $t \in T(P_1) \cap T(P_2)$, again contrary to the fact  that $P_1$ is non-adjacent to $P_2$ in
  $H$. This proves~\eqref{disjoint}. 
  
  \sta{$H$ has no stable set of size three. \label{stable}}

  Suppose $P_1,P_2,P_3$ is a stable set in $H$. By \eqref{disjoint}
  the sets $P_1^*, P_2^*, P_3^*$ are pairwise disjoint and anticomplete to
  each other in $G$. But now $P_1 \cup P_2 \cup P_3$ is a theta with ends $a,b$
  in $G$, a contradiction. This proves~\eqref{stable}.
  \\
  \\
  Since $H$ is chordal, it follows from \eqref{stable} and a result from \cite{Golumbic}  that there exist
  cliques $K_1,K_2$ of $H$ such that $V(H)=K_1 \cup K_2$. Again by  \cite{Golumbic} we deduce that   for each $i \in \{1,2\}$, there exists  $t_i \in K_i$ such that
  $t_i \in T(P)$ for every $P \in K_i$. Consequently,
  $(\chi(t_1) \cup \chi(t_2)) \cap P^* \neq \emptyset$ for every
  $P \in \mathcal{P}$, and the first statement of the theorem holds.

  To prove the second statement, suppose $a,b \in N(D)$ for some component
  $D$ of $G \setminus (\chi(t_1) \cup \chi(t_2))$. Then
  $D \cup \{a,b\}$ is connected, and so there is a path $P$
  from $a$ to $b$ with $P^* \subseteq D$. But now
  $P^* \cap (\chi(t_1) \cup \chi(t_2))=\emptyset$, a contradiction.
  This completes the proof.
  \end{proof}
  
\subsection{Centers of tree decompositions.}

We finish this section with a well-known theorem about tree decompositions. Recall that for a graph $G$, a function  $w: V(G) \rightarrow [0,1]$ is a {\em normal weight function} on $G$ if $w(V(G))=1$, where we use the notation $w(X)$ to mean $\sum_{x \in X} w(x)$ for a set $X$ of vertices. 

Let $G$ be a  graph and let $(T, \chi)$ be a tree decomposition of $G$.
Let $w: V(G) \rightarrow [0,1]$ be a normal weight function on $G$.
A vertex $t_0$ of $T$ is a  {\em center} of $(T, \chi)$ if for every
$t' \in N_T(t_0)$ we have
$w(G_{t_0 \rightarrow t'} \setminus \chi(t_0)) \leq \frac{1}{2}$. 

The following is well-known; a proof can be found in \cite{TWX}, for example. 
\begin{theorem}
  \label{center}
  Let $(T,\chi)$ be a tree decomposition of a graph $G$. Then $(T, \chi)$ has a center.
\end{theorem}

\section{Wheels and star cutsets} \label{cutsets}

Recall that a \emph{wheel} in $G$ is a pair $(H, x)$ such that $H$ is  a hole
and $x$ is a vertex that has at least three neighbors in $H$. Wheels play an
important role in the study of even-hole-free and odd-signable graphs. Graphs in this classes that contain no wheels are much easier to handle; on the other hand, the presence of a wheel forces the graph to admit a certain kind of
decomposition.

A \emph{sector} of $(H,x)$ is a path $P$ of $H$ whose ends are distinct and adjacent to $x$, and such that $x$ is anticomplete to $P^*$. A sector $P$ is a \emph{long sector} if $P^*$ is non-empty.  We now define several types of wheels that we will need. 

A wheel $(H, x)$
is a \emph{universal wheel} if $x$ is complete to $H$. A wheel $(H, x)$ is a \emph{twin wheel} if $N(x) \cap H$ induces a path of length two.
A wheel $(H, x)$ is a \emph{short pyramid} if $|N(x) \cap H| = 3$ and $x$ has exactly two adjacent neighbors in $H$. A wheel is \emph{proper} if it is not a twin wheel or a short pyramid. We say that $x \in V(G)$ is a {\em hub} if there exists $H$ such that 
$(H, x)$ is a proper wheel in $G$.
 We denote by $\Hub(G)$ the set of all hubs of $G$.

We need the following result, which was observed in \cite{TWI}:

\begin{theorem}[Abrishami, Chudnovsky, Vu\v{s}kovi\'c \cite{TWI}]\label{wheelstarcutset}
  Let $G \in \mathcal{C}$   and let $(H,v)$ be a proper  wheel in $G$.
  Then there is no component  $D$ of $G \setminus N[v]$
such that $H \subseteq N[D]$. 
\end{theorem}

We will revisit this result in Section~\ref{sec:domsep}.
A {\em star cutset} in a graph $G$ is a cutset $S\subseteq V(G)$ such that  either $S=\emptyset$ or for some $x\in S$, $S\subseteq N[x]$.
A large portion of this paper is devoted to dealing with hubs and star cutsets arising from them in graphs in $\mathcal{C}$, but in the remainder of this section we focus on the case when $\Hub(G)$ is very restricted.
To do that, we use a result from \cite{TWX}:

\begin{theorem}[Abrishami, Alecu, Chudnovsky, Hajebi, Spirkl \cite{TWX}]
    \label{bananatohub}
  For every  integer $t \geq 1$ there exists an integer $k=k(t)$ such that
  if $G \in \mathcal{C}_t$ and $xy$ is a $k$-banana in $G$, then
      $N(x) \cap \Hub(G) \neq \emptyset$ and $N(y) \cap \Hub(G) \neq \emptyset$.
  \end{theorem}

We immediately deduce:

\begin{theorem}
\label{banana_emptyhub}
For every integer $t\geq 1$, there exists a constant $k=k(t)$ with the following
property.
Let $G \in \mathcal{C}_{t}$ and let  $a,b \in V(G)$ be non-adjacent.
Assume that $Hub(G) \subseteq \{a,b\}$.
Then $ab$ is not a $k$-banana in $G$.
In particular, if $Hub(G) = \emptyset$, there is no $k$-banana in $G$.
\end{theorem}

We will also use the following:
 \begin{theorem} [Abrishami, Alecu, Chudnovsky, Hajebi, Spirkl \cite{TWX}]
   \label{noblocksmalltw_Ct}
  For every  integer $t\geq 1$, there exists an integer $\gamma=\gamma(t)$ such that every $G \in \mathcal{C}_t$ with $\Hub(G)=\emptyset$ satisfies $\tw(G) \leq \gamma-1$.
\end{theorem}
    
\section{Stable   sets  of safe hubs} \label{sec:centralbag_banana}

As we discussed in Section~\ref{intro}, in the course of the proof of
Theorem~\ref{banana}, we will repeatedly   decompose the graph by star cutsets
arising from a stable set of appropriately chosen hubs (using Theorem~\ref{wheelstarcutset}).
In this section, we prepare the tools for handling one such step: a stable set of safe hubs.

Throughout this section, we fix the following: Let $t, d \in \mathbb{N}$ and let $G \in \mathcal{C}_t$ be a graph with $|V(G)|=n$. 
Let $m=k+2d$ where $k = k(t)$ is as in Theorem \ref{bananatohub}.
Let $a, b \in V(G)$ such that $ab$ is a $(4m+2)(m-1)$-banana in $G$, and let $\mathcal{P}$ be the set
of all paths in $G$ with ends $a, b$.
Let $(T,\chi)$ be an $m$-atomic
tree decomposition of $G$. 
By Theorems \ref{atomictolean} and \ref{atomictotight}, we have that $(T,\chi)$ is  tight and $m$-lean.
By Theorem \ref{basket}, there exist $t_1,t_2 \in T$ such that
$P^* \cap (\chi(t_1) \cup \chi(t_2)) \neq \emptyset$ for every $P \in \mathcal{P}$.
We observe:

\begin{lemma}
  \label{abinbasket}
        { We have $a,b \in \chi(t_1) \cup \chi(t_2)$.}
\end{lemma}

\begin{proof}
  Suppose that $a \not \in \chi(t_1) \cup \chi(t_2)$.
  Let $D$ be the component of $G \setminus (\chi(t_1) \cup \chi(t_2))$ such that
  $a \in D$. Then $N(D)$ is contained in the union of at most two adhesions of
  $(T,\chi)$, at most one for each of $\chi(t_1)$ and $\chi(t_2)$.
  Since  $(T,\chi)$ is $m$-lean, it follows that  $|N(D)|\leq 2(m-1)$.
  Since $P^* \cap (\chi(t_1) \cup \chi(t_2)) \neq \emptyset$ for every $P \in \mathcal{P}$, it follows that $P^* \cap N(D) \neq \emptyset$ for every
  $P \in \mathcal{P}$. But then $\mathcal{P}$ contains at most $|N(D)|$
  pairwise internally vertex-disjoint paths, contrary to the fact that
  $ab$ is a $(4m+2)(m-1)$-banana in $G$.
  \end{proof}

We say that
    a vertex $v$ is {\em $d$-safe} if $|N(v) \cap \Hub(G)| \leq d$.
The goal of the next two lemmas is to classify $d$-safe vertices into
``good ones'' and ``bad ones,'' and show that the bad ones are rare.
Let $t_0 \in V(T)$. A vertex $v \in V(G)$ is {\em $t_0$-cooperative}
if either $v \not\in \chi(t_0)$, or
$\deg_{\hat\chi(t_0)}(v) < 2m(m-1)$. 
It was shown in \cite{TWX} that
$t_0$-cooperative vertices have the
following important property (note that \cite{TWX} assumed that $t_0$ is a center of $(T, \chi)$, but this was not used in the proof of the following lemma):

    \begin{lemma}[Abrishami, Alecu, Chudnovsky, Hajebi, Spirkl \cite{TWX}]
      \label{noncoop}
      Let $t_0 \in V(T)$. If $u, v \in \chi(t_0)$ are $d$-safe and not $t_0$-cooperative, then $u$ is adjacent to $v$.
    \end{lemma}

    A vertex $v \in G$ is $ab$-cooperative if there exists a
    component $D$ of $G \setminus N[v]$ such that $a,b \in N[D]$. We show:

    \begin{lemma}
      \label{abnoncoop}
      If $v \in G$  is  $d$-safe and not $ab$-cooperative, then
      $v$ is adjacent to both $a$ and $b$. In particular, the set of
      vertices that are not $ab$-cooperative is a clique.
    \end{lemma}

    \begin{proof}
      Suppose $v$ is non-adjacent to $a$ and $v$ is not $ab$-cooperative.
      Since $ab$ is a  $(4m+2)(m-1)$-banana and $(4m+2)(m-1)\geq k+d$,
      we can choose $P_1, \dots, P_{k+d} \in \mathcal{P}$ such that
      the sets $P_1^*, \dots, P_{k+d}^*$ are pairwise vertex-disjoint.
      Since $v$ is not $ab$-cooperative, it follows that $v$ has a
      neighbors in each of $P_1^*, \dots, P_{k+d}^*$, and
      so there exist pairwise vertex-disjoint paths
      $Q_1, \dots Q_{k+d}$, each  with ends $a,v$, and such that
      $Q_i^* \subseteq P_i^*$ for every $i \in \{1, \dots, k+d\}$.
      Since $v$ is $d$-safe, we may assume (by renumbering the paths
      $Q_1, \dots, Q_{k+d}$ if necessary), that
      $N(v) \cap \Hub(G)  \cap Q_i = \emptyset$ for $i \in \{1, \dots, k\}$.
      But now we get a contradiction applying Theorem~\ref{bananatohub}
      to the graph $H=Q_1 \cup \dots \cup Q_k$ and the $k$-banana
      $av$ in $H$. This proves the first statement of \ref{abnoncoop}.
      Since $G$ is $C_4$-free, the second statement follows.
      \end{proof}

     In view of Theorem \ref{basket}, it would be convenient if we could work with the graph
   $\hat \chi(t_1) \cup \hat \chi(t_2)$, but unfortunately this graph may not be in $\mathcal{C}_t$, or even $\mathcal{C}$.
    In \cite{TWX}, a tool was designed
  to  construct a safe alternative to torsos by adding a set of vertices for each adhesion, rather than turning the adhesion into a clique: 
   
  \begin{theorem} [Abrishami, Alecu, Chudnovsky, Hajebi, Spirkl \cite{TWX}]
    \label{connectors}
    Let $t_0 \in T$.
Assume that $G$ does not admit a clique cutset. For every $t \in N_T(t_0)$,
        there exists 
         $\Conn(t_0,t) \subseteq G_{t_0 \rightarrow t}$ such that
        \begin{enumerate}
        \item  $\adh(t, t_0) = \chi(t) \cap \chi(t_0)  \subseteq \Conn(t_0,t)$.
          \label{C-1}
                \item $\Conn(t_0,t) \setminus \chi(t_0)$ is connected and
                  $N(\Conn(t_0,t) \setminus \chi(t_0))=\chi(t_0) \cap \chi(t)$.
\label{C-2}
\item No vertex of $\Conn(t_0,t) \setminus \chi(t_0)$ is a hub in the graph
  $(G \setminus G_{t_0 \rightarrow t})  \cup \Conn(t_0,t)$.
\label{C-3}
        \end{enumerate}
  \end{theorem}

 Similarly to the construction in \cite{TWX}, we now define a graph that is a safe alternative to
  $\hat \chi(t_1) \cup \hat \chi(t_2)$. Let $S'$ be a stable set of hubs of $G$ with $S' \cap \{a,b\}=\emptyset$, and assume that every $s \in S'$ is $d$-safe. Let $S_{bad}$ denote the set of all vertices in $S'$ that are common neighbors of $a$ and $b$, or are not $t_1$-cooperative, or are not $t_2$-cooperative.
  Since $S'$ is stable, 
Lemma \ref{noncoop} implies that $|S_{bad}| \leq 3$.
  By Lemma \ref{abnoncoop}, every vertex in $S' \setminus S_{bad}$ is $ab$-cooperative.
  If $t_1 \neq t_2$, for  $i \in \{1,2\}$, let $t_i'$ be the neighbor of $t_i$ in the
  (unique) path in $T$ from $t_1$ to $t_2$.
  If $t_1=t_2$, set $t_1'=t_2'=t_1$.
  Let $S=S' \setminus S_{bad}$ and set  $$\beta(S')=\left(\chi(t_1) \cup \chi(t_2) \cup \left(\bigcup_{i\in \{1,2\}}\bigcup_{t \in N_T(t_i) \setminus \{t_i'\}}\Conn(t_i,t)\right)\right) \setminus S_{bad} .$$
  Write $\beta=\beta(S')$. We fix $S', S, S_{bad}$ and $\beta$ throughout this section. 
  It follows that for every $i \in \{1,2\}$ and for every 
  $t \in N_T(t_i) \setminus \{t_i'\}$, we have that $\beta \subseteq (G \setminus G_{t_i \rightarrow t}) \cup \Conn(t_i,t)$. 
  Let $X \subseteq \beta$.
For $i\in \{1,2\}$,   we define  $\delta_i(X)$
    to be the set of
    all vertices 
    $t \in N_T(t_i) \setminus \{t_i'\}$
    such that  $X \cap (G_{t_i \rightarrow t} \setminus (\chi(t_1) \cup \chi(t_2))) \neq \emptyset$;  let
    $\delta(X)=\delta_1(X) \cup \delta_2(X)$.
    Write
    $$\Delta(X)=\bigcup_{t \in \delta_1(X)}\adh(t_1,t)
    \cup \bigcup_{t \in \delta_2(X)}\adh(t_2,t).$$
    
  Next we summarize several key properties of $\beta$.  
  \begin{lemma}
    \label{smallC}
    Suppose that $G$ does not admit a clique cutset and let
    $s \in S \cap \Hub(\beta)$. Then the following hold. 
    \begin{enumerate}
      \item \label{smallC-1} $s \in \chi(t_1) \cup \chi(t_2)$.
\item \label{smallC-2}     For $i \in \{1,2\}$, we have $|N_{\hat{\chi}(t_i)}(s)| <  2m(m-1)$.
\item   \label{smallC-3} $|N_{\chi(t_1) \cup \chi(t_2)}(s)| < 4m(m-1)$.
\item \label{smallC-4} There is a component $B(s)$ of $\beta \setminus N[s]$
  such that $a,b \in N[B(s)]$.
\end{enumerate}
  \end{lemma}

  \begin{proof}
    Let $s \in S \cap \Hub(\beta)$.
    It follows from Theorem~\ref{connectors}\eqref{C-3} that
    $s \in \chi(t_1) \cup \chi(t_2)$.
    Since $s$ is $t_i$-cooperative for $i \in \{1,2\}$, we have that
    $| N_{\hat{\chi}(t_i)}(s)| < 2m(m-1)$, and the second assertion of the
    theorem holds.   
    Since $\chi(t_i)$ is a subgraph of $\hat{\chi}(t_i)$,
    the third assertion follows immediately from the second.

    We now prove the fourth assertion.
    Suppose there is no component $D$  of $\beta \setminus N[s]$ such that
    $a,b \in N[D]$.
        Write
    $$\hat \Delta(s)=
    N_{\chi(t_1) \cup \chi(t_2)}[s] \cup  
    \adh(t_1,t_1') \cup \adh(t_2,t_2') \cup \Delta(N_{\beta}[s]).$$
    Since $\hat \Delta(s) \subseteq N_{\hat{\chi}(t_1)}[s] \cup N_{\hat{\chi}(t_2)}[s] \cup \adh(t_1,t_1') \cup \adh(t_2,t_2')$,
    and $\adh(T,\chi) \leq m-1$, and by the second assertion of the theorem,
    we have that
    $|\hat \Delta(s)| < (4m+2)(m-1)$.

    \sta{$\hat \Delta(s) \cap Q^* \neq \emptyset$ for every path in $\beta$ with ends $a,b$. \label{Deltasep}}

    Suppose there is a path $Q$ in $\beta$ with ends $a,b$ such that
    $\hat \Delta(s) \cap Q^*=\emptyset$. Since there is no component
    $D$ of   $\beta \setminus N[s]$ such that
    $a,b \in N[D]$, it follows that $N_{\beta}(s) \cap Q^* \neq \emptyset$.
    Consequently, there is an $i \in \{1,2\}$ and $t \in N_T(t_i) \setminus \{t_i'\}$
    such that $N_\beta(s) \cap Q^* \cap \Conn(t_i,t) \neq \emptyset$.
    Since $Q$ is a path from $a$ to $b$ in $G$,
    it follows that $Q^* \cap (\chi(t_1) \cup \chi(t_2)) \neq \emptyset$.
    Since $Q^*$ is connected and $\adh(t_i,t)$
    separates $G_{t_i \rightarrow t} \setminus \chi(t_i)$ from
    $(\chi(t_1) \cup \chi(t_2)) \setminus \chi(t)$, we deduce that
    $Q^* \cap \adh(t_i,t) \neq \emptyset$.
    But since $s$ has a neighbor in $\Conn(t_i,t)$, it follows that 
    $\adh(t_i,t) \subseteq \hat \Delta(s)$,
    a contradiction. This proves \eqref{Deltasep}.

        \sta{$P^* \cap \hat \Delta(s) \neq \emptyset$ for every path $P$ of
      $G$    with ends $a,b$. \label{snonsep}}

        Let $P$ be a path in $G$ with ends $a,b$. We prove by induction
        on $|P \setminus \beta|$ that $\hat \Delta(s) \cap P^* \neq \emptyset$.
If $P \subseteq \beta$, the claim follows from \eqref{Deltasep}, and this is the base case.

Let $p \in P^* \setminus \beta$ and let $t_0 \in T$ such that $p \in \chi(t_0)$.
Suppose first that  $t_0$ belongs to the component $T'$ of $T \setminus \{t_1,t_2\}$  such that  $t_1' \in T'$. Then $t_2' \in T'$.  Since $P^*$ is connected and $P^* \cap (\chi(t_1) \cup \chi(t_2)) \neq \emptyset$, it follows that
$P^* \cap (\adh(t_1,t_1') \cup \adh(t_2,t_2')) \neq \emptyset$,
and so $P^* \cap \hat \Delta(s) \neq \emptyset$. 
Thus we may assume that $p \in G_{t_1 \rightarrow t} \setminus \chi(t_1)$ for some $t \in N_T(t_1) \setminus \{t_1'\}$.  Since $P^*$ is connected and $P^* \cap (\chi(t_1) \cup \chi(t_2)) \neq \emptyset$, it follows that
$P^* \cap \adh(t_1,t) \neq \emptyset$.  
Write $P=p_1 \dd \cdots \dd p_l$ where $p_1=a$ and $p_l=b$.
Let $i$ be minimum and $j$ maximum such that
$p_i \in \adh(t_1,t)$. Since $p \in G_{t_1 \rightarrow t} \setminus \chi(t_1)$,
it follows that $i<j-1$. By Theorem \ref{connectors}\eqref{C-2}  there is a
path $R$ form $p_i$ to $p_j$ with $R^* \subseteq \Conn(t_1,t)$.
Then $P'=p_1 \dd \cdots \dd p_i \dd R \dd p_j \dd  \dots \dd p_l$
is a path from $a$ to $b$ in $G$ with
$|P' \setminus \beta| < |P \setminus \beta|$.
Inductively, $\hat \Delta(s) \cap P'^* \neq \emptyset$.
But $\hat \Delta(s) \subseteq \chi(t_1) \cup \chi(t_2)$ and $P'^* \cap (\chi(t_1) \cup \chi(t_2)) \subseteq P^* \cap (\chi(t_1) \cup \chi(t_2))$,
and so $\hat \Delta(s) \cap P^* \neq \emptyset$.
This proves \eqref{snonsep}.
\\
\\
Since $ab$ is a $(4m+2)(m-1)$-banana in $G$, and since $\hat \Delta(s)<(4m+2)(m-1)$,
\eqref{snonsep} leads to a  contradiction. This completes the proof.

  \end{proof}
      
  Recall that a {\em separation} of  $\beta$ is a triple $(X,Y,Z)$ of pairwise disjoint
  subsets of $\beta$ with $X \cup Y \cup Z = \beta$ such that $X$ is anticomplete to
  $Z$. We are now ready to move on to star cutsets. We will work in the graph $\beta$ and take advantage of its special properties. At the end of the section
  we will explain how to connect our results back to the graph $G$ that
  we are interested in.
  
  As in other papers in the series, we associate a certain unique star separation to every  vertex of $S \cap \Hub(\beta)$. The separation here is
  chosen differently from the way it was done in the past, but the behavior
  we observe is  similar. The reason for the difference is
  that unlike in \cite{TWVIII} or \cite{TWI}, our here  goal is to disconnect two given
  vertices, rather than find a ``balanced separator'' in the graph
  (more on this in Section~\ref{sec:proof}).

  Let $v \in S \cap \Hub(\beta)$. By Theorem~\ref{smallC}, there is a component
  $D$ of $\beta \setminus N[v]$ such that $a,b \in N[D]$.
  Since $s \not \in S_{bad}$, it follows that
  $s$ is not complete to $\{a,b\}$; consequently $D \cap \{a,b\} \neq \emptyset$, and so  $D$ is unique.
  Let $B(v)$ be the   unique component of 
    $\beta \setminus N[v]$
  such that $a,b \in N[B(v)]$ (this choice of $B(v)$ is different from what
  we have done in earlier papers).
    Let $C(v)=N(B(v)) \cup \{v\}$, and  finally, let $A(v)=\beta \setminus (B(v) \cup C(v))$. Then $(A(v), C(v),B(v))$ is the 
    {\em canonical star separation of $\beta$ corresponding to $v$}.
    The next lemma is a key step in the proof of the fact that
    decomposing by canonical star separations
    makes the graph simpler.    We show:
\begin{lemma} \label{nohub}
  The vertex $v$ is not a hub of $\beta \setminus A(v)$.
    \end{lemma}
    
    \begin{proof}
      Suppose
      that $(H,v)$ is a proper wheel in $\beta \setminus A(v)$.
      Then $H \subseteq N[B(v)]$, contrary to Theorem \ref{wheelstarcutset}.
           This proves Lemma~\ref{nohub}.
    \end{proof}

 We need just a little more set up. 
Let $\mathcal{O}$ be a linear order on $S \cap \Hub(\beta)$.  Following \cite{wallpaper}, we say
     that two vertices of $S \cap \Hub(\beta)$ are {\em star twins} if $B(u)= B(v)$, $C(u) \setminus \{u\} = C(v) \setminus \{v\}$, and
    $A(u) \cup \{u\} = A(v) \cup \{v\}$. (Note that every two of these conditions imply the third.)

Let $\leq_A$ be a relation on $S \cap \Hub(\beta)$ defined as follows: 
\begin{equation*}
\hspace{2.3cm}
x \leq_A y \ \ \ \text{ if} \ \ \  
\begin{cases} x = y; \\ 
\text{$x$ and $y$ are star twins and $\mathcal{O}(x) < \mathcal{O}(y)$; or}\\ 
\text{$x$ and $y$ are not star twins and } y \in A(x).\\
\end{cases}
\end{equation*} 
Note that if $x \leq_A y$, then either $x = y$, or $y \in A(x).$

The conclusion of the next lemma is the same as of Lemma 6.2 from  \cite{TWIII}, but  the  assumptions are different.

  \begin{lemma} \label{shields}
If $y \in A(x)$, then  $A(y) \cup \{y\} \subseteq A(x) \cup \{x\}$.
\end{lemma}

  \begin{proof}
    Since $C(y)\subseteq N[y]$ and $y$ is anticomplete to $B(x)$, we have $B(x) \subseteq G \setminus N[y]$. Since $B(x)$ is connected, there is a
    component $D$ of 
    $G \setminus N[y]$ such that $B(x) \subseteq D$. 
    Since $x \not \in S_{bad}$, it follows that 
    $\{a,b\} \cap  B(x) \neq \emptyset$, and so $D=B(y)$.
        Let $v \in C(x)\setminus \{x\}$. Then $v$ has a neighbor in $B(x)$ and thus in $B(y)$. If $v \in N[y]$, then $v \in C(y)$. If $v \not \in N[y]$, then $v \in B(y)$. It follows that $C(x) \setminus \{x\} \subseteq C(y) \cup B(y)$.  But now  $A(y) \setminus \{x\} \subseteq A(x)$, as required. This
  proves Lemma~\ref{shields}.
  \end{proof}

  The proofs of the next two lemmas are identical to the proofs  of
  Lemmas 6.3 and Lemma 6.4 in   \cite{TWIII} (using Lemma~\ref{shields} instead of Lemma 6.2 of \cite{TWIII}), and we omit them.
  
\begin{lemma} \label{Aorder}
      $\leq_A$ is a partial order on $S \cup \Hub(\beta)$. 
\end{lemma}

In view of Lemma~\ref{Aorder}, let  $\Core(S')$ be the set of all $\leq_A$-minimal elements of $S \cap \Hub(\beta)$. 

  \begin{lemma}\label{looselylaminar}
    Let $u,v \in \Core(S')$. Then
    $A(u) \cap C(v)=C(u) \cap A(v)=\emptyset$.
  \end{lemma}

    We have finally reached our goal: we can define a subgraph of $G$
  that is simpler than $G$ itself, but carries all the 
  information we need.
 Define
 $$\beta^A(S')=\bigcap_{v \in \Core(S')} (B(v) \cup C(v)).$$
 The next theorem summarizes the important aspects of the  behavior of
 $\beta^A(S')$.

\begin{theorem} 
  \label{A_centralbag}
    The following hold:
    \begin{enumerate}
    \item For every $v \in \Core(S')$, we have $C(v) \subseteq \beta^{A}(S')$. \label{A-1}
\item  For every $v \in \Core(S')$, $|C(v) \cap (\chi(t_1) \cup \chi(t_2))| \leq 4m(m-1)$.
  \label{A-2}
\item  For every $v \in \Core(S')$,  $|\Delta(C(v))| \leq 4m(m-1)$.
    \label{A-2.5}
      \item For every component $D$ of $\beta  \setminus \beta^{A}(S')$, there exists $v \in \Core(S')$ such that $D \subseteq A(v)$. Further, if $D$ is a component of $\beta \setminus \beta^A(S')$ and $v \in \Core(S')$ such that $D \subseteq A(v)$, then $N_{\beta}(D) \subseteq C(v)$. \label{A-3}
      \item $S' \cap \Hub(\beta^A(S'))=\emptyset$. \label{A-4}
    \end{enumerate}
\end{theorem}

\begin{proof}
    \eqref{A-1} is immediate
    from Lemma  \ref{looselylaminar}, and \eqref{A-2} follows 
    from Lemma \ref{smallC}.

    Next we prove \eqref{A-2.5}. By Lemma \ref{smallC}\eqref{smallC-1}, it follows that $v \in \chi(t_1) \cup \chi(t_2)$. Observe that if $t \in T \setminus \{t_1,t_2,t_1',t_2'\}$ and
    $(G_{t_i \rightarrow t} \setminus (\chi(t_1) \cup \chi(t_2))  \cap C(v) \neq \emptyset$ for some $i \in \{1, 2\}$,
    then $v \in \chi(t_1) \cap \chi(t)$ or $v \in \chi(t_2) \cap \chi(t)$.
    It follows that for $i \in \{1,2\}$, $\adh(t_i,t) \subseteq \Delta(C(v))$
    only if $v \in \adh(t_i,t)$. 
    Consequently, 
    $|\Delta(C(v))| \leq \deg_{\hat\chi(t_1)}(v)+\deg_{\hat\chi(t_2)}(v)$,
    and so \eqref{A-2.5}  follows 
    from Lemma \ref{smallC}\eqref{smallC-3}.    
    
Next we prove  \eqref{A-3}. Let $D$ be a component of
$\beta \setminus \beta^A(S')$. Since  $\beta \setminus \beta^A(S')=\bigcup_{v \in \Core(S')}A(v)$, there exists
    $v \in \Core(S')$ such that $D \cap A(v) \neq \emptyset$.
    If $D \setminus A(v) \neq \emptyset$, then, since $D$ is connected, it follows that $D \cap N(A(v)) \neq \emptyset$; but then $D \cap C(v) \neq \emptyset$, contrary to \eqref{A-1}. Since $N_{\beta}(D) \subseteq \beta^A(S')$ and $N_{\beta}(D) \subseteq A(v) \cup C(v)$, it follows that $N_{\beta} (D) \subseteq C(v)$. 
    This proves \eqref{A-3}.

    To prove \eqref{A-4}, let $u \in S' \cap \Hub(\beta^A(S'))$.
    Since $\beta^A(S') \subseteq \beta$, we deduce that $u \not \in S_{bad}$,
    and so   $u \in S \cap \Hub(\beta)$.
    By Theorem \ref{smallC}\eqref{smallC-1}, we have that $u \in \chi(t_1) \cup \chi(t_2)$.
       By  Lemma \ref{nohub}, it follows that
$\beta^A(S') \not \subseteq B(u) \cup C(u)$, and therefore $u \not \in \Core(S')$.
But then $u \in A(v)$ for some $v \in \Core(S')$, and so $u \not \in \beta^A(S')$,
a contradiction. This proves \eqref{A-4} and completes the proof of
 Theorem \ref{A_centralbag}.
  \end{proof}

We now explain how $\beta^A(S')$ is used.
In the course of the proof of Theorem \ref{banana}, we will inductively obtain
a small  cutset separating $a$ from $b$ in
$\beta^A(S')$. The next theorem allows us  to transform this cutset
into a cutset separating $a$ from $b$ in $\beta$.

\begin{theorem}
  \label{smallsepinbeta}
Let $(X,Y,Z)$ be a separation of $\beta^A(S')$ such that $a \in X$ and
$b \in Z$. Let $Y'=(Y \cup \bigcup_{s \in Y \cap \Core(S')} C(s)) \setminus \{a,b\}$.
Then
\begin{enumerate}
  \item $Y'$
    separates $a$ from $b$ in $\beta$.
    \item $|Y' \cap (\chi(t_1) \cup \chi(t_2))| \leq |Y|+4m(m-1)|Y \cap \Core(S')|$.
    \item $|\Delta(Y') \setminus \Delta(Y)| \leq  4m(m-1)|Y \cap \Core(S')|$.
\end{enumerate}
\end{theorem}

\begin{figure}[t!]
    \centering
\includegraphics[scale=0.7]{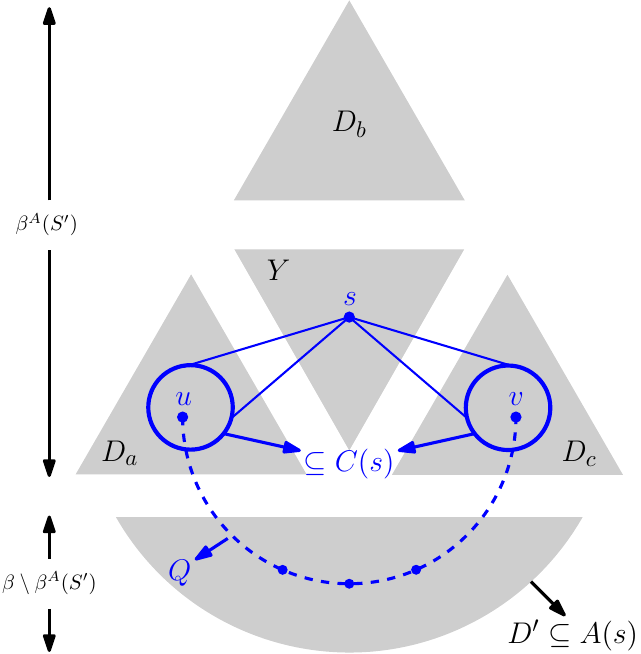}
    \caption{Proof of Theorem \ref{smallsepinbeta}.}
    \label{fig:411}
\end{figure}
  
\begin{proof}
  Suppose that $Y'$ does not separate $a$ from $b$ in $\beta$. Let $P$ be a path from $a$ to $b$ in $\beta \setminus Y'$. Let $D_a,D_b$ be the components of $\beta^A(S') \setminus Y$ such that $a \in D_a$ and $b \in D_b$. Since the first vertex of $P$ is in $D_a$ and the last vertex of $P$ is in $D_b$, it follows that there is a subpath $Q$ of $P$ such that $Q$ has ends $u, v$ and: 
  \begin{itemize}
      \item The vertex $u$ is in $D_a$; 
      \item $Q^* \cap \beta^A(S') = \emptyset$; 
      \item The vertex $v$ is in $\beta^A(S') \setminus D_a$. 
  \end{itemize}
  See Figure \ref{fig:411}.  Since $Y \subseteq Y'$, it follows that $v$ is in a component $D_c$ of $\beta^A(S') \setminus Y$ with $D_c \neq D_a$ (possibly $D_c = D_b$). This implies that $uv \not\in E(G)$, and so $Q^* \neq \emptyset$. 
  
    By Theorem~\ref{A_centralbag}\eqref{A-3}, there is an $s \in \Core(S') \subseteq \beta^A(S')$ such that $Q^* \subseteq A(s)$ and $u, v \in C(s)$.   Since $D_a,D_c$ are distinct components of $\beta^A(S') \setminus Y$,   and since $N_{\beta^A(S')}(s)$ meets both $D_a$ and $D_b$,
  it follows that $s \in Y$. Therefore $C(s) \setminus \{a,b\} \subseteq Y'$, and so $u, v \in \{a, b\}$. Since $u \neq v$, it follows that $u = a$ and $b = v$, and so $a,b \in N(s)$,
  contrary to the fact that $s \not \in S_{bad}$.
  This proves the first assertion of the theorem.

  The second assertion follows  from  Theorem~\ref{A_centralbag}\eqref{A-2}, and the third assertion follows from  Theorem~\ref{A_centralbag}\eqref{A-2.5}.
  \end{proof}

We finish this section with a theorem that allows us to transform
a cutset separating $a$ from $b$ in $\beta$ into
a cutset separating $a$ from $b$ in $G$.

\begin{theorem}
  \label{smallseping}
Assume that $G$ does not admit a clique cutset.
  Let $(X,Y,Z)$ be a separation of $\beta$ such that $a \in X$ and
  $b \in Z$.
  Let
$$Y'=(S_{bad} \cup \adh(t_1,t_1') \cup \adh(t_2,t_2') \cup (Y\cap (\chi(t_1) \cup \chi(t_2)))  \cup \Delta(Y)) \setminus \{a,b\}.$$
Then $Y'$ separates $a$ from $b$ in $G$.
\end{theorem}

\begin{proof}
  Suppose not. 
  Let $D_a,D_b$ be the components of $\beta \setminus Y$ such that $a \in D_a$
  and $b \in D_b$. Since $Y'$ does not separate $a$ from $b$ in $G$, and
  $Y' \cap \{a,b\}=\emptyset$, 
  there is a component $D$ of $G \setminus Y'$ such that
  $a,b \in D$. Let $P$ be a path from $a$ to $b$ with $P \subseteq D$.
  
  As before, it follows that there is a subpath $Q$ of $P$ such that $Q$ has ends $u, v$ and: 
  \begin{itemize}
      \item The vertex $u$ is in $D_a$; 
      \item $Q^* \cap \beta= \emptyset$; 
      \item The vertex $v$ is in $\beta \setminus D_a$. 
  \end{itemize}
    We first consider the case that $v \in Y$ (and so $v \in P^*$). Then $v \in \Conn(t_i, t) \setminus (\chi(t_1) \cup \chi(t_2))$ for some $i \in \{1, 2\}$ and $t \in N_T(t_i) \setminus \{t_i'\}$. It follows that $t \in \delta_i(Y)$, and so $\adh(t_i, t) \setminus \{a, b\} \subseteq \Delta(Y) \setminus \{a, b\} \subseteq Y'$.  It follows that $P^* \cap \adh(t_i, t) = \emptyset$, and from Theorem \ref{basket}, we know that $P^* \cap (\chi(t_1) \cup \chi(t_2)) \neq \emptyset$. Therefore, $P^* \cap (G_{t_i \rightarrow t} \setminus (\chi(t_1) \cup \chi(t_2))) = \emptyset$, contrary to the fact that $v \in \Conn(t_i, t) \setminus (\chi(t_1) \cup \chi(t_2))$. This is a contradiction, and proves that $v \not\in Y$.

  Therefore, there is a component $D_c \neq D_a$ (possibly $D_b = D_c$) of $\beta \setminus Y$ such that $v \in D_c$. Since there are no edges between $D_a$ and $D_c$, it follows that $Q^* \neq \emptyset$.

Consequently,  there is a component $\tilde{T}$ of $T \setminus \{t_1,t_2\}$
such that $Q^* \subseteq (\bigcup_{t \in \tilde{T}} \chi(t)) \setminus (\chi(t_1) \cup \chi(t_2))$. Let $\tilde{D}$ be the component of
$G \setminus (\chi(t_1) \cup \chi(t_2))$ such that $Q^* \subseteq \tilde{D}$.
By Theorem~\ref{basket}, it follows that $|N(\tilde{D}) \cap \{a,b\}| \leq 1$.

  Suppose first that $t_1' \in \tilde{T}$.
  Then $t_1' \neq t_2$, and $t_2' \neq t_1$, and $t_2' \in \tilde{T}$. 
  Since $N_G(\tilde{D}) \subseteq \adh(t_1,t_1') \cup \adh(t_2,t_2')$, it follows that $N_{G \setminus Y'}(\tilde{D}) \subseteq \{a, b\}$, and so $|N_{G \setminus Y'}(\tilde{D})| \leq 1$. Since $P$ is a path with ends $a, b \not\in \tilde{D}$ and with $P \subseteq G \setminus Y'$, it follows that $P \cap \tilde{D}= \emptyset$, and so $Q \cap \tilde{D} = \emptyset$. Similarly, $t_2' \not \in \tilde{T}$.

  We deduce that exactly one of $t_1,t_2$ has a neighbor in
  $\tilde{T}$ (unless $t_1 = t_2$), and this neighbor is unique; denote it by $t$.    By symmetry, we may assume that $tt_1 \in E(T)$, and so $\tilde{T} = T_{t_1 \rightarrow t}$. 
    By Theorem~\ref{connectedbranches}, we deduce that
    $\tilde{D}= G_{t_1 \rightarrow t} \setminus \chi(t_1)$.
    Since $u \in D_a \cap N(Q^*) \subseteq \beta \cap N[\tilde{D}]$, it follows that $u \in \Conn(t_1,t)$. Similarly, we have $v \in \Conn(t_1, t)$.     Since $\Conn(t_1,t) \setminus \chi(t_1)$ is connected 
and  $N(\Conn(t_1,t) \setminus \chi(t_1))=\chi(t_1) \cap \chi(t)$ by Theorem \ref{connectors}\eqref{C-2}, it follows that $\Conn(t_1, t)$ contains a path $Q'$ with ends $u,v$ and interior in $\Conn(t_1, t) \setminus \chi(t_1)$. Since $Q'$ is a path from $D_a$ to $D_c$ in $\beta$, it follows that $Y \cap Q'^* \neq \emptyset$, and so  $Y \cap (\Conn(t_1,t) \setminus \chi(t_1)) \neq \emptyset$. Therefore, $Y \cap (G_{t_1 \rightarrow t} \setminus \chi(t_1)) \neq \emptyset$, which implies that  $\adh(t_1,t) \setminus \{a,b\} \subseteq Y'$.  It follows that $P^* \cap \adh(t_1, t) = \emptyset$, and from Theorem \ref{basket}, we know that $P^* \cap (\chi(t_1) \cup \chi(t_2)) \neq \emptyset$. Therefore, $P^* \cap (G_{t_1 \rightarrow t} \setminus \chi(t_1) ) = \emptyset$, contrary to the fact that $\emptyset \neq Q^* \subseteq G_{t_1 \rightarrow t} \setminus \chi(t_1)$. This is a contradiction, and concludes the proof. 
  \end{proof}

\section{Connectifiers} \label{sec:connectifier}

We start this section by  describing  minimal connected subgraphs
containing the neighbors of a large number of vertices from a given
stable set. We then use this result to deduce what a pair
of two such subgraphs can look like (for the same stable set) assuming that
they are anticomplete
to each other and the graph in question is in $\mathcal{C}$.
This generalizes results from \cite{TWVIII} and \cite{TWX}.

What follows is mostly  terminology from  \cite{TWX}, but there
are also some new notions.
Let $G$ be a graph, let $P=p_1 \dd \cdots \dd p_n$ be a path in $G$ and let
$X=\{x_1, \dots, x_k\}  \subseteq G \setminus P$. We say that $(P,X)$ is an
{\em alignment} if 
$N_P(x_1)=\{p_1\}$, $N_P(x_k)=\{p_n\}$, 
every vertex of $X$  has a neighbor in $P$, and there exist
$1< j_2 < \dots< j_{k-1} < j_k =  n$ such that
$N_P(x_i) \subseteq p_{j_i} \dd P \dd p_{j_{i+1}-1}$ for $i \in \{2, \dots, k-1\}$.
We also say that $x_1, \dots, x_k$ is {\em the order on $X$ given by the
  alignment $(P,X)$}.
An alignment $(P,X)$ is {\em wide}  if 
each of $x_2, \dots, x_{k-1}$ has at least two non-adjacent neighbors in  $P$,
{\em spiky} if  each of $x_2, \dots, x_{k-1}$ has a unique neighbor in $P$ and
{\em triangular} if 
each of $x_2, \dots, x_{k-1}$ has exactly two neighbors in $P$  and they are
adjacent. An alignment is {\em consistent} if it is
wide, spiky or triangular.

By a {\em caterpillar} we mean a tree $C$ with
  maximum degree three such that  there exists  a path $P$ of $C$ such that all vertices of degree $3$ in $C$ belong to $P$.
We call a minimal such path $P$ the {\em spine} of $C$. 
   By a \textit{subdivided star} we mean a graph isomorphic to
  a subdivision of the complete bipartite graph $K_{1,\delta}$ for some $\delta\geq 3$.
  In other words, a subdivided star is a tree with exactly one vertex of degree at least three,
  which we call its \textit{root}. For a graph $H$, a vertex $v$ of $H$ is
  said to be \textit{simplicial} if $N_H(v)$ is a clique. We denote by
  $\mathcal{Z}(H)$ the set of all simplicial vertices of $H$. Note that for
  every tree $T$, $\mathcal{Z}(T)$ is the set of all leaves of $T$. An edge
  $e$ of a tree $T$ is said to be a \textit{leaf-edge} of $T$ if $e$ is
  incident with a leaf of $T$. It follows that if $H$ is the line graph of a
  tree $T$, then $\mathcal{Z}(H)$ is the set of all vertices in $H$
  corresponding to leaf-edges of $T$.

 Let $H$ be a graph that is either a path, or a caterpillar, or the line graph of a caterpillar,  or a subdivided star with root $r$,  or the line graph of
  a subdivided star with root $r$.
We define an induced subgraph of $H$, denoted by $P(H)$, which we will use
throughout the paper.
If $H$ is a path, we let $P(H)=H$.
If $H$ is a caterpillar, we let $P(H)$ be the spine of $H$.
  If $H$ is the line graph of a caterpillar $C$, let $P(H)$ be the path in $H$
  consisting of the vertices of $H$ that correspond to the edges of the spine of $C$. If $H$ is a subdivided star  with root $r$, let $P(H)=\{r\}$.
  It $H$ is the line graph of a subdivided star $S$  with root $r$, let
  $P(H)$ be the clique of $H$ consisting of the vertices of $H$ that
  correspond to the edges of $S$ incident with $r$.
  The {\em legs} of $H$ are the components of $H \setminus P(H)$.

  Next, let $H$ be a caterpillar or the line graph of a caterpillar 
and let $S$ be a set of vertices disjoint from $H$ such that
every vertex of $S$ has a unique neighbor in $H$ and  
$H\cap N(S)=\mathcal{Z}(H)$.
Let $X$ be the set of vertices of $H \setminus P(H)$ that have neighbors in
$P(H)$. Then the neighbors of elements of $X$ appear in $P(H)$ in order (there may be ties at the ends of $P(H)$, which we break arbitrarily); 
let $x_1, \dots, x_k$ be the corresponding order
on $X$. Now, order the vertices of $S$ as
$s_1, \dots, s_k$ where $s_i$ has a neighbor in the leg of $H$ containing
$x_i$ for $i \in \{1, \dots, k\}$.
We say that $s_1, \dots, s_k$ is {\em the order on $S$ given by $(H,S)$}.

Next, let  $H$ be   an induced subgraph of $G$ that is  a caterpillar, or the line graph of a caterpillar, or a subdivided star or the line graph of a subdivided star and let $X \subseteq G \setminus H$ be such that every vertex of $X$ has a unique neighbor in $H$ and  $H \cap N(X)=\mathcal{Z}(H)$. We say that $(H,X)$ is a {\em consistent connectifier} and it is  {\em spiky, triangular,  stellar, or clique} respectively. If $H$ is a single vertex and $X \subseteq N(H)$,
we also call $(H,X)$ a {\em stellar connectifer}.
If $H$ is a subdivided star, a singleton  or the line graph of a subdivided star, we say that $(H,X)$ is a {\em concentrated} connectifier. 

The following was  proved in \cite{TWX}:

\begin{theorem}[Abrishami, Alecu, Chudnovsky, Hajebi, Spirkl \cite{TWX}]
\label{connectifier2}
  For every integer $h\geq 1$, there exists an integer $\nu=\nu(h)\geq 1$ with the following property. Let $G$ be a connected graph with no clique of cardinality $h$. Let $S \subseteq G$ such that  $|S|\geq \nu$, $G \setminus S$ is connected  and every vertex of $S$ has a neighbor in $G \setminus S$.
  Then  there is a set $\tilde S \subseteq S$ with $|\tilde S|=h$ and an induced subgraph $H$ of $G \setminus S$ for which one of the following holds.
  \begin{itemize}
    \item $H$ is a path and every vertex of $\tilde S$ has a neighbor in $H$; or    \item $H$ is a caterpillar, or the line graph of a caterpillar, or a subdivided star. Moreover, every vertex of $\tilde S$ has a unique neighbor in $H$ and  
      $H\cap N(\tilde S)=\mathcal{Z}(H)$.
          \end{itemize}
\end{theorem}

We now  prove a version of Theorem~\ref{connectifier2} that does not
assume a bound on the clique number. This result may be of independent use in the future (See Figure~\ref{fig:connectifier} for a depiction of the outcomes).
\begin{figure}
    \centering
    \includegraphics[scale=0.6]{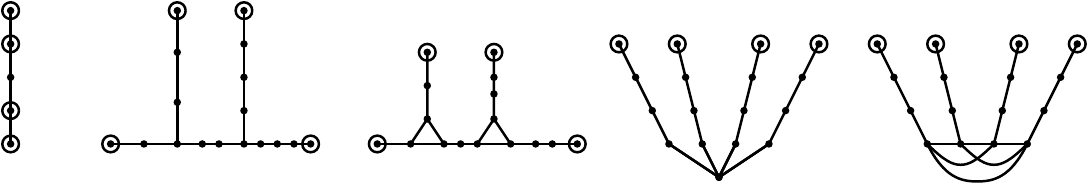}
    \caption{Outcomes of Theorem~\ref{connectifier2general}. Circled nodes are the vertices in $H\cap N(Y')$.}
    \label{fig:connectifier}
\end{figure}

\begin{theorem}
\label{connectifier2general}
For every integer $h \geq 1$, there exists an integer $\mu=\mu(h)\geq 1$ with the following property. Let $G$ be a connected graph.
Let $Y \subseteq G$ such that  $|Y|\geq \mu$, $G \setminus Y$ is connected  and every vertex of $Y$ has a neighbor in $G \setminus Y$.
  Then  there is a set $Y' \subseteq Y$ with $|Y'|=h$ and an induced subgraph $H$ of $G \setminus Y$ for which one of the following holds.
  \begin{itemize}
  \item $H$ is a path and every vertex of $Y'$ has a neighbor in $H$; or
  \item $H$ is a caterpillar, or the line graph of a caterpillar, or a subdivided star or the line graph of a subdivided star. Moreover, every vertex of
    $Y'$ has a unique neighbor in $H$ and  
      $H \cap N(Y')=\mathcal{Z}(H)$.
          \end{itemize}
\end{theorem}
  
\begin{proof}
  Let $\mu(h)=\nu(h+1)$ where $\nu$ is as in Theorem~\ref{connectifier2}.
  Let $F$ be a minimal connected subset of $G \setminus Y$ such that every $y \in Y$ has a neighbor in $F$.
  
\sta{If $F$ contains a clique of size $h$, then
  there exists $H \subseteq F$
    and $Y' \subseteq Y$  with $|Y'|=h$ such that $(H,Y')$ is a clique connectifier.
    \label{boundclique}}

  Suppose  that there is a clique $K$ of size $h$ in $F$. It follows from the minimality of $F$ that for every $k \in K$ one of the following holds:
  \begin{itemize}
  \item There is $y(k) \in Y$ such that $y(k)$ is anticomplete to $F \setminus k$; in
    this case set $C(k,y(k))=\emptyset$.
  \item $F \setminus \{k\}$ is not connected, and for every component
    $C$ of $F  \setminus k$ there is $y \in Y$
    such that 
    $y$ is anticomplete to $F  \setminus (C \cup \{k\})$.
    Since $K$ is a clique, there is a component $C$ of  $F \setminus k$
    such that $K \cap C=\emptyset$. Let $y(k)$ be a vertex of $Y$
    that is anticomplete to $F \setminus (C \cup \{k\})$;
    write $C(k,y(k))=C$.
  \end{itemize}
  Let $Y'$ be the set of all vertices $y(k)$ as above.
  For every $k \in K$, let  $H_k$ be a path from $k$ to $y(k)$ with $H_k^* \subseteq C(k,y(k))$. Write $H=\bigcup_{k \in K} H_k$. We will show that $|Y'|=h$ and
  $(H,Y')$ is a clique connectifier. It follows from the definition of $H_k$
  that every vertex of $Y'$ has a neighbor in $H$.

  Next we claim that if  $k \neq k'$, then $H_{k}$ is disjoint from and anticomplete to $H_{k'}$. 
    Recall that $H_k \subseteq C(k,y(k)) \cup \{k\}$,
  and $y(k)$ is anticomplete to $F  \setminus (C(k,y(k)) \cup \{k\})$.
  It follows that $H_k \cap K = \{k\}$, and so $k \not \in H_{k'}$ and $k' \not \in H_k$.
  Let $D$ be the component of $F  \setminus k$ such that $k' \in D$.
  By the definition of $C(k,y(k))$, we have that $D \neq C(k,y(k))$.
  Since $H_{k'}$ is connected and $k \not \in H_{k'}$, it follows that
  $H_{k'} \subseteq D$. Consequently, $H_k$ and $H_{k'}$ are disjoint and anticomplete to each other. Since $y(k)$ has no neighbor in $D$, we deduce that
  $y(k)$ is anticomplete to $H_{k'}$, and in particular $y(k) \neq y(k')$.
  Similarly, $y(k')$ is anticomplete to $H_k$. This  proves the claim.

  Now it follows from the claim that if $k \neq k'$, then $y(k) \neq y(k')$.
  Consequently, $|Y'|=|K|=h$, and $(H,Y')$ is a clique connectifier.
  This proves \eqref{boundclique}.
  \\
  \\
  By \eqref{boundclique}, we may assume that
$F$ is $K_{h}$-free. Since $S$ is a stable set,
it follows that $F \cup S$ is $K_{h+1}$-free.
Now the result follows from  Theorem~\ref{connectifier2}
applied to $F \cup S$.
\end{proof}

Now we move to two anticomplete connected subgraphs and prove the following:

\begin{theorem} 
\label{twosidesgeneral}
  For every integer $x\geq 1$, there exists an integer $\phi=\phi(x) \geq 1$ with the following property.
    Let $G \in \mathcal{C}$ and assume that $V(G)=D_1 \cup D_2 \cup Y$
  where
  \begin{itemize}
  \item         $Y$ is a stable set with $|Y| = \phi$,
  \item $D_1$ and $D_2$ are components of $G \setminus Y$, and 
  \item $N(D_1)=N(D_2)=Y$.
  \end{itemize}
  Then
     there exist $X \subseteq Y$ with $|X|=x$,
$H_1 \subseteq D_1$ and $H_2 \subseteq D_2$
   (possibly with the roles of $D_1$ and $D_2$ reversed) such that either: 
     \begin{enumerate}
       \item Not both $(H_1,X)$ and $(H_2,X)$ are alignments, and
     \begin{itemize}
    \item   $(H_1,X)$ is  a triangular connectifier or a clique connectifier or a triangular alignment; and 
  \item   $(H_2,X)$ is a stellar connectifier, or a spiky connectifier, or a
    spiky alignment or a wide alignment.
  \end{itemize}
or
   \item  Both  $(H_1,X)$ and $(H_2,X)$ are alignments, and
 at least one of  $(H_1,X)$ and $(H_2,X)$  is not a spiky alignment.
\end{enumerate}
Moreover, if neither of $(H_1,X),(H_2,X)$ is a concentrated connectifier, then
  the orders given on $X$ by $(H_1,X)$ and by $(H_2,X)$ are the same.
  \end{theorem}

In this paper we do not need the full generality of
Theorem~\ref{twosidesgeneral}; we are only interested in two special cases:
when $D_1$ is a path and when the clique number of $G$ is bounded.
However, it is easier to prove the more general result first using the symmetry between $D_1$ and $D_2$, and then use it to
handle the special cases. We also believe that Theorem~\ref{twosidesgeneral} will be useful in the future.

 We start by recalling a well known theorem of Erd\H{o}s and Szekeres \cite{ES}.
  \begin{theorem}[Erd\H{o}s and Szekeres \cite{ES}]
    \label{ESz}
Let $x_1, \dots, x_{n^2+1}$ be a sequence of distinct reals. Then
there exists either an increasing or a decreasing $(n + 1)$-sub-sequence.
\end{theorem}

We start with  two lemmas.

\begin{lemma}
  \label{twosides_lemma}
    Let $G \in \mathcal{C}$ and assume that $V(G)=H_1 \cup H_2 \cup X$
    where $X$ is a stable set with $|X| \geq 3$ and $H_1$ is anticomplete to
    $H_2$.
        Suppose that  for $i \in \{1,2\}$, the pair $(H_i,X)$ is a consistent alignment, or a consistent connectifier.
    Assume also that  if neither of $(H_1,X),(H_2,X)$ is concentrated,
    then
  the orders given on $X$ by $(H_1,X)$ and by $(H_2,X)$ are the same.
  Then (possibly switching the roles of $H_1$ and $H_2$), we have that either:
    \begin{enumerate}
       \item Not both $(H_1,X)$ and $(H_2,X)$ are alignments, and
     \begin{itemize}
    \item   $(H_1,X)$ is  a triangular connectifier or a clique connectifier or a triangular alignment; and 
  \item   $(H_2,X)$ is a stellar connectifier, or a spiky connectifier, or a
    spiky alignment or a wide alignment.
  \end{itemize}
or
   \item  Both  $(H_1,X)$ and $(H_2,X)$ are alignments, and
 at least one of  $(H_1,X)$ and $(H_2,X)$  is not a spiky alignment, and at least one of $(H_1,X)$ and $(H_2,X)$ is not a triangular alignment.
    \end{enumerate}
        \end{lemma}

\begin{proof}
  If at least one $(H_i,X) \subseteq \{(H_1,X), (H_2,X)\}$
  is not a concentrated connectifier, we
  let $x_1, \dots, x_k$ be the order given on $X$ by $(H_i,X)$.
  If both of $(H_i,X)$ are  concentrated  connectifiers, we let $x_1, \dots, x_k$ be an arbitrary order on $X$.
  Let $H$ be the unique hole contained in $H_1 \cup H_2 \cup \{x_1,x_k\}$.
  For $j \in \{1,2\}$ and $i \in \{1, \dots, k\}$, if
  $H_j$ is a connectifier, let
  $D_i^j$ be the leg of $H_j$ containing a neighbor of $x_i$; and
  if $H_j$ is an alignment let $D_i^j=\emptyset$.

  Suppose first that  $(H_1,X)$ is a triangular alignment, a clique connectifier  or a triangular connectifier.
  If $(H_2,X)$ is a triangular alignment, a clique connectifier  or a triangular connectifier, then 
   for every $i \in \{2, \dots, k-1\}$, the graph 
  $H \cup D_i^1 \cup D_i^2 \cup \{x_i\}$ is either a prism or an even wheel with center $x_i$, a contradiction.
  This proves that $(H_2,X)$ is a stellar connectifier, or  a spiky connectifier, or 
  a spiky alignment, or a wide alignment, as required.

  Thus we may assume that for $i \in \{1,2\}$, the pair 
  $(H_i, X)$ is  a stellar connectifier, or  a spiky connectifier,  or 
  a spiky alignment, or a wide alignment.
  If  $(H_1,X)$ is a stellar connectifier or a spiky connectifier,
  then for every
  $x_i \in X \setminus \{x_1,x_k\}$, the graph
$H \cup D_i^1 \cup D_i^2 \cup \{x_i\}$ contains a theta, a contradiction.

It follows that for $i \in \{1,2\}$, the pair 
$(H_i, X)$ is an alignment.
We may assume that $(H_1,X)$ and $(H_2,X)$ are either both spiky alignments or both triangular alignments,
for otherwise the second outcome of the theorem holds.
But now for every
  $x_i \in X \setminus \{x_1,x_k\}$, the graph
$H \cup \{x_i\}$ is either a theta or an even wheel,
a contradiction.
\end{proof}

\begin{figure}[t!]
    \centering
\includegraphics[scale=0.7]{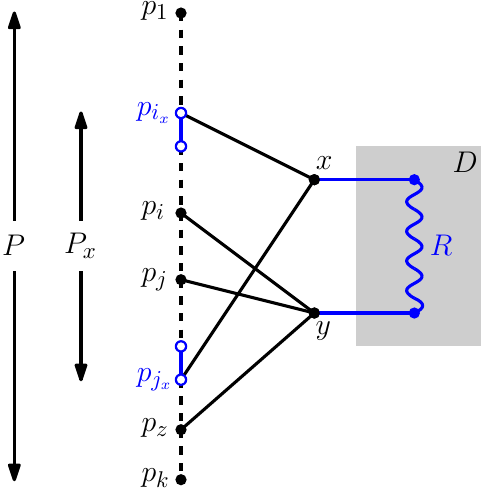}
    \caption{Proof of Lemma \ref{pathnotinterleave}.}
    \label{fig:56}
\end{figure}

\begin{lemma}
  \label{pathnotinterleave}
  Let $G$ be a theta-free graph and let
  $P=p_1 \dd \cdots \dd p_k$ be a path in $G$. Let $D$ be a connected subset of $G$ such that
  $D$ is anticomplete to $P$. Let $x,y \in N(P) \cap N(D)$
  such that $xy$ is not an edge. Assume that each of $x$ and $y$ has at least two non-adjacent neighbors in $P$. 
  Let $i_x$ be minimum and $j_x$ be maximum such that $x$ is adjacent to $p_{i_x}$ and $p_{j_x}$ and write $P_x=p_{i_x} \dd P \dd p_{j_x}$.
Suppose that $y$ has a neighbor in $P_x$.
  Then
  $y$ has a neighbor in $N_{P_x}[p_{i_x}] \cup N_{P_x}[p_{j_x}]$.
\end{lemma}

\begin{proof}
  Suppose that $y$ is anticomplete to $N_P[p_{i_x}] \cup N_P[p_{j_x}]$. 
Let $R$ be a path from $x$ to $y$ with $R^* \subseteq D$. (See Figure \ref{fig:56}.)
  Let $i \in \{{i_x}, \dots ,{j_x}\}$ be minimum and
  $j \in \{{i_x}, \dots, {j_x}\}$ be maximum such that $y$ is adjacent
  to $p_i, p_j$. Then there there is a path $P_i$ from $x$ to $y$
    with interior in $p_{i_x} \dd P \dd p_i$, and a path $P_j$
  from $x$ to $y$ with interior in $p_j \dd P \dd p_{j_x}$.
  If $j>i+1$, then $P_i \cup P_j \cup R$ is a theta with ends $x,y$,
  a contradiction; therefore $j \leq i+1$. Since $y$ has at least two non-adjacent neighbors in $P$, 
  it follows that $y$ has a neighbor $p_z$ in $P \setminus P_x$. We may assume
  that $z>j_{x}$, and therefore there is a path $Q$ from $x$ to $y$
  with $Q^* \subseteq p_{j_x} \dd P \dd p_z$.
  Since  $y$ is anticomplete to $N_{P_x}[p_{i_x}] \cup N_{P_x}[p_{j_x}]$, we have that
  $j<j_x -1$, and so $P_i^*$ is anticomplete to $Q^*$. But now
  $P_i \cup Q \cup R$ is a theta with ends $x,y$, a contradiction.
  \end{proof}

  We can now  prove Theorem~\ref{twosidesgeneral}. The proof follows along the lines of \cite{TWX}, but the assumptions are different.

  \begin{proof}
     Let $z=12^236^2 x^6$ and let $\phi(x)=\mu(\mu(z))$, where $\mu$ is as in Theorem \ref{connectifier2general}.
  Applying Theorem \ref{connectifier2general} twice, we obtain a set $Z \subseteq Y$ with $|Z|=z$
  and $H_i \subseteq D_i$ such that either 
   \begin{itemize}
    \item $H_i$ is a path and every vertex of $Z$ has a neighbor in $H_i$; or
    \item $(H_i,X)$ is  a consistent connectifier
        \end{itemize}
        for every $i \in \{1, 2\}$.

           \sta{\label{align} Let $i \in \{1, 2\}$ and $y \in \mathbb N$. If $H_i$ is a path and every vertex of $Z$ has a neighbor in $H_i$, then either some vertex of $H_i$ has $y$ neighbors in $Z$, or
     there exists $Z' \subseteq Z$ with $|Z'| \geq \frac{|Z|} {12y}$ and a  subpath $H_i'$ of $H_i$ such that $(H_i',Z')$ is a consistent alignment.}

Let $H_i=h_1 \dd \cdots \dd h_k$. We may assume that $H_i$ is chosen  minimal satisfying Theorem \ref{connectifier2general}, and so there exist
   $z_1,z_k \in Z$ such that $N_{H_i}(z_j)=\{h_j\}$ for $j \in \{1,k\}$. 

We may assume that $|N_{Z}(h)| <y$ for every $h \in H_i$. Let $Z_1$ be the set of vertices in $Z$ with exactly one neighbor in $H_i$. Suppose that
$|Z_1| \geq \frac {|Z|} {3}$. It  follows that $Z_1$ contains a set $Z'$ with $|Z'| \geq \frac{|Z_1|} {y}  \geq \frac {|Z|} {3y}$ such that no two vertices in $Z'$ have a common neighbor in $H_i$. We may assume that $z_1, z_k \in Z'$. Therefore, $(H_i, Z')$ is a spiky alignment, as required. Thus we may assume that $|Z_1|<\frac{|Z|}{3}$.

Next, let $Z_2$ be the set of vertices in $z \in Z$ such that either $z \in \{z_1, z_k\}$ or  $z$ has exactly two neighbors in $H_i$, and they are adjacent.
Suppose that  $|Z_2| \geq \frac {|Z|} {3}$.  By choosing $Z'$ greedily, it follows that $Z_2$ contains a subset $Z'$ with the following specifications:
   \begin{itemize}
       \item $z_1, z_k \in Z'$; 
       \item $|Z'| \geq \frac {|Z_2|} {2y} \geq \frac {|Z|} {6y}$; and
       \item no two vertices in $Z'$ have a common neighbor in $H_i$. 
   \end{itemize}
   But then $(H_i, Z')$ is a triangular alignment, as required. 
   Thus we may assume that $|Z_2| < \frac{|Z|}{3}$.

    Finally, let $Z_3 = \{z_1, z_k\} \cup (Z \setminus (Z_1 \cup Z_2))$. It follows that  $|Z_3| \geq \frac{|Z|}{3}$.   Let $R$ be a path from $z_1$ to $z_k$ with $R^* \subseteq H_{3-i}$,
    and let $H$ be the hole $z_1 \dd H_i \dd z_k \dd R \dd z_1$.
    Let  $z \in Z_3 \setminus \{z_1,z_k\}$. Define $H_i(z)$ to be the  minimal subpath of $H_i$
    containing $N_{H_i}(z)$. Let the ends of $H_i(z)$ be $a_{z}$ and $b_{z}$.
    Next let  $Bad(z)=N_{H_i(z)}[a_{z},b_{z}]$. Since
    $H_i$ is a path, it follows that    $|Bad(z)| \leq 4$ for every $z$.  
Since  $N_{Z}(h) <y$ for every $h \in H_i$,  we can greedily choose
   $Z' \subseteq Z_3$ with     $|Z'| \geq \frac{|Z_3|}{4y} \geq \frac{|Z|}{12y}$,  where $z_1,z_k \in Z'$ and  such that
if $z, z' \in Z'$, then $z'$ is anticomplete to $Bad(z)$. It follows from
Lemma~\ref{pathnotinterleave} that $H_i(z) \cap H_i(z')=\emptyset$ for every
$z,z' \in Z'$, and so  $(H_i,Z')$ is a wide alignment.
      This proves \eqref{align}.

\sta{\label{consistent}
There exist    $Z' \subseteq Z$ with $|Z'|\geq x^2$,  
  and  $H_i' \subseteq H_i$ for $i=1,2$ such that
  $(H_i', Z')$ is a consistent alignment or a consistent
  connectifier.}

If both $(H_1,Z)$ and $(H_2,Z)$ are consistent connectifiers, \eqref{consistent} holds. Thus we may
assume that $H_1$ is a path and every vertex of $Z$ has a neighbor in $H_1$.
Suppose first that some  $h \in H_1$ has at least $36x^2$ neighbors in $Z$.
Let $H_1'=\{h\}$, and let $Z''\subseteq Z \cap N(h)$
with $|Z''|=36x^2$.  If $(H_2,Z)$ is a connectifier,  
we can clearly choose $H_2' \subseteq H_2$ such that  \eqref{consistent} holds. So we may assume that $H_2$ is a path and every vertex of $Z$ has a neighbor in $H_2$.
Moreover, no vertex $f$ of $H_2$ has three or more neighbors in $Z'$, for otherwise $\{f, h\} \cup (N(f) \cap Z')$ contains a theta with ends $h,f$. Now by \eqref{align} applied with $y = 3$, there is a set $Z'' \subseteq Z'$ with $|Z''| \geq x^2$
such that $(H_2,Z'')$ is a consistent alignment,   and \eqref{consistent} holds. Similarly, we may assume that every $h \in H_2$ has fewer than $36x^2$ neighbors in $Z$. 

Therefore, we may assume that every  vertex $h \in H_1$ has strictly fewer than $36x^2$ neighbors in $Z$.
      Applying  \eqref{align} with $y = 36x^2$, we conclude that there exists $Z_1\subseteq Z$ with      $|Z_1|\geq 36 \times 12x^4$ and a path $H_1' \subseteq H_1$ such that
      $(H_1', Z_1)$ is a consistent alignment.
      If $(H_2,Z)$ is a consistent connectifier, then \eqref{consistent} holds,
      so we may assume that $H_2$ is a path and every vertex of $Z$ has a neighbor in $H_2$. Since every vertex $h \in H_2$ has fewer than $36x^2$ neighbors in $Z$, we apply  \eqref{align} with $y = 36x^2$ again, and we conclude that there exists $Z' \subseteq Z_1$ with      $|Z'|\geq x^2$ and a path $H_2' \subseteq H_2$ such that
      $(H_2', Z')$ is a consistent alignment. This proves~\eqref{consistent}.

      \sta{There exist    $\hat{Z} \subseteq Z$ with $|\hat{Z}|\geq x$,  
  and  $\hat{H_i} \subseteq H_i$ for $i=1,2$ such that
  $(\hat{H}_i, \hat{Z})$ is a consistent alignment or a consistent
  connectifier.
        Moreover, if neither of $(\hat{H_1},\hat{Z)}, (\hat{H_2},\hat{Z})$ is a concentrated  connectifier, then 
  the order given on
$\hat{Z}$ by $(\hat{H_1},\hat{Z})$ and $(\hat{H_2}, \hat{Z})$
   is the same. \label{order}}

      Let $Z',H_1',H_2'$ be as in \eqref{consistent}.
      We may assume that neither of $(H_1',Z')$ and $(H_2',Z')$ is  a concentrated  connectifier. For $i \in \{1, 2\}$, let $\pi_i$ be the order given on $Z'$ by $(H_i',Z')$. By Theorem \ref{ESz}
   there exists $\hat{Z} \subseteq Z'$ such that (possibly reversing $H_i'$)
   the orders $\pi_i$ restricted to $\hat{Z}$ are the same, as required.
   This proves \eqref{order}.
   \\
   \\
   Now Theorem \ref{twosidesgeneral} follows from Lemma \ref{twosides_lemma}.
\end{proof}

  We now refine Theorem~\ref{twosidesgeneral}  in the two special cases
  we need here. The first one is when $D_1$ is a path. It is useful to state this result in terms of the following definition.

  Given a graph class $\mathcal{G}$, we say $\mathcal{G}$ is \textit{amiable} if, for every integer $x\geq 2$, there exists an integer $\sigma=\sigma(x) \geq 1$ such that the following holds for every graph $G \in \mathcal{G}$. Assume that $V(G)=D_1 \cup D_2 \cup Y$
  where
  \begin{itemize}
  \item         $Y$ is a stable set with $|Y| = \sigma$,
  \item $D_1$ and $D_2$ are components of $G \setminus Y$,
  \item $N(D_1)=N(D_2)=Y$,  
  \item $D_1=d_1 \dd \cdots \dd d_k$ is a path, and
  \item for every $y \in Y$ there exists $i(y) \in \{1, \dots, k\}$ such
    that $N(d_{i(y)}) \cap Y=\{y\}$.
  \end{itemize}
    Then there exist $X \subseteq Y$ with $|X|=x+2$,
$H_1 \subseteq D_1$ and $H_2 \subseteq D_2$ such that
\begin{enumerate}
\item One of the following holds:
  \begin{itemize}
  \item   $(H_1,X)$ is  a triangular alignment  and $(H_2,X)$ is a
    stellar connectifier or a spiky connectifier;
   \item 
     $(H_1,X)$ is a spiky alignment, and $(H_2,X)$ is a triangular connectifier
     or a clique connectifier;
  \item $(H_1,X)$ is a wide alignment and $(H_2,X)$ is a triangular
    connectifier or a clique connectifier; or
\item
  Both  $(H_1,X)$ and $(H_2,X)$ are alignments, and
 at least one of  $(H_1,X)$ and $(H_2,X)$  is not a spiky alignment, and at least one of $(H_1,X)$ and $(H_2,X)$ is not a triangular alignment.
   \end{itemize}   
  \item 
  If $(H_2,X)$ is not a  concentrated  connectifier, then
  the orders given on $X$ by $(H_1,X)$ and by $(H_2,X)$ are the same.
\item For every $x \in X$ except at most two, we have that $N_{D_1}(x)=N_{H_1}(x)$.
\end{enumerate}

\begin{theorem} 
\label{twosides_path}
  The class $\mathcal{C}$ is amiable.
\end{theorem}

\begin{proof}
  Let $x \geq 2$, and let $G \in \mathcal{C}, D_1 =d_1 \dd \cdots \dd d_k, D_2, Y$ be as in the definition of an amiable  class.   Let $\sigma(x)=\phi(x+4)$ where $\phi$ is as in Theorem~\ref{twosidesgeneral}.
  We start with the following.

  \sta{For every $d \in D_1$, we have that $|N_Y(d)| \leq 4$.
  \label{4nbrs}}

  Suppose there exists $i \in \{1, \dots, k\}$ and $y_1,y_2,y_3,y_4,y_5 \in Y \cap N(d_i)$.
  We may assume that $i(y_1)< \dots <i(y_5)$.
  By reversing $D_1$ if necessary, we may assume that $i(y_3) < i$, and so $i(y_2)<i-1$.
  Let $P$ be a path with ends $y_1,y_2$ and with interior in
  $d_{i(y_1)} \dd D_1 \dd d_{i(y_2)}$. Let $R$ be a path from $y_1$
  to $y_2$ with interior in $D_2$. Now there is a theta in $G$
  with ends $y_1,y_2$ and paths $P$,$R$ and $y_1 \dd d_i \dd y_2$,
  a contradiction. This proves~\eqref{4nbrs}.
  \\
  \\
  Now we apply Theorem~\ref{twosidesgeneral} and obtain a set $X'$ with
  $|X'|=x+4$ satisfying one of the outcomes of Theorem~\ref{twosidesgeneral}.
  Let $(H_1,X')$ and
  $(H_2,X')$ be as in Theorem~\ref{twosidesgeneral} where
  $H_1 \subseteq D_1$ (and so Theorem \ref{twosidesgeneral} may hold with the roles of $H_1$ and $H_2$ reversed). 
It follows from
  \eqref{4nbrs} that every vertex in $D_1$ has degree at most 6 in $G$ and degree 2 in $G[D_1]$, and so $(H_1,X')$ is an alignment.
  Let $z_1, \dots, z_{x+4}$
  be the order on $X'$ given by $(H_1,X')$.

  \sta {There exist at most two values of  $i \in \{2, \dots, x+3\}$
    for which  $N(z_i) \cap D_1 \neq N(z_i) \cap H_1$.
  \label{outside}}

  Suppose there exist three such values of $i$. Since $H_1$ is a path,
  there exist $i,j \in \{1, \dots, k\}$ such that $H_1=d_i \dd D_1 \dd d_j$,
  and we may assume (by reversing $D_1$ if necessary) that for two values
  $p,q \in \{2, \dots, x+3\}$, both $z_p$ and $z_q$ have neighbors
  in $d_1 \dd D_1 \dd d_{i-1}$. Let $P$ be a path from $z_p$ to $z_q$ with $P^* \subseteq d_1 \dd D_1 \dd d_{i-1}$.
  Since $p,q >1$, there is a path $Q$ from $z_p$ to $z_q$  with
  $Q^* \subseteq H_1 \setminus d_i$. But now we get a theta with ends
  $z_p,z_q$ and path $P$, $Q$, and a path from $z_p$ to $z_q$ with interior in
  $D_2$, a contradiction. This proves~\eqref{outside}.
  \\
  \\
  By \eqref{outside}, there exists $X \subseteq \{z_2, \dots, z_{x+3}\}$
  with $|X| = x$ and such that $N_{D_1}(z)=N_{H_1}(z)$ for every
  $z \in Z$. We obtain that  $(H_1,X \cup \{z_1, z_{x+4}\})$ and $(H_2,X \cup \{z_1, z_{x+4}\})$
  satisfy the first   statement
  of Theorem~\ref{twosides_path}. The second statement of
  Theorem~\ref{twosides_path} follows immediately from Theorem~\ref{twosidesgeneral}, and the third statement holds by the choice of $X$.
  \end{proof}

The second special case is when the clique number of $G$  is bounded, and $Y \cap \Hub(G) =\emptyset$; it is a result from \cite{TWX}.
\begin{theorem} [Abrishami, Alecu, Chudnovsky, Hajebi, Spirkl \cite{TWX}]
\label{twosides}
  For every pair of integers $t,x\geq 1$, there exists an integer $\tau=\tau(t,x) \geq 1$ with the following property.
    Let $G \in \mathcal{C}_t$ and assume that $V(G)=D_1 \cup D_2 \cup Y$
  where
  \begin{itemize}
  \item         $Y$ is a stable set with $|Y| = \tau$,
  \item $D_1$ and $D_2$ are components of $G \setminus Y$, and 
  \item $N(D_1)=N(D_2)=Y$.
  \end{itemize}
  Assume that $Y \cap \Hub(G)=\emptyset$.
  Then there exist $X \subseteq Y$ with $|X|=x$,
$H_1 \subseteq D_1$ and $H_2 \subseteq D_2$
   (possibly with the roles of $D_1$ and $D_2$ reversed) such that
  \begin{itemize}
    \item   $(H_1,X)$ is  a triangular connectifier or a triangular alignment;
  
  \item   $(H_2,X)$ is a stellar connectifier, or a spiky connectifier, or a
    spiky alignment or a wide alignment; and
    \item  if $(H_1,X)$ is a triangular alignment, then $(H_2,X)$ is not a
  wide alignment.
      \end{itemize}
  Moreover, if neither of $(H_1,X),(H_2,X)$ is a stellar connectifier, then
  the orders given on $X$ by $(H_1,X)$ and by $(H_2,X)$ are the same.
  \end{theorem}

\section{Bounding the number of non-hubs in a minimal separator}
\label{boundnonhubs}

The goal of this section is to prove Theorem~\ref{bound}.
This is the second main ingredient of the proof of Theorem~\ref{banana}
(in addition to
the machinery of Section~\ref{sec:centralbag_banana}). Theorem~\ref{bound}
is what allows us to iterate the construction of Section~\ref{sec:centralbag_banana} for $\mathcal{O}(\log|V(G)|)$  rounds (rather than just constantly many), which is what
we  need in the proof. We need the following definition and result. Given a graph $H$, a \emph{$(\leq p)$-subdivision} of $H$ is a graph that arises from $H$ by replacing each edge $uv$ of $H$ by a path from $u$ to $v$ of length at least 1 and at most $p$; the interiors of these paths are pairwise disjoint and anticomplete. 
\begin{theorem}[Lozin and Razgon \cite{lozin}]\label{ramsey2}
For all positive integers $p$ and $r$, there exists a positive integer $m = m(p,r)$ such that every
graph $G$ containing a ($\leq p$)-subdivision of $K_m$ as a subgraph contains either $K_{p,p}$ as a subgraph or a
proper ($\leq  p$)-subdivision of $K_{r,r}$ as an induced subgraph.
\end{theorem} 

Since graphs in $\mathcal{C}_t$ contain neither $K_{t+1}$ nor $K_{2,2}$ as an induced subgraph, we conclude that graphs in $\mathcal{C}_t$ do not contain $K_{t+1, t+1}$ as a subgraph. Likewise, graphs in $\mathcal{C}_t$ do not contain subdivisions of $K_{2,3}$ (in other words, thetas) as an induced subgraph. Therefore, letting $m = m(t+1, 3)$, we conclude: 
\begin{lemma} \label{lem:lozinrazgon}
    Let $t \in \mathbb{N}$. Then there exists an $m = m(t) \in \mathbb{N}$ such that the following holds:     Let $G$ in $\mathcal{C}_t$. Then $G$ does not contain a $(\leq 2)$-subdivision of $K_m$ as a subgraph. 
\end{lemma}

Recall that the Ramsey number $R(t,s)$ is the minimum integer such that every graph on
at least $R(t,s)$ vertices contains either a clique of size $t$ or a stable set of size $s$.

\begin{theorem}
  \label{bound}
  For every integer $t\geq 1$ there exists $\Gamma=\Gamma(t)$ with the following property. 
  Let $G \in \mathcal{C}_t$ and assume that $V(G)=D_1 \cup D_2 \cup Y$
  where
  \begin{itemize}
  \item $D_1$ and $D_2$ are components of $G \setminus Y$.
\item $Y$ is a stable set.
  \item $N(D_1)=N(D_2)=Y$.
    \item $Y \cap \Hub(G)=\emptyset$.
     \item There exist $a_1 \in D_1$ and $a_2 \in D_2$ such that
       $a_1a_2$ is a $|Y|$-banana in $G$.
  \end{itemize}
  Then $|Y| \leq \Gamma$.
  \end{theorem}

\begin{proof}
  Let $k=k(t)$ be as in Theorem~\ref{bananatohub}. We may assume that $k \geq 3$. Let $m = m(t)$ as in Lemma \ref{lem:lozinrazgon}. Let $\lambda = R(m, k).$
    Let $\tau=\tau(t, \lambda)$
  be as in Theorem \ref{twosides}. Set $\Gamma=\tau$.
    Applying  Theorem~\ref{twosides}, we deduce that there exist
  $H_1 \subseteq D_1$, $H_2 \subseteq D_2$ and $X \subseteq Y$
   with $|X|=\lambda$  such that:
  \begin{itemize}
    \item   $(H_1,X)$ is  a triangular connectifier or a triangular alignment;
  
  \item   $(H_2,X)$ is a stellar connectifier, or a spiky connectifier, or a
    spiky alignment, or a wide alignment; and
    \item  if $(H_1,X)$ is a triangular alignment, then $(H_2,X)$ is not a
  wide alignment.
      \end{itemize}

  \sta{Suppose $(H_2,X)$ is a wide alignment, and let $x_1, \dots, x_\lambda$
    be the order on $X$ given by $H_2$. Then for every $i \in \{2, \dots, \lambda-1\}$,
    $x_i$ has exactly 
    three neighbors in $H_2$, and two of them are adjacent.
  \label{wide}}

  Let $i \in \{2, \dots, \lambda-1\}$. Let $R$ be a path from $x_1$ to
  $x_\lambda$ with interior in $H_1$. Then $H=x_1 \dd H_2 \dd x_\lambda \dd R \dd x_1$
  is a hole. Since $(H_2,X)$ is a wide alignment, $x_i$ has two non-adjacent
  neighbors in $H$. Since $(H,x_i)$ is not a wheel in $G$, and $x_i$ is non-adjacent to $x_1$ and $x_\lambda$,  \eqref{wide}
  follows.
  \\
  \\
  Next we show:  

    \sta{Every vertex in $D_1$ has at most two neighbors in $X$.
  \label{threenbrs}}
  Suppose that there is a vertex $d \in D_1$ such that $d$ has three distinct neighbors $x_i, x_j, x_k$ in $X$; without loss of generality, we may assume that $i < j < k$. We consider two cases. First, if $H_2$ is a wide alignment, then $G$ contains a wheel $(C, x_j)$ where $C$ is a cycle consisting of $x_i, d, x_k$ and the path from $x_k$ to $x_i$ with interior in $H_2$. By \eqref{wide}, $x_j$ has four neighbors in $C$: $d$, as well as three neighbors in $H_2$. But $x_j$ is not a hub in $G$, a contradiction.
  
  It follows that $H_2$ is not a wide alignment. Now, for $r, s \in \{i, j, k\}$, let $P_{rs}$ be a path from $x_r$ to $x_s$ with interior in $H_2$. Then $\{d\} \cup P_{ij} \cup P_{jk} \cup P_{ik}$ is a theta in $G$ (see Figure \ref{fig:61}). This proves \eqref{threenbrs}.

  \begin{figure}[t!]
      \centering
\includegraphics[scale=0.6]{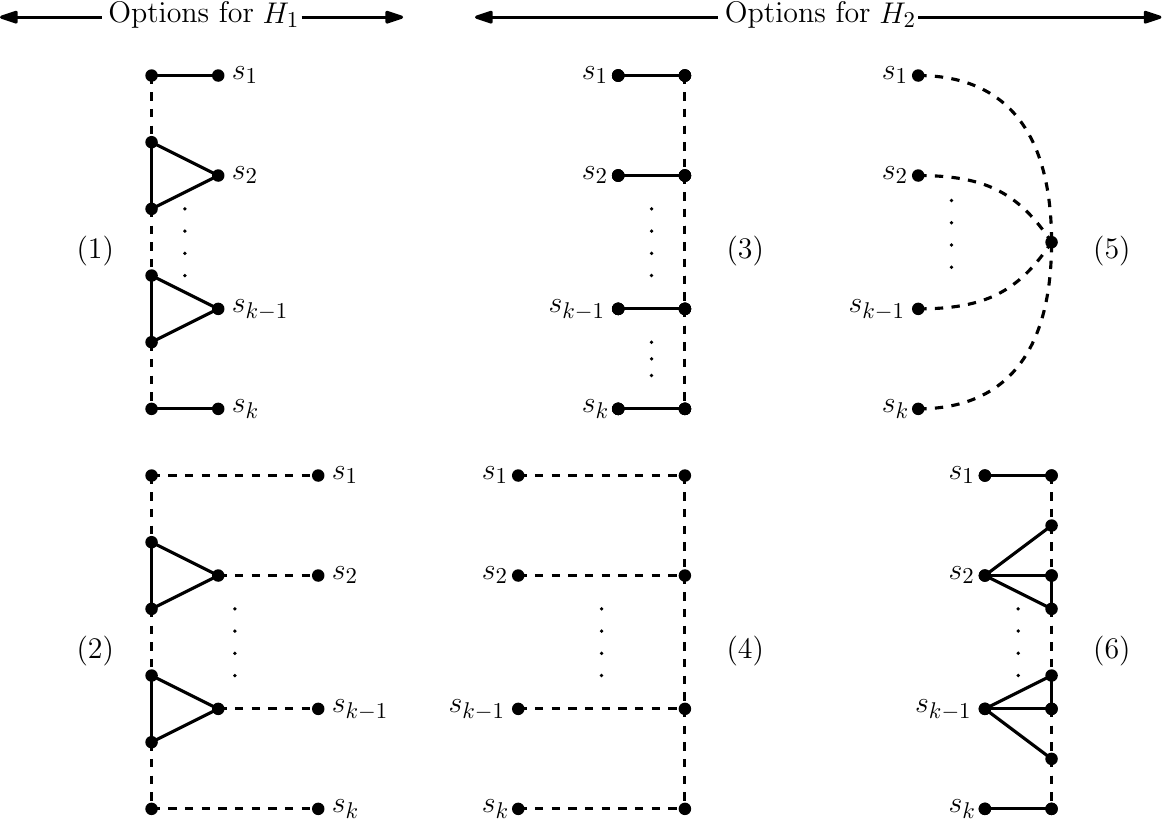}
      \caption{Options in the proof of Theorem \ref{bound} (except (1) and (6) cannot be the case together).}
      \label{fig:61}
  \end{figure}
  
  \sta{There exists $S \subseteq X$ with $|S|=k$, such that for every
    $s,s' \in S$ and for every path $P$ from $s$ to $s'$ with interior in
    $D_1$ we have that $|E(P)|>2$.
  \label{safeset}}

  Define a graph $H$ with vertex set  $X$ such that $xx' \in E(H)$ if and only if
  there is a path of length two with ends $x,x'$ and interior in $D_1$.
  Since $|X|=\lambda = R(m,k)$, there is $S \subseteq X$ with $|S|=k$ such that
  $S$ is either a clique or a stable set in $H$.

  If $S$ is a stable set, then
  \eqref{safeset} holds, so we may assume that $S$ is a clique of size $m$ in $H$. Write $S = \{s_1, \dots, s_m\}$. For $i, j \in \{1, \dots, m\}$ with $i \neq j$, we let $d_{ij}$ be a common neighbor of $s_i$ and $s_j$ in $D_1$ (which exists, since $S$ is a clique in $H$). By \eqref{threenbrs}, no vertex has three neighbors in $S$, and so the vertices $d_{ij}$ are pairwise distinct. But now $S \cup \{d_{ij} : i, j \in \{1, \dots, m\}, i \neq j\}$ contains a $(\leq 2)$-subdivision of $K_{m}$ as a subgraph, contrary to Lemma \ref{lem:lozinrazgon}. This is proves \eqref{safeset}.
  \\
  \\
  From now on let $S \subseteq X$ be as in \eqref{safeset}. Let
  $G'$ be the graph obtained from $D_1 \cup S$ by adding a new vertex $v$
  with $N(v)=S$. We need the following two facts about $G'$:

  \sta{$S \cap \Hub(G')=\emptyset$.
  \label{Snothub}}

  Suppose not, let $s \in S$ and let $(H,s)$ be a wheel in $G'$.
  Since $S \cap \Hub (G)=\emptyset$, it follows that $v \in H$.
  Let $N_H(v)=\{s',s''\}$; then $H \setminus v$ is a path $R$ from
  $s'$ to $s''$ with $R^* \subseteq D_1$. Since $(H,s)$ is a wheel,
  it follows that $s$ has two non-adjacent neighbors in $R$.
  Consequently, there is a path $P'$ from $s$ to $s'$ 
  and a path $P''^*$ form $s$ to $s''$, both with interior in $R$, such that
  $P'^*$ is disjoint from and anticomplete to $P''^*$.
  
   Let $T$ be a path from $s'$ to $s''$ with interior in $H_2$.
  Then $H'=s'  \dd R \dd s'' \dd T \dd s'$ is a hole.
  Since $X \cap \Hub(G)=\emptyset$, it follows that $(H',s)$ is not a
  wheel in $G$, and so $s$ is anticomplete to $T$.

  Suppose first that $(H_2,X)$ is a wide alignment. Since $s$ is anticomplete to $T$,
  it follows that (reversing the order on $X$ if necessary, and  exchanging the
  roles of $s',s''$ if necessary) $s,s',s''$ appear in
  this order in the order given by $H_2$ on $X$. 
  Now we get a theta with ends $s,s'$ and paths $s \dd P' \dd s'$,
  $s \dd P'' \dd s'' \dd H_2 \dd s'$, and
  $s \dd H_2 \dd s'$, a contradiction. This proves that $(H_2,X)$ is not a wide alignment.

  Since $(H_2,X)$ is not a wide alignment, there is a vertex $a \in H_2$,
  and three paths $Q,Q',Q''$ all with interior in $H_2$, where
  $Q$ is from $a$ to $s$, $Q'$ is from $a$ to $s'$, and $Q''$ is from
  $a$ to $s''$, and the sets $Q \setminus a$, $Q' \setminus a$ and
  $Q'' \setminus a$ are pairwise disjoint and anticomplete to each other. Note that
  $T=s' \dd Q' \dd a \dd Q'' \dd s''$. Since $s$ is anticomplete to $T$,
  it follows that $s$ is non-adjacent to $a$. But now we get a theta with ends
  $a,s$ and path $s \dd Q \dd a$, $s \dd P' \dd s' \dd Q' \dd a$,
  and $s \dd P'' \dd s'' \dd Q'' \dd a$, a contradiction. This proves
  \eqref{Snothub}.

\begin{figure}[t!]
    \centering
\includegraphics[scale=0.8]{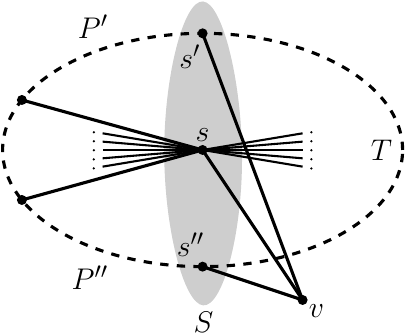}
\caption{Proof of \eqref{Snothub}.}
    \label{fig:63}
\end{figure}

  \sta{$G' \subseteq \mathcal{C}_t$.
  \label{G'}}

  Suppose not. Let $Z \subseteq G'$ be such that $Z$ is a $K_t$, a $C_4$, a theta, a prism or an even  wheel. Then $v \in Z$. Since $N_Z(v)$ is a stable set,
  it follows that $Z$ is not a $K_t$.
  Suppose first that
  $|N_Z(v)|=2$; then $N_Z(v)=\{s,s'\} \subseteq S$.
  Suppose that $Z$ is a $C_4$. Then there is $d \in D_1$ such that
  $Z=v \dd s \dd d \dd s' \dd v$,
  contrary to the fact that $S$ was chosen as
  in \eqref{safeset}.
  It follows that $Z$ is a theta, a prism or an even wheel.
  Let $Q$ be a path from $s$ to $s'$ with interior in $D_2$.
  Now $Z'=(Z \setminus v) \cup Q$ is an induced subgraph of $G$.
  Moreover, if $Z$ is a theta then $Z'$ is a theta, and if $Z$ is a prism
  then $Z'$ is a prism. Since $G \in \mathcal{C}$, it follows that
  $Z$ is an even wheel; denote it by $(H,w)$. Then $v \in H$.
  By \eqref{Snothub},
  $w \not \in S$. It follows that $N_Z(w)=N_{Z'}(w')$, and so
$Z'$ is an even
  wheel, contrary to the fact that $G \in \mathcal{C}$.

  This proves that $|N_Z(v)|>2$. Since $N_Z(v)$ is a stable set,
  it follows that $Z$ is not a prism.
  Suppose that $Z$ is a theta. Then the ends of $Z$ are $v$ and $d \in D_1$;
  let the paths of $Z$ be $Z_1,Z_2,Z_3$ where $N_{Z_i}(v)=s_i$. By renumbering
  $Z_1,Z_2,Z_3$, we may assume that $s_1,s_2,s_3$ appear in this order
  in the order on $X$ given by $H_1$.
  Then for every $i,j \in \{1, \dots, 3\}$, there is a unique path
  $F^1_{ij}$ from $s_i$ to $s_j$ with interior in $H_1$.
Similarly, for all $i,j \in \{1, \dots, 3\}$ there is a unique path
  $F^2_{ij}$ from $s_i$ to $s_j$ with interior in $H_2$.
  Observe that
  $F^1_{12} \cup F^1_{13} \cup F^1_{23} \cup F^2_{12} \cup F^2_{13} \cup F^2_{23}$
  contains a unique pyramid $\Sigma$ with the following specifications:
  \begin{itemize}
      \item  The paths of $\Sigma$ are $P_1,P_2,P_3$, where $s_i \in P_i$.
      \item  Denote the apex of $\Sigma$ by $a$, and let the base of $\Sigma$ be
     $b_1b_2b_3$ where $b_i$ is an end of $P_i$. Then either $a \in H_2$,
  or $(H_2,X)$ is a wide alignment and $a=s_2$. 
    \item $b_1, b_3 \in H_1$, and $b_2 \in H_1 \cup \{s_2\}$. 
  \end{itemize}

  Let $\Sigma_2=\Sigma \setminus D_1$.
  If $d$ is non-adjacent to $a$, then $(Z \setminus v) \cup \Sigma_2$
  is a theta with ends $a,d$ and paths $d \dd  Z_i \dd s_i \dd P_i \dd a$,
  contrary to the fact that $G \in \mathcal{C}_t$. This proves that
  $d$ is adjacent to $a$. Consequently, $a=s_2$ and $(H_2,X)$
  is a wide alignment. But now $H=s_1 \dd Z_1 \dd d \dd Z_3 \dd s_3 \dd F^2_{13}\dd s_1$
  is a hole, and, since $s_2$ is adjacent to $d$,  $(H,s_2)$ is a wheel, contrary to the fact that $S \cap \Hub(G)=\emptyset$.
  This proves that $Z$ is not a theta.

  Consequently, $Z$ is an even wheel $(H,w)$. Since $|N_Z(v)| >2$,
  it follows that either $v=w$, or $v$ is adjacent to $w$. If
  $v$ is adjacent to $w$, then $w \in S$, contrary to \eqref{Snothub}.
  It now follows that $w=v$, and so $H$ is a hole in
  $D_1 \cup S$, where $|S \cap H| = l\geq 3$. Let $S \cap H=\{s_1, \dots, s_l\}$
  where $s_1, \dots s_l$ appear in this order in the order  given on $S$ by $(H_2,X)$. 
    If $(H_2,X)$ is a spiky alignment or a wide alignment,
  or a spiky connectifier, then we get a theta with ends
  $s_1,s_2$ two of whose paths are contained in $H$, and the
  third is the path  from $s_1$ to $s_2$ with interior in $H_2$.
  It follows that $(H_2,X)$ is a
  stellar connectifier. Therefore there exist 
  paths 
    $Q_1, \dots, Q_l$, all with a common end $a \in H_2$,  such that
  $Q_i$ is from $a$ to $s_i$, $Q_i^* \subseteq H_2$, and the sets $Q_1 \setminus a, \dots, Q_k \setminus a$ are pairwise disjoint and anticomplete to each other.
  Since $(H,a)$ is a not an even wheel in $G$, we deduce that $a$ 
  is not complete to $\{s_1, \dots, s_l\}$; let $s_i \in \{s_1, \dots, s_l\}$
  be non-adjacent to $a$. Since $l \geq 3$, there exist distinct
  $p,q \in \{1, \dots, l\} \setminus \{i\}$  and paths $R_q$ and $R_p$ of $H$ such that
  \begin{itemize}
  \item $R_p$ is from $s_i$ to $s_p$;
  \item $R_q$ is from $s_i$ to $s_q$;
  \item $R_p^* \cap R_q^* = \emptyset$; and
  \item $R_p^* \cap \{s_1, \dots, s_l\}=R_q^* \cap \{s_1 \dots, s_l\}=\emptyset$.    \end{itemize}
  Now 
  we get a theta with ends $s_i,a$ and path $s_i \dd Q_i \dd a$,
  $s_i \dd R_p \dd s_p \dd Q_p \dd a$ and $s_i \dd R_q  \dd s_q \dd Q_q \dd a$.
    This proves \eqref{G'}.
  \\
  \\
  Observe that $a_1v$ is a $k$-banana in $G'$. But by \eqref{G'},
  $G' \in \mathcal{C}_t$, and by \eqref{Snothub}, $N_{G'}(v) \cap \Hub(G')=\emptyset$, contrary to Theorem~\ref{bananatohub}.
    \end{proof}

\section{The proof of Theorem~\ref{banana}}
\label{sec:bananaproof}

We can now prove our first main result. The ``big picture''  of the proof
is similar to \cite{TWX} and  \cite{TWIII}, but the context here is different.
For the remainder of the paper, all logarithms are taken in base 2.
We start with the following theorem from \cite{TWIII}:

\begin{theorem}[Abrishami, Chudnovsky, Hajebi, Spirkl \cite{TWIII}]
  \label{logncollections}
  Let $t \in \mathbb{N}$, and let $G$ be (theta, $K_t$)-free with $|V(G)|=n$.
  There exist an integer $d=d(t)$  
  and   a partition
  $(S_1, \dots, S_k)$   of $V(G)$ with the following properties:
  \begin{enumerate}
  \item $k \leq  \frac{d}{4} \log n$.
  \item $S_i$ is a stable set for every $i \in \{1, \dots, k\}$.
  \item For every $i \in \{1, \dots, k\}$ and $v \in S_i$ we have
    $\deg_{G \setminus \bigcup_{j <i}S_j}(v)  \leq d$. \label{hubsequence-3}
  \end{enumerate}
  \end{theorem}

Let $G \in \mathcal{C}_{t}$ be a graph and let $a,b \in V(G)$. A
{\em hub-partition with respect to $ab$} of $G$ is a partition
$S_1, \dots, S_k$ of $\Hub(G) \setminus \{a,b\}$ as in
Theorem \ref{logncollections};
we call $k$ the {\em order} of the partition.
We call the {\em hub-dimension} of $(G,ab)$ (denoting it by
$\hdim(G,ab)$) the smallest $k$ such that $G$ has a hub-partition of order $k$
with respect to $ab$.

For the remainder of this section, let us fix $t \in \mathbb{N}$. Let $d = d(t)$ be as in Theorem \ref{logncollections}.
Let $m=k+2d$ where $k = k(t)$ is as in Theorem \ref{bananatohub}.
Let $C_t=(4m+2)(m-1)$.
Let $\Gamma=\max\{k, \Gamma(t)\}$ where $\Gamma(t)$ is as in
Theorem~\ref{bound}.

Since in view of Theorem~\ref{logncollections}, we have
$\hdim(G,ab) \leq  \frac{d}{4} \log n$ for all $a,b \in V(G)$,
Theorem~\ref{banana} follows immediately from the next result:

\begin{theorem}
  \label{diminduction_banana}
  Let $G \in \mathcal{C}_{t}$ and with $|V(G)|=n$ and let $a,b \in V(G)$.
  Then $ab$ is not a  $C_t+8m^2\Gamma \hdim(G,ab)$-banana in $G$.
\end{theorem}

\begin{proof}
  Suppose that $ab$ is a  $C_t+8m^2\Gamma \hdim(G,ab)$-banana  in $G$.
  We will get a contradiction by induction on $\hdim(G,ab)$, and for fixed $\hdim$ by induction on $n$.
  Suppose that  $\hdim(G,ab)=0$. Then $\Hub(G) \subseteq \{a,b\}$ and  by Theorem \ref{banana_emptyhub}, we have that
  there is no $k$-banana in $G$. Now the statement holds since $k<C_t$.
    Thus we may assume $\hdim(G,ab)>0$.

    \sta{We may assume that $G$ admits no clique cutset.
    \label{cliquecutset}}

    Suppose $C$ is a clique cutset in $G$. Since $G$ is $K_t$-free, it
    follows that there is a component $D$ of $G \setminus C$ such that
    $ab$ is  $C_t+8m^2\Gamma \hdim(G,ab)$-banana in $G'=D \cup C$.
    But $|V(G')|<n$ and $\hdim(G',ab) \leq \hdim(G,ab)$, consequently we get
    contradiction
    (inductively on $n$). This proves \eqref{cliquecutset}.
    \\
    \\
    Let $S_1, \dots, S_q$ be a hub-partition of $G$ with respect to
    $ab$ and with $q=\hdim(G,ab)$.
  We now use notation and terminology from Section \ref{sec:centralbag_banana} (note that the definitions of $k$ and $m$ agree; we use $d$, $a$, $b$, as defined above).
It follows from the definition of $S_1$ that every vertex in $S_1$
is $d$-safe.
Let $(T,\chi)$ be an $m$-atomic
tree decomposition of $G$. Let $t_1,t_2$  be  as in Theorem~\ref{basket}.
and let $\beta=\beta(S_1)$ and $\beta^A(S_1)$ be as in Section \ref{sec:centralbag_banana}.  Then by Lemma~\ref {abinbasket}, we have $a,b \in \beta^A(S_1)$.
By Theorem \ref{A_centralbag}\eqref{A-4}, we have that   $S_1 \cap \Hub(\beta^A(S_1))=\emptyset$ and
   $S_2 \cap \Hub(\beta^A(S_1)), \dots, S_q \cap \Hub(\beta^A(S_1))$ is a hub-partition of
   $\beta^A(S_1)$ with respect to $ab$. 
It follows that $\hdim(\beta^A(S_1),ab) \leq q-1$.
Inductively (on $\hdim(\cdot, ab)$), we have that
$ab$ is not a $C_t+8m^2\Gamma (q-1)$-banana in $\beta^A(S_1)$.

Our first goal is to  prove:

\sta{There is a set $Y' \subseteq \beta$ such that
\begin{enumerate}
\item  $Y'$ separates $a$ from $b$
  in $\beta$;
\item  $|Y' \cap (\chi(t_1) \cup \chi(t_2))| \leq C_t+8m^2\Gamma(q-1)+4m^2\Gamma$;
  and
  \item  $|\Delta(Y')| \leq (4m+1)(m-1)\Gamma$.
  \end{enumerate}
  \label{beta}} 
Since $ab$ is not a $(C_t+8m^2\Gamma (q-1))$-banana in $\beta^A(S_1)$, 
Theorem~\ref{Menger_vertex} implies that there is a separation
$(X,Y,Z)$ of $\beta^A(S_1)$ such that $a \in X$ and
$b \in Z$ and  $|Y| \leq C_t+8m^2\Gamma (q-1)$. 
We may assume that $(X,Y,Z)$ is chosen with $|Y|$ as small as possible.
Let $D_1$ be the component of $X$ such that $a \in D_1$,
and and let $D_2$ be the component of $Z$ such that $b \in D_2$.
It follows from the minimality of $|Y|$ that $N(D_1)=N(D_2)=Y$
and $ab$ is a $|Y|$-banana in $D_1 \cup D_2 \cup Y$.

We claim that $|Y \cap S_1| \leq \Gamma$.
Let $G'=D_1 \cup D_2 \cup (Y \cap S_1)$. Again by the minimality of
$Y$, it follows that $ab$ is a $|Y \cap S_1|$-banana in $G'$.
If $ab$ is not a $k$-banana in $G'$, then $|Y \cap S_1| <k$ and
the claim follows immediately since
$k \leq \Gamma$; thus we may assume that $ab$ is a $k$-banana in $G'$.
Now since $S_1$ is a stable set and since no vertex of $S_1$ is a hub in
$\beta^A(S_1)$, applying Theorem~\ref{bound} to $G'$
we deduce that $|Y \cap S_1| \leq \Gamma$, as required. This proves the claim.

 Next we claim that $|\delta_1(Y) \cup \delta_2(Y)| \leq \Gamma$, and
$|\Delta(Y)| \leq (m-1)\Gamma$. Write $\beta_0=\chi(t_1) \cup \chi(t_2)$.
Observe that for distinct $t,t' \in \delta_1(Y) \cup \delta_2(Y)$, say with $t \in \delta_i(Y)$ and $t' \in \delta_{i'}(Y)$ for $i, i' \in \{1, 2\}$, the sets
$G_{t_i \rightarrow t} \setminus \beta_0$ and $G_{t_{i'} \rightarrow t'} \setminus \beta_0$ are 
disjoint and anticomplete to each other.
Let $W$ be a subset of $Y$ such that for all $i \in \{1, 2\}$ and $t \in \delta_i(Y)$, we have
$|W \cap (G_{t_i \rightarrow t} \setminus \beta_0))|=1$.  Then $W$ is  stable set and $|W|=|\delta_1(Y) \cup \delta_2(Y)|$. 
Let $G'=D_1 \cup D_2 \cup W$. 
If follows from the minimality of $Y$ that $ab$ is a $|W|$-banana in $G'$.
If $ab$ is not a $k$-banana in $G'$, then $|W| <k$ and
the claim follows immediately since
$k \leq \Gamma$; thus we may assume that $ab$ is a $k$-banana in $G'$.
By Theorem \ref{connectors}\eqref{C-3}, $W \cap \Hub(\beta)=\emptyset$. 
Now since $W$ is a stable set, applying Theorem~\ref{bound} to $G'$
we deduce that $|W| \leq \Gamma$, as required.
Since $\adh(T,\chi) \leq m-1$, it follows that $|\Delta(Y)| \leq (m-1)\Gamma$.
This proves the claim.

Now let $Y'$ be as in  Theorem~\ref{smallsepinbeta}.
Then $Y'$ separates $a$ from $b$ in $\beta$. Moreover, by
Theorem~\ref{smallsepinbeta}, 
$$|Y' \cap (\chi(t_1) \cup \chi(t_2))| \leq |Y|+4m(m-1)|Y \cap \Core(S_1)|.$$
Since $|Y \cap \Core(S_1)| \leq |Y \cap S_1| \leq \Gamma$, it follows that
$$|Y' \cap (\chi(t_1) \cup \chi(t_2))| \leq C_t+8m^2\Gamma(q-1)+4m^2\Gamma.$$
Finally, by Theorem~\ref{smallsepinbeta},
$|\Delta(Y') \setminus \Delta(Y)| \leq  4m(m-1)|Y \cap \Core(S_1)|$. Since $|\Delta(Y)| \leq (m-1)\Gamma$ and $|Y \cap \Core(S_1)| \leq \Gamma$,
it follows that $|\Delta(Y')| \leq (4m+1)(m-1)\Gamma$.
This proves~\eqref{beta}.
\\
\\
Let $Y'$ be as in \eqref{beta}, and let
$(X',Y',Z')$ be a separation of $\beta$ such that $a \in X'$
and $b \in Z'$. 
Let $Y''$ be obtained from $Y'$ as in Theorem~\ref{smallseping}.
Then $Y''$ separates $a$ from $b$ in $G$. 
Recall that  by \eqref{beta}, we have  $|\Delta(Y')| \leq (4m+1)(m-1)\Gamma$.
Now since $\adh(T, \chi) < m$ and since $|{S_1}_{bad}| \leq 3$, 
it follows from Theorem~\ref{smallseping} that
\begin{align*}
    |Y''| &\leq |Y' \cap (\chi(t_1) \cup \chi(t_2))|+(4m+1)(m-1)\Gamma +2(m-1)+3\\
    &\leq |Y' \cap (\chi(t_1) \cup \chi(t_2))|+4m^2\Gamma.
\end{align*}
Since $|Y' \cap (\chi(t_1) \cup \chi(t_2))| \leq C_t+8m^2\Gamma (q-1)+4m^2\Gamma$ by \eqref{beta}, we
deduce that $|Y''| \leq C_t+8m^2\Gamma q$ as required.
\end{proof}

\section{Dominated balanced separators} \label{sec:domsep}

Several more steps are needed to complete the proof of
Theorem~\ref{main}. We take the first one in this section.
The goal of this section is to prove Theorem~\ref{balancedsep}, which we
restate:

\begin{theorem}
  \label{balancedsep_2}
There is an integer  $d$ with the following property.
Let $G \in \mathcal {C}$ and let $w$ be a normal weight function on $G$.
Then there exists $Y \subseteq V(G)$ such that
\begin{itemize}
\item $|Y| \leq d$, and
\item $N[Y]$ is a $w$-balanced separator in $G$.
\end{itemize}
\end{theorem}

We start with two decomposition theorems that will become the engine for obtaining separators.  The first is a more explicit version of
Theorem~\ref{wheelstarcutset}; it shows that the presence of a wheel in the graph forces a decomposition that is an extension of what can be locally observed
in the wheel (see \cite{wallpaper} for detailed treatment of this concept).

\begin{theorem}[Addario-Berry, Chudnovsky, Havet, Reed, Seymour \cite{Addario-Berry2008BisimplicialGraphs}, da Silva, Vu\v{s}kovi\'c \cite{daSilva2013Decomposition2-joins}]
  Let $G$ be a $C_4$-free odd-signable graph that contains a proper wheel
  $(H, x)$ that is not a universal wheel. Let $x_1$ and $x_2$ be the endpoints of a long sector $Q$ of $(H, x)$. Let $W$ be the set of all vertices $h$ in $H \cap N(x)$
  such that the subpath of $H \setminus \{x_1\}$ from $x_2$ to $h$ contains an even number of neighbors of $x$, and let $Z = H \setminus (Q \cup N(x))$. Let $N' = N(x)\setminus W$. Then, $N' \cup \{x\}$ is a cutset of $G$ that separates $Q^*$ from $W \cup Z$.
\label{lemma:proper_wheel_forcer}
\end{theorem}

The second is a similar kind of theorem, but we start with a pyramid rather than a wheel.

\begin{theorem}
    \label{apexnbrs}
    Let $G \in \mathcal{C}$.
    Let $\Sigma$ be a pyramid with paths $P_1, P_2, P_3$, apex $a$, and base $b_1b_2b_3$ in $G$ and let $i,j  \in \{1,2,3\}$ be distinct.
     For $i \in \{1, \dots, 3\}$, let $a_i$ be the neighbor of $a$ in $P_i$.
     Let $P$ be a path from a vertex of $P_i \setminus \{a,a_i,b_i\}$ to a
     vertex of 
     $P_j \setminus \{a,a_j,b_j\}$.
Then at least one one of $a,b_1,b_2,b_3$ has a neighbor in $P^*$.
     \end{theorem}

  \begin{proof}
    It follows from the defintion of $P$ that there exist distinct $i,j \in \{1,2,3\}$ and a subpath
    $P'=p_1 \dd \cdots p_k$ of $P$ such that $p_1$ has a neighbor in $P_i\setminus \{a,a_i, b_i\}$, and $p_k$ has a neighbor in  $P_j\setminus \{a,a_j,b_j\}$.
    Let $P'$ be chosen minimal with this property.
    It follows that $P' \subseteq P^*$, and so we  may assume that $\{a,b_1,b_2,b_3\}$ is anticomplete to $P'$. We may also assume $i=1$ and $j=3$. See Figure \ref{fig:81}. It follows that
   and  $a_1 \neq b_1$, $a_3 \neq b_3$.

\begin{figure}
    \centering
    \includegraphics[scale=0.8]{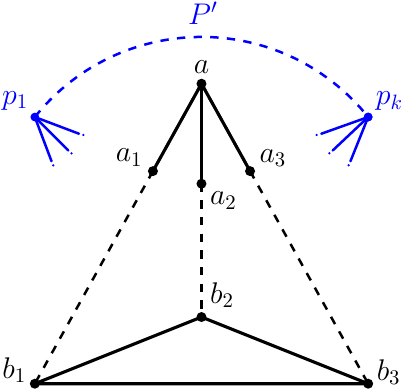}
    \caption{Proof of Theorem \ref{apexnbrs}.}
    \label{fig:81}
\end{figure}
   \sta{\label{antiP2} $P'$ is anticomplete to $P_2 \setminus a_2$.}

    Suppose not; then some vertex in $P'$ has a neighbor in $P_2 \setminus \{a, a_2, b_2\}$. It follows from the minimality of $P'$ that
    $k=1$. Then $p_1$ has a neighbor in each of the paths
    $P_1 \setminus b_1,  P_2 \setminus b_2$ and $P_3 \setminus b_3$.
Now we get a theta with ends
    $p_1,a$ whose paths are subpaths of $P_1,P_2,P_3$, a contradiction.
This proves \eqref{antiP2}.

\sta{\label{a1P'notpk} If $k>1$ and $a_1$ is adjacent to $p_k$, then $a_1$ has a neighbor in
      $P' \setminus p_k$.}

    Suppose that $a_1$ is adjacent to $p_k$ and anticomplete to
    $P'\setminus p_k$.
Since $a_2 \dd a \dd a_1 \dd p_k \dd a_2$ is not a $C_4$ in $G$, and by \eqref{antiP2}, it follows that 
        $p_k$ is anticomplete to $P_2$. If $a_2$ has a neighbor in $P' \setminus p_1$, then we get a theta with
    ends $a_1,a_2$,
    two of whose paths are contained in the hole 
    $b_1 \dd P_1 \dd a \dd P_2 \dd b_2 \dd b_1$, and third one has interior
    on $P' \setminus p_1$, a contradiction. Consequently, $a_2$ is anticomplete to $P' \setminus p_1$. Next suppose that $a_3$ has a neighbor in $P'$.
    Since $a_1 \dd a \dd a_3 \dd p_k \dd a_1$ is not a $C_4$ in $G$,
    it follows that $a_3$ is not adjacent to $p_k$. Now we get a theta with ends $a_3,p_k$ and path $p_k \dd a_1 \dd a \dd a_3$, a path with interior in
    $P'$ and a path with interior in $P_3$. This proves that $a_3$ is anticomplete to $P'$.

    Let $x$ be the neighbor of $a_1$ in $P_1 \setminus a$.
    Suppose that  $p_1$ has a neighbor in
    $P_1 \setminus \{a,a_1,x\}$. Then there is a path $T$ from $p_1$ to $a$ with $T^* \subseteq (P_1 \cup P_2) \setminus \{a,a_1,x\}$ and  
    we get  a theta with ends  $p_k,a$ and paths $p_k \dd a_1 \dd a$ and $p_k \dd P' \dd p_1 \dd T \dd a $, and a path with interior in $P_3^*$,  a contradiction. It follows that $p_1$ is anticomplete to $P_1 \setminus \{a_1,a,x\}$
    and therefore 
    $p_1$ is adjacent to $x$. Since $a_1$ is anticomplete to
    $P' \setminus p_k$ we deduce that $N_{P_1}(p_1)=\{x\}$.
    Now we get a theta with ends $x,p_k$ and paths $x \dd a_1 \dd p_k$, $x \dd  P' \dd p_k$ and $x \dd P_1 \dd b_1 \dd b_3 \dd P_3 \dd p_k$, a contradiction. This proves \eqref{a1P'notpk}.

    \sta{\label{notbotha1a3} Either $N_{P'}(a_1) \subseteq \{p_1\}$ or
      $N_{P'}(a_3) \subseteq \{p_k\}$.}

    Suppose not.  Then $k>1$. If both $a_1$ and $a_3$ have a neighbor in $P'^*$, then there is a theta in $G$ with ends $a_1,a_3$, and paths
     $a_1 \dd a \dd a_3, a_1 \dd P_1 \dd b_1 \dd b_3 \dd P_3 \dd a_3$ and $a_1 \dd P' \dd a_3$, a contradiction. By symmetry we may assume that $a_1$ is anticomplete to
     $P'^*$; now by \eqref{a1P'notpk}, it follows that $a_1$ is adjacent to both $p_1$ and $p_k$.
     Since $a_1 \dd p_1 \dd a_3 \dd a \dd a_1$ is not a $C_4$, we deduce that
     $a_3$ is non-adjacent to $p_1$. Similarly $a_3$ is non-adjacent to  $p_k$.
     It follows that $a_3$ has a neighbor in $P'^*$. Let $S$ be a path from $a_3$ to $p_k$ with interior in $P'^*$. Now we get a theta with ends
     $a_3,p_k$ and path $a_3 \dd S \dd p_k$, $a_3 \dd a \dd a_1 \dd p_k$
     and a path with interior in $P_3$, a contradiction.
     This proves \eqref{notbotha1a3}.

     \sta{\label{notbotha1a2} If $a_1$ has a neighbor in $P' \setminus p_1$,
       then     $a_2$ is anticomplete  to $P'$.}

       Suppose that $a_1$ has a neighbor in $P' \setminus p_1$ and $a_2$ has a
       neighbor in $P'$. Then $k>1$. If both $a_1$ and $a_2$ have a neighbor in $P'\setminus p_1$, then there is a theta in $G$ with ends $a_1,a_2$, and paths
       $a_1 \dd a \dd a_2, a_1 \dd P_1 \dd b_1 \dd b_2 \dd P_2 \dd a_2$ and $a_1 \dd P' \dd a_2$, a contradiction. This proves that $N_{P'}(a_2)=\{p_1\}$.

 Let $S_1$ be a path from $a_1$ to $p_1$ with $S_1^* \subseteq P'$.
     Since $a_1 \dd p_1 \dd a_2 \dd a \dd a_1$ is not a $C_4$, it follows that $a_1$ is non-adjacent to $p_1$. But now we get a theta with ends $a_1,p_1$ and paths $a_1 \dd a \dd a_2 \dd p_1$, $a_1 \dd S_1 \dd p_1$ and a path
     from $a_1$ to $p_1$ with interior in $P_1$,
     a contradiction. This proves \eqref{notbotha1a2}.

     \sta{\label{nochords} $N_{P'}(a_1) \subseteq \{p_1\}$ and
      $N_{P'}(a_3) \subseteq \{p_k\}$.} 

    Suppose $a_1$ has a neighbor in $P' \setminus p_1$. Then $k>1$.
    By \eqref{notbotha1a3}, 
    $a_3$ is anticomplete to $P' \setminus p_k$, and by \eqref{notbotha1a2}, 
    $a_2$ is anticomplete to $P'$.
    Let $H_1$ be the hole
    $b_1 \dd P_1 \dd p_1 \dd P' \dd p_k \dd P_3 \dd b_3 \dd b_1$,
    and let $H_2$ be the hole $p_1 \dd P_1 \dd b_1 \dd b_2 \dd P_2 \dd a \dd P_3 \dd p_k \dd P' \dd p_1$ ($H_2$ is a hole since $p_k$ has a neighbor in $P_3 \setminus b_3$). Then $N_{H_2}(a_1)=N_{H_1}(a_1) \cup \{a\}$. Let $p_0$ be the neighbor of
    $p_1 \in H_1 \setminus P'$ and let $Q$ be the path $p_0 \dd p_1 \dd P' \dd p_k$.
    Observe that $Q \subseteq H_2$, and $N_{H_1}(a_1) \subseteq Q$. Since neither of
    $(H_1,a_1)$ and $(H_2,a_1)$ is an even wheel, it follows that $a_1$ has exactly two neighbors in $Q$ and they are adjacent. Let $s \in \{0, \dots, k-1\}$ be such that $N_Q(a_1)=\{p_s,p_{s+1}\}$. Since $a_1$ has a neighbor in $P_1 \setminus p_1$, it follows that $s >0$. Now we get a prism with triangles
    $p_sa_1p_{s+1}$ and $b_1b_2b_3$ and paths $a_1 \dd a \dd P_2 \dd b_2$,
    $p_s \dd P' \dd p_1 \dd p_0 \dd P_1 \dd b_1$ and $p_{s+1} \dd P' \dd p_k \dd P_3 \dd b_3$, a contradiction. This proves \eqref{nochords}.

    \sta{\label{P2anti} $P_2$  is anticomplete to $P'$.}
Suppose not. By \eqref{antiP2}, $a_2$ has a neighbor in $P'$. 
Then $a_2 \neq b_2$.
    Let $H_1$ be the hole
    $p_1 \dd P' \dd p_k \dd P_3 \dd b_3 \dd b_1 \dd P_1 \dd p_1$ and
    let $H_2$ be the hole $p_1 \dd P' \dd p_k \dd P_3 \dd a_3 \dd a \dd a_1 \dd P_1 \dd p_1$ ($H_2$ is a hole since $\{b_1,b_3\}$ is anticomplete to
    $P'$). Since  $a_2 \neq b_2$, we have that
    $N_{H_2}(a_2)=N_{H_1}(a_2) \cup \{a\}$. Since
    since neither of
    $(H_1,a_2)$, $(H_2,a_2)$ is an even wheel, it follows that
            there exists $s \in \{1, \dots, k-1\}$ such that
$N_{P'}(a_2)=\{p_s,p_{s+1}\}$.
Now we get a prism with triangles
$b_1b_2b_3$ and $p_sa_2p_{s+1}$ and paths $p_s \dd P' \dd p_1 \dd P_1 \dd b_1$,
$a_2 \dd P_2 \dd b_2$ and $p_{s+1} \dd P' \dd p_k \dd P_3 \dd b_3$,
a contradiction. This proves~\eqref{P2anti}.

\sta{\label{notunique} $p_1$ has at least two neighbors in $P_1$, and
    $p_3$ has at least two neighbors in $P_3$.}

    By symmetry it is enough to prove the first assertion of \eqref{notunique}.
    Suppose that $p_1$ has a unique neighbor $x$ in $P_1$. It follows that $x \neq a_1$. 
   By \eqref{P2anti}, $P_2$ is anticomplete to $P'$.
    Now we get a theta with   ends $x,a$ and paths $x \dd P_1 \dd a$, $x \dd P_1 \dd b_1 \dd b_2 \dd P_2 \dd a$ and $x \dd p_1 \dd P' \dd p_k \dd P_3 \dd a$, a contradiction. This proves \eqref{notunique}.

\sta{\label{nonadj} $p_1$ has  two non-adjacent neighbors in $P_1$, and
  $p_3$ has two non-adjacent  two neighbors in $P_3$.}

By symmetry it is enough to prove the first statement.
By \eqref{notunique}, we may assume that $p_1$ has exactly two neighbors $x,y$
in $P_1$, and $x$ is adjacent to $y$. We may assume that $P_1$ traverses
$b_1,x,y,a_1$ in this order.
By \eqref{P2anti}, $P_2$ is anticomplete to $P'$.
Let $S$ be the path
from $p_k$ to $b_3$ with interior in $P_3$. It follows from the defintion of
$P'$ that $a$ is anticomplete to $S$. Now we get a prism with triangles
$xyp_1$ and $b_1b_2b_3$ and paths $x \dd P_1 \dd b_1$, $y \dd P_1 \dd a \dd P_2 \dd b_2$
and $p_1 \dd P' \dd p_k \dd S \dd b_3$, a contradiction. This proves \eqref{nonadj}.
\\
\\
By \eqref{nonadj}, there exist paths $P_1'$ from $p_1$ to $a$ and $P_1''$
from $p_1$ to $b_1$, both with interior in $P_1^*$,
and such that $P_1' \setminus p_1$ is anticomplete to
$P_1'' \setminus p_1$. Let  $P_3'$ be a path from $p_k$ to $a$
with interior in $P_3^*$.
By \eqref{P2anti}, $P_2$ is anticomplete to
$P'$. Now we get a theta with ends
$p_1,a$ and paths $p_1 \dd P_1' \dd a$, $p_1 \dd P_1'' \dd b_1 \dd b_2 \dd P_2 \dd a$
and $p_1 \dd P' \dd p_k \dd P_3' \dd a$, a contradiction.
  \end{proof}

We need a technical definition. Given an integer $m>0$ and a graph class $\mathcal{G}$, we say $\mathcal{G}$ is \textit{$m$-amicable} if $\mathcal{G}$ is amiable, and the following holds for every graph $G\in \mathcal{G}$.  Let $\sigma$ be as in the definition of an amiable class for $\mathcal{G}$ and let $V(G)=D_1\cup D_2\cup Y$ such that $D_1=d_1 \dd \cdots \dd d_k, D_2$ and $Y$ satisfy the first five bullets from the definition of an amiable class with $|Y| = \sigma(7)$. Let $X\subseteq Y$, $H_1\subseteq D_1$ and $H_2\subseteq D_2$ be as in the definition of an amiable class with $|X|=9$, and let $\{x_1, \dots, x_7\} \subseteq X$ such that: 
\begin{itemize}
    \item $x_1, \dots, x_7$ is the order given on $\{x_1, \dots, x_7\}$ by $(H_1,X)$; and 
    \item For every $x \in \{x_1, \dots, x_7\}$, we have $N_{D_1}(x) = N_{H_1}(x)$.
\end{itemize} Let $i$ be maximum such that $x_1$ is adjacent to $d_i$, and let $j$ be minimum such that $x_7$ is adjacent to $d_j$. Then there exists a subset $Z \subseteq D_2 \cup \{d_{i+2}, \dots, d_{j-2}\} \cup  \{x_4\}$
    with $|Z| \leq m$  such that $N[Z]$ separates  $d_i$ from $d_j$. It follows that $N[Z]$ separates $d_1 \dd D_1 \dd d_i$ from
    $d_j \dd D_1 \dd d_k$.

We deduce:
  \begin{theorem}
    \label{separatepath}
   The class $\mathcal{G}$ is $4$-amicable.
    \end{theorem}
    
  \begin{proof}
From Theorem~\ref{twosides_path}, we know that $\mathcal{G}$ is amiable. With the notation as in the definition of an amicable class, we need to show that there exists a subset
    $Y \subseteq D_2 \cup \{d_{i+2}, \dots, d_{j-2}\} \cup  \{x_4\}$
    with $|Z| \leq 4$  such that $N[Z]$ separates  $d_i$ from $d_j$. 
Let $R$ be the path from $x_1$ to $x_7$ with interior in $H_2$.

    Let $H$ be the hole
    $x_1 \dd H_1 \dd x_7 \dd R \dd x_1$. Let $L=l_1 \dd \cdots \dd  l_q$ be the path in
    $H_2 \cup \{x_4\}$ such that $l_1=x_4$ and $l_q$ has a neighbor in $P(H_2)$.
    Thus if $(H_2,X)$ is an alignment
    then $L=x_4$, and if $(H_2,X)$ is
    a connectifier then $L \setminus x_4$ is the leg of $H_2$ containing a
    neighbor of $x_4$. See Figure \ref{fig:84}. (Note that $L \setminus x_4$ may be empty if $(H_2,X)$ is a clique connectifier or a stellar connectifier.)
    We consider different possibilities for the behavior of $(H_1,X)$ and
    $(H_2,X)$. 

    \begin{figure}
        \centering
        \includegraphics[scale=0.7]{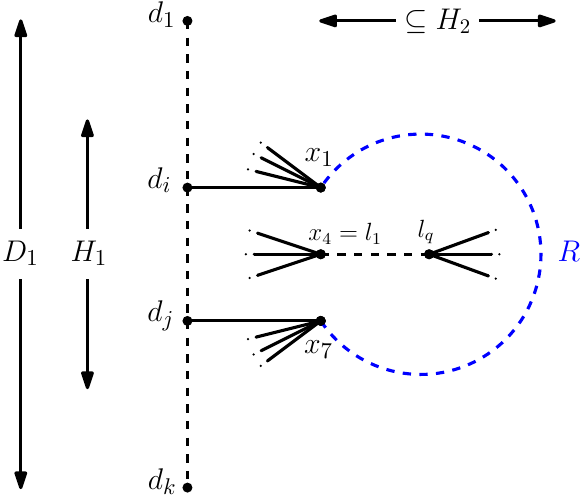}
        \caption{Proof of Theorem \ref{separatepath}.}
        \label{fig:84}
    \end{figure}

    \sta{If $(H_1,X)$ is a wide alignment, then the theorem holds.
    \label{H1wide}}

    Suppose first that $(H_2,X)$ is an alignment.  Then $(H,x_4)$ is a proper wheel, and we are done by Theorem~\ref{lemma:proper_wheel_forcer} and setting $Z=\{x_4\}$. 
    Thus we may assume that $(H_2,X)$ is a connectifier.
    It now follows from Theorem~\ref{twosides_path}  that $(H_2,X)$ is a
    triangular connectifier or a clique connectifier. Let $N_H(l_q)=\{h,h'\}$.
    Then $H \cup L$ contains a pyramid $\Sigma$ with apex $x_4$
    and base $l_qhh'$. 
    Since $\{d_i,d_j\}$ is anticomplete to
    $\{x_4,h,h',l_q\}$, it follows that $d_i$ belongs to the interior of
the path of $\Sigma$ with ends $x_4,h$, and $d_j$  belongs to the interior of
the path of $\Sigma$ with ends $x_4,h'$ (switching the roles of $h,h'$ if necessary), and that $d_i, d_j$ are not adjacent to the apex $x_4$ of $\Sigma$.
    Since $\{d_i,d_j\} \cap N(x_4)=\emptyset$,  Theorem~\ref{apexnbrs} implies
    that $N[x_4,h,h',l_1]$ separates $d_i$ from $d_j$. Since
    $h,h',l_q \in D_2$, we are done. This proves~\eqref{H1wide}.
    
    \sta{If $(H_2,X)$ is a wide alignment, then the theorem holds.
      \label{H2wide}}
Since $(H_1,X)$ is an alignment, we deduce that 
      $(H,x_4)$ is a proper wheel. Now by Theorem~\ref{lemma:proper_wheel_forcer} and setting $Z=\{x_4\}$, we are done.
This proves~\eqref{H2wide}.

\sta{If $(H_1,X)$ is a triangular alignment, then the theorem holds.
\label{H1triangular}}

 Let $l \in \{i, \dots, j\}$ be such that
 $N(x_4) \cap D_1=\{d_l,d_{l+1}\}$.
 Since $x_2$ has a neighbor in the path $d_i \dd H_1 \dd d_l$, it
 follows that $l>i+1$. Since  $x_6$ has a neighbor in the path $d_{l+1} \dd H_1 \dd d_j$, it follows that $l<j-2$.
By Theorem~\ref{twosides_path} and in view of \eqref{H2wide},
we may assume that $(H_2,X)$ is a spiky alignment, a stellar connectifier
or a spiky connectifier. 
In all cases, $l_q$ has a unique neighbor in $P(H_2)$; denote it by $s$ (where $P(H_2)$ is as defined in Section \ref{sec:connectifier}).
Now $H \cup L$ is a pyramid $\Sigma$ with apex $s$ and base $x_4d_ld_{l+1}$.
Since $s \in D_2$ and since $x_4$ is non-adjacent to $d_i,d_j$, it
follows that $d_i$ belongs to the interior of
the path of $\Sigma$ with ends $d_l,s$, and $d_j$  belongs to the interior of
the path of $\Sigma$ with ends $d_{l+1},s$, and $d_i, d_j$ are not adjacent to the apex $s$ of $\Sigma$.
    Since $\{d_i,d_j\} \cap N(s) = \emptyset$, Theorem~\ref{apexnbrs} implies
    that $N[x_4,s,d_l,d_{l+1}]$ separates $d_i$ from $d_j$. Since
    $s \in D_2$ and $i+1 <  l < j-2$, we are done.  This proves~\eqref{H1triangular}.
    \\
    \\
    By Theorem~\ref{twosides_path} and in view of \eqref{H1wide} and \eqref{H1triangular}, we deduce that $(H_1,X)$ is a spiky alignment. Let
    $d_l$ be the unique neighbor of $x_4$ in $H_1$.
    By Theorem~\ref{twosides_path} and in view of \eqref{H2wide}, it follows that $(H_2,X)$ is either
    a triangular alignment or a triangular connectifier or a clique connectifier. Let $h,h'$ be the
    neighbors of $l_q$ in $H$. Now  $H \cup L$ is a pyramid $\Sigma$ with apex $d_l$ and base $l_qhh'$.
    Since $x_2$ has a neighbor in the path $d_i \dd H_1 \dd d_l$,
    it follows from the definition of a spiky alignment that $i<l-1$.
Similarly, $j>l+1$.
Since $h,h' \in D_2$, we deduce that
$d_i$ belongs to the interior of
the path of $\Sigma$ with ends $d_l,h$, and $d_j$  belongs to the interior of
the path of $\Sigma$ with ends $d_l,h'$ (possibly switching the roles of $h,h'$), and $d_i, d_j$ are non-adjacent to the apex $d_l$ of $\Sigma$ since   $i+1<l<j-1$. Therefore, Theorem~\ref{apexnbrs} implies
    that $N[d_l,h,h',l_q]$ separates $d_i$ from $d_j$. Since
    $h,h' \in D_2$ and $i+1<l<j-1$, we conclude that $\mathcal{C}$ is 4-amicable.

    We now prove that $N[Z]$ separates $d_1 \dd D_1 \dd d_i$ from
    $d_j \dd D_1 \dd d_k$. Let $D$ be the component of $G \setminus N[Z]$ such that $d_i \in D$, and let $D'$ be the component of
    $G \setminus N[Z]$    such that $d_j \in D'$. Then $D \neq D'$.
     Since
    $Y \subseteq D_2 \cup \{d_{i+2}, \dots, d_{j-2}\} \cup  \{x_4\}$,
    it follows that $N[Z]$ is disjoint from $d_1 \dd D_1 \dd d_j$
    and $d_j \dd D_1 \dd d_k$. Therefore, $d_1 \dd D_1 \dd d_j \in D$
    and $d_j \dd D_1 \dd d_k \in D'$ as required.
  \end{proof}

We need the following lemma: 
\begin{lemma}[Chudnovsky, Pilipczuk, Pilipczuk, Thomass\'e \cite{QPTAS}, Lemma 5.3] \label{lem:gyarfas}
    Let $G$ be a connected graph with a weight function $w$. Then there is an induced path $P = p_1 \dd \dots \dd p_k$ in $G$ such that $N[P]$ is a $w$-balanced separator. 
\end{lemma}
  
  For a graph $G$ and a set $X \subseteq G$ we denote by $\gamma(X)$
  the minimum size of a set $Y$ such that $X \subseteq N[Y]$.
  We are now ready to prove Theorem~\ref{balancedsep_2}. We prove the following strengthening, which along with Theorem~\ref{separatepath} yields Theorem~\ref{balancedsep_2} at once. This will be used in a future paper.

\begin{theorem}\label{balancedsep_3}
For every integer $m>0$ and every $m$-amicable graph class $\mathcal{G}$, there is an integer $d>0$ with the following property. Let $G \in \mathcal{G}$ and let $w$ be a normal weight function on $G$.
Then there exists $Y \subseteq V(G)$ such that
\begin{itemize}
\item $|Y| \leq d$, and
\item $N[Y]$ is a $w$-balanced separator in $G$.
\end{itemize}
\end{theorem}
\begin{proof}
It suffices to consider the unique component of $G$ with weight greater than $1/2$; if there is no such component, we set $Y = \emptyset$. Therefore, we may assume that $G$ is connected. 

Since $\mathcal{G}$ is amicable, it is also amiable. Let $K=\sigma(7)$ where $\sigma$ is as in the definition of an amiable class for $\mathcal{C}$.
Let $d=6K+2m+1$. We claim that $d$ satisfies the conclusion of the theorem.

By Lemma \ref{lem:gyarfas}, there is a path $P$ in $G$ such that 
    $w(D) \leq \frac 1 2$ for every component $D$ of $G \setminus N[P]$.
    Let $P=p_1 \dd \cdots \dd p_k$ be such a path chosen with $k$ as small as possible. It follows from the minimality of $k$ that
    there is a component $B$ of $G \setminus N[p_1 \dd P \dd p_{k-1}]$ with
      $w(B) > \frac{1}{2}$. Let $T=N(P \setminus p_k) \cap N(B)$.

\sta{Let $D$ be a connected subset of  $G$ such that $D \cap (T \cup N[p_k])=\emptyset$.
        Then $w(D) \leq \frac{1}{2}$.
              \label{meetT}}

      Since $D \cap T = \emptyset$ and $D$ is connected,
      it follows that either $D \subseteq B$
      or $D \cap B =\emptyset$. If $D \cap B =\emptyset$, then
      $w(D) \leq 1 - w(B) < \frac{1}{2}$. Thus  we may assume that
      $D \subseteq B$. Since $D \cap N[p_k] = \emptyset$, we deduce that
      $D$ is  contained in a component of $G \setminus N[P]$, and
      so $w(D) \leq \frac{1}{2}$. This proves~\eqref{meetT}.
      \\
      \\
      Let $r$ be minimum such that $\gamma[T \cap N(p_{r+1}, \dots, p_{k-1})] \leq 2K$.

\sta{We may assume that $r>0$. \label{r>1}}
Suppose $r=0$. Then 
$\gamma(T) < 2K$. Let  $Y' \subseteq V(G)$ be such that 
$T \subseteq N[Y']$ and with $|Y'| \leq 2K$.
Let $Y=Y' \cup \{p_k\}$.
Then every component of $G \setminus N[Y]$ is disjoint from
$T \cup N[p_k]$ and \eqref{meetT} implies that
$Y$ satisfies the conclusion of the theorem. This proves~\eqref{r>1}.
\\
\\
Let $Y_0' \subseteq V(G)$ be such that 
$T \cap N(p_r, \dots,p_{k-1})  \subseteq N[Y_0']$ and with $|Y_0'| \leq 2K$.
Let $Y_0=Y_0' \cup \{p_k\}$.
For every $i \in \{1, \dots, r\}$, let $\soff(i)$ be the minimum $j>i$
  such that $\gamma(T \cap N[\{p_i, \dots, p_{\soff(i)}\}]) = 2K$.
  Let $Z_i$ be a set of size at most $2K$ such that
 $T \cap N[\{p_i, \dots, p_{\soff(i)}\}] \subseteq N[Z_i]$.

  Our next goal is, for every $i \in \{1, \dots, r\}$, to define two
  sets $Q_i$ and $N_i$.  To that end, we fix  $i \in \{1, \dots, r\}$.
Let $T_1'=T \cap N(\{p_i, \dots, p_{\soff(i)}\})$.
  Let $S_1' \subseteq \{p_i, \dots, p_{\soff(i)}\}$ be such that
       $T_1' \subseteq N(S_1')$, and assume that $S_1'$ is chosen with $|S_1'|$ as small as possible. Order the vertices of $S_1'$ as $q_1, \dots, q_l$ in the order that they appear in $P$ when $P$ is traversed starting from $p_1$. It follows from the minimality of $|S_1'|$ that there is $n_1 \in T_1'$ such that
       $N_{S_1'}(n_1)=\{q_1\}$.
              Let $t(1)=1$ and let $S_1=\{q_{t(1)}\}$ and $T_1=\{n_1\}$.
       So far we have defined $S_1',S_1, T_1', T_1$ and $t(1)$.
       Suppose that we have defined  sequences of subsets $S_1', \dots, S_s'$, $S_1, \dots, S_s$, $T_1', \dots, T_s'$, and
       $T_1, \dots, T_{s}$  and a sequence of integers
       $t(1), \dots, t(s)$ such that $T_s' \subseteq N(S_s')$ and
       with the following properties (for every
       $l \in \{1, \dots, s\}$):
       \begin{itemize}
       \item $T_l$ is a stable set.
       \item If $t<t(l)$ and $q_t \in S_{l}'$, then $q_t$ is anticomplete to
         $T_{l}'$.
           \item $t(j)<t(l)$ for every $j < l$.
           \item If $j<l$, then $n_j$ is anticomplete to $S_l'$,
             and  $N_{S_l'}(n_l)=q_{t(l)}$.
          \item For every $j \in \{1, \dots, l\}$,
            $N_{S_l}(n_j)=q_{t(j)}$.
                      \end{itemize}
                \begin{figure}
                          \centering
                          \includegraphics[scale=0.8]{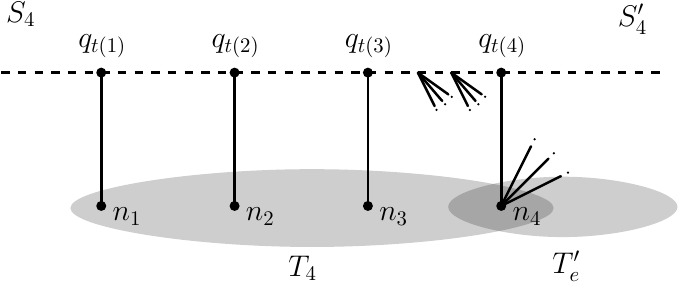}
                          \caption{Proof of Theorem \ref{balancedsep_3}.}
                          \label{fig:85}
                      \end{figure}
       See Figure \ref{fig:85}. Now we either terminate the construction or construct the sets $T_{s+1}'$, $T_{s+1}$, $S_{s+1}'$, $S_{s+1}$ and
       define $t(s+1)$,
       and verify that the properties above continue to hold.
       Let $T_{s+1}'=T_s' \setminus N[S_s \cup T_s]$.
       If $T_{s+1}'=\emptyset$, let $Q_i=S_s$ and $N_i=T_s$;
       the construction terminates here.
       Now suppose that $T_{s+1}' \neq \emptyset$.
              Let $S_{s+1}'$ be a minimal subset of $S_s'$ such that
       $T_{s+1}' \subseteq N(S_{s+1}')$. It
       follows from the second bullet and the minimality of
       $t(s)$ that if $q_t \in S_s'$ and $t<t(s)$, then $q_t$
       is anticomplete to $T_s'$. Since $S_{s+1}' \subseteq S_s'$
       and $T_{s+1}' \subseteq T_s'$, we deduce that
if $q_t \in S_{s+1}'$, then $t>t(s)$.
Let $t$ be minimum such that $q_t \in S_{s+1}'$ and set $t(s+1)=t$.
Then   $t(s+1)>t(s)$, and   the third bullet continues to hold.
By the minimality of $t$, the second bullet continues to hold.
It follows from
       the minimality of $S_{s+1}'$ that there exists 
       $n_{s+1} \in  T_{s+1}'$ such that
              $N_{S_{s+1}'}(n_{s+1})=\{q_{t(s+1)}\}$.
       Set $S_{s+1}=S_s \cup \{q_{t(s+1)}\}$ and $T_{s+1}=T_{s} \cup \{n_{s+1}\}$.
       Since $T_s$ is anticomplete to $T_{s+1}'$, the first bullet
       continues to hold.
       Since $N_{S_{s+1}'}(n_{s+1})=q_{t(s+1)}$, and
       since $S_{s+1}' \subseteq S_s' \setminus \{q_{t(s)}\}$, it follows that
the fourth bullet continues to hold, and consequently 
the        fifth bullet continues to hold.
                           Thus our construction maintains the required properties.
       When we finish, write $N_i=\{n_1, \dots, n_m\}$ and 
       $Q_i=\{q_{t(1)}, \dots, q_{t(m)}\}$, and we have:

       \sta{The following statements hold.
         \begin{itemize}
                    \item $N_i$ is a stable set.
          \item  For every $j \in \{1, \dots, m\}$,
           $N_{Q_i}(n_j)=q_{t(j)}$.
         \item If $j < j'$, then $t(j)<t(j')$.
           \item $T \cap N(\{p_i, \dots, p_{\soff(i)}\}) \subseteq N[Q_i \cup N_i]$.
           \end{itemize}
           \label{QN}}

     To simplify notation, we renumber the elements of $Q_i$, replacing each index
     $t(j)$ by $j$. In the new notation we have
     $Q_i=\{q_1, \dots, q_m\}$, and for every $j \in \{1, \dots, m\}$,
     $N_Q(n_j)=\{q_{j}\}$. Observe that  $q_1, \dots, q_m$ appear in this order
     in $P$ when $P$ is traversed from $p_1$ to $p_k$.
Let  $L_i=p_1 \dd P \dd q_{1}$, $M_i=q_1 \dd P \dd q_{m}$ and $R_i=q_{m+1} \dd P \dd p_k$.

    \sta{For every $i$, there exists $r_i \in N_i$ and
      $Y_i \subseteq  (G \setminus N[P]) \cup M_i \cup \{r_i\}$ with $|Y_i| \leq m$ such that $N[Y_i]$ separates $L_i$ from $R_i$.
    \label{window}}

    By \eqref{QN}, $N_i$ is a stable set.
    Since  $T \cap N(\{p_i, \dots, p_{\soff(i)}\}) \subseteq N[Q_i \cup N_i]$,
    and $|Q_i|=|N_i|$, it follows from \eqref{QN} that $|N_i|  \geq K$.
    Let $N_i' \subseteq N_i$ with $|N_i'|=K$. Recall that $K=\sigma(7)$.
          Now \eqref{window}
      follows from the definition of an $m$-amicable class with
      $D_1=P$, $D_2=B$ and $Y=N_i'$. This proves~\eqref{window}.
      \\
      \\
    We may assume that for every $i$ there is a component $D_i$
      of $G \setminus N[Z_i \cup Y_i]$ with $w(D_i)>\frac{1}{2}$, for
      otherwise the conclusion of Theorem~\ref{balancedsep_2} holds
      setting $Y=Z_i \cup Y_i$.  By \eqref{window}, either $D_i$ is anticomplete
      to $L_i$, or $D_i$ is anticomplete to $R_i$.
      Let us say that $Q_i$ is of {\em type $L$} if $D_i$ is anticomplete to
      $L_i$, and that $Q_i$ is of {\em type $R$} if $D_i$ is anticomplete to
      $R_i$. By \eqref{window}, every $Q_i$ is of at least one type.

      \sta{$Q_1$ is of type $L$. \label{Q1}}

      By \eqref{meetT},  $D_1 \cap (T \cup N[p_k]) \neq \emptyset$.
      It follows that $D_1$ contains a vertex $t$ such that  $t \in N[p_k]$
      or 
       $t \in T \setminus N[Z_1]$.
      Since $T \cap N(\{p_1, \dots, p_{\soff(1)}\}) \subseteq N[Z_1]$,
      and since $L_1 \cup M_1$ is a subpath of $p_1 \dd P \dd p_{\soff(1)}$, 
      it follows that $D_1$ contains a vertex  of $N[R_1]$.
Consequently, 
            $Q_1$ is not of type $R$, and so $Q_1$ is of type $L$.
            This proves~\eqref{Q1}.

     \sta{We may assume that $Q_{r}$ is of type $R$.\label{Q_last}}

     Suppose that $Q_r$ is of type $L$. Then
     $D_r$ is anticomplete to $L_r$. We claim that the set 
     $Y=Z_r \cup Y_r \cup Y_0$ satisfies the conclusion of the theorem.
     Let $D$ be a component of $G \setminus N[Y]$, and suppose
     that $w(D) > \frac{1}{2}$. 
     By \eqref{meetT}  there exists $t \in D  \cap (T \cup N[p_k])$.
Since $Z_r \cup Y_r \subseteq Y$,
     we have that $D \subseteq D_r$. It follows that $D$ is anticomplete
     to $L_r$, and so $t \not \in N(L_r)$. Since
     $T \cap N(\{p_r, \dots, p_{\soff(r)}\}) \subseteq N[Z_r]$,
     it follows that $t \not \in \{p_r, \dots, p_{\soff(r)}\}$.
       We deduce that $t \in N(\{p_{\soff(r)+1}, \dots, p_k\})$.
     But then $t \in N[Y_0]$, contrary to the fact that $Y_0 \subseteq Y$.
     This proves the claim, and     \eqref{Q_last} follows.
     \\
     \\    
       From now on we make the assumption of \eqref{Q_last}.
     By \eqref{Q_last} we can choose $j$ minimum such that $Q_j$ is
     of type $R$. By \eqref{Q1}, $j>1$. Let
     $$Y=Z_{j-1}  \cup Y_{j-1} \cup Z_j  \cup Y_j \cup Y_0.$$
Then $|Y| \leq 6k+2m+1=d$.
     We claim that $w(D) \leq \frac{1}{2}$ for every component $D$ of
     $G \setminus N[Y]$.  
     Suppose not, and let $D$ be a component of $G \setminus N[Y]$ with
     $w(D)>\frac{1}{2}$.
     Since $Z_{j-1}  \cup Y_{j-1} \subseteq Y$, we have that
     $D \subseteq D_{j-1}$. Since $Q_{j-1}$ is of type $L$,
     it follows that $D$ is anticomplete to $L_{j-1}$.
          Since
     $Z_j \cup Y_{j} \subseteq Y$, we have that
          $D \subseteq D_{j}$. Since $Q_j$ is of type $R$, it follows that
$D$ is anticomplete to $R_j$.
          Since
          $T \cap N(\{p_{j-1}, \dots, p_{\soff(j)}\}) \subseteq N[Z_{j-1} \cup Z_j]$ and since $p_k \in Y_0$, we deduce that $D \cap (T \cup N[p_k])=\emptyset$.
                    But then
     $w(D) < \frac{1}{2}$ by \eqref{meetT}, a contradiction.
         \end{proof}
    
\section{From balanced to small}
\label{sec:smallsep}
The second step in the proof of Theorem~\ref{main} is to transform
balanced separators with small domination number into balanced separators
of small size. We show:

\begin{theorem}
  \label{smallsep}
  Let $t$ be an integer, let $d$ be as in Theorem~\ref{balancedsep_2} and let $c_t$ be as in
  Theorem~\ref{banana}.
Let $G \in \mathcal {C}_t$ and let $w$ be a weight function on $G$. Then there exists $Y \subseteq V(G)$ such that
\begin{itemize}
\item $|Y| \leq 3c_td \log n+t$, and
\item $Y$ is a $w$-balanced separator in $G$.
\end{itemize}
\end{theorem}

The proof uses  Theorems~\ref{banana} and \ref{balancedsep_2}.
In order to prove Theorem~\ref{smallsep}, we first prove a more general
statement (below). We will then explain how to deduce Theorem~\ref{smallsep}
from it.

\begin{theorem}
\label{bettersep}
Let $L>0$ be an integer, 
  let $G$ be a graph and let $w$ be a weight function on $G$.   Assume that
  there is no $L$-banana
  in $G$. There exists a clique $K$ in $G$ with the following property.
  Let $X \subseteq V(G) \setminus K$ be such that $w(D) \leq \frac{1}{2}$ for
  every component
  $D$ of $G \setminus (K \cup N[X])$. 
  Then there exists a balanced separator $Y$ in $G$ such that
  $|Y \setminus K| \leq 3L|X|$.
  \end{theorem}

\begin{proof}
  
 Let $(T,\chi)$ be a  $3L$-atomic
tree decomposition of $G$.
By Theorems \ref{atomictolean} and \ref{atomictotight}, we have that $(T,\chi)$ is  tight and $3L$-lean.
By Theorem \ref{center}, there exists $t_0 \in T$ such that $t_0$ is a center for $T$.
A vertex $v \in V(G)$ is {\em $t_0$-friendly}
if $v$ is not separated from $\chi(t_0) \setminus v$ by a set of size $<3L$.
We show that  $t_0$-friendly vertices have the
following important property.

\sta {If $u,v \in \chi(t_0)$ are not $t_0$-friendly, then $u$ is adjacent to $v$.
\label{bananacoop}}

Suppose that $u$ and $v$ are not $t_0$-friendly  and that
$u$ is non-adjacent to $v$. Since there is no $L$-banana in $G$,
Theorem~\ref{Menger_vertex} implies that  there exists a
set $X$ with $|X| <L$ such $X$ separates $u$ from $v$. But this contradicts
Theorem~\ref{smallsepnonsep}. This proves~\eqref{bananacoop}.
      \\
      \\
      Let $K$ be the set of all the vertices of $G$ that are not $t_0$-friendly.
      By \eqref{bananacoop}, $K$ is a clique.
      We show that $K$ satisfies the conclusion of Theorem~\ref{bettersep}.
      Let $X \subseteq G \setminus K$ be such that
 $w(D) \leq  \frac{1}{2}$ for
  every component
  $D$ of $G \setminus (K \cup N[X])$.

      For every $v \in G \setminus K$, define the set $\Delta(v)$ as follows.
      Let $X_v \subseteq V(G)$ be  such that $X_v$ separates $v$ from
      $\chi(t_0) \setminus v$, chosen with $|X_v|$ minimum. Let
      $\Delta(v)=X_v \cup \{v\}$. Then $|\Delta(v)| \leq 3L$
      and $N_{\chi(t_0)}(v) \subseteq \Delta(v)$.
            Let
            $$Y=\bigcup_{v \in X} \Delta(v).$$
            It follows that $|Y| \leq 3L|X|$.
            
            To complete the proof it is enough to show that
        $Y \cup K$ is a $w$-balanced separator of $G$.
        Suppose not, and let $D$ be a component of $G \setminus (Y \cup K)$
        with $w(D)> \frac {1} {2}$.

        \sta {If  $D \cap (G_{t_0 \rightarrow t} \setminus \chi(t_0)) \neq \emptyset$ for some $t \in N_T(t_0)$, then $X$ is     anticomplete to $D$.
                  \label{sides}}

                  Suppose that $v \in X$ has a neighbor $d$ in $D$.
                   Let $(A, X_v,B)$ be a separation such that
                  $\chi(t_0) \setminus X_v \subseteq B$ and $v \in A$.
                  Since $d \not \in X_v$ 
                  and $d$ is adjacent to $v$,
                  it follows that $d \in A$. But then, since $D$ is
                  a component of $G \setminus (K \cup Y)$ and $X_v \subseteq Y$, it follows that $D \subseteq A$. Consequently, 
                   $D \cap \chi(t_0)=\emptyset$.
        It follows that
        $D \subseteq G_{t_0 \rightarrow t} \setminus \chi(t_0)$,
        and so $w(D) \leq w( G_{t_0 \rightarrow t} \setminus \chi(t_0)) \leq \frac{1}{2}$, a contradiction. This proves~\eqref{sides}.
                
          \sta {$D \cap N[X] = \emptyset$.
          \label{nbrs}}

          Suppose there is $v \in X$ such that $v$ has a neighbor $d \in D$.
          By \eqref{sides}, $d \in \chi(t_0)$. But then $d \in \Delta(v) \subseteq Y$, a contradiction. This proves \eqref{nbrs}.
          \\
          \\
          By \eqref{nbrs}, $D$ is a component of $G \setminus (K \cup N[X])$,
    and therefore      
        $w(D) \leq \frac{1}{2}$, a contradiction.
          \end{proof}

We now prove Theorem~\ref{smallsep}.
\begin{proof}
  Let $K$ be as in Theorem~\ref{bettersep}. Let $c_t$ be as in Theorem~\ref{banana} and let $L=c_t \log n$. Let $d$ be as Theorem~\ref{balancedsep_2}.
  Define $w': V(G) \setminus K \rightarrow [0,1]$ as
  $w'(v)=\frac{w(v)}{1-w(K)}$. Then $w'$ is  a normal weight function on $G \setminus K$.
  By Theorem~\ref{balancedsep_2}, there exists $X \subseteq V(G) \setminus K$ such that
\begin{itemize}
\item $|X| \leq d$, and
\item $N[X]$ is a $w'$-balanced separator in $G \setminus K$.
\end{itemize}
It follows that $w(D) \leq \frac{1}{2}$ for every component
$D$ of $G \setminus (N[X] \cup K)$.
Now Theorem~\ref{bettersep} applied with $L = c_t \log_n$ implies that 
there exists a $w$-balanced separator $Y$ in $G$ such that
  $|Y \setminus K| \leq 3L|X|=3c_t (\log n) |X| \leq 3c_td \log n$.
Since $G \in \mathcal{C}_t$, it follows that $|Y| \leq 3c_td \log n +t$.
\end{proof}

\section{The proof of Theorem~\ref{main}} \label{sec:proof}

We are now ready to complete the proof of Theorem~\ref{main}.
The following result was originally proved by Robertson and Seymour in \cite{RS-GMII},
  and tightened by Harvey and Wood in \cite{params-tied-to-tw}.
  It was then  restated and proved in the language of $(w, c)$-balanced separators in an earlier paper of this series \cite{wallpaper}. 

\begin{theorem}[Robertson, Seymour \cite{RS-GMII}, see also \cite{wallpaper, params-tied-to-tw}]\label{lemma:bs-to-tw}
  Let $G$ be a graph, let $c \in [\frac{1}{2}, 1)$, and let $d$ be a positive integer. If $G$ has a $(w, c)$-balanced separator of size at most $d$ for every
    normal weight function $w: V(G) \to [0, 1]$, then $\tw(G) \leq \frac{1}{1-c}k$. 
\end{theorem}

We prove the following, which immediately implies  Theorem~\ref{main}.

\begin{theorem}
  \label{mainmain}
  Let $t$ be an integer. Let  let $c_t$ be as in Theorem~\ref{banana}
  and let $d$ be as in Theorem~\ref{balancedsep}.
    Then  every
 $G \in \mathcal{C}_{t}$ satisfies
    $\tw(G)\leq 6c_td \log n+2t$.
    
\end{theorem}

\begin{proof}
  By Theorem~\ref{smallsep}, for every normal weight function $w$ of $G$ there
  exists $w$-balanced separator of size at most 
  $3c_td \log n+t$. Now the result follows immediately from
  Lemma~\ref{lemma:bs-to-tw}.
  \end{proof}

\section{Algorithmic consequences}\label{algsec}
We now summarize the algorithmic consequences of our structural results, specifically of Theorems~\ref{balancedsep} and \ref{main}.

The consequences for graphs in $\mathcal{C}_{t}$ are the most immediate. 
In particular, using the factor $2$ approximation algorithm of Korhonen~\cite{Korhonen23} (or the simpler factor $4$ approximation algorithm of Robertson and Seymour~\cite{RS-GMXIII}) for treewidth, we can compute a tree decomposition of width at most $O(c_t\log n)$ in time $2^{O(c_t\log n)}n^{O(1)} = n^{O(c_t)}$. 
A number of well-studied graph problems, including \textsc{Stable Set},
\textsc{Vertex Cover}, 
\textsc{Feedback Vertex Set}
\textsc{Dominating Set} and \textsc{$r$-Coloring} (for fixed $r$)
admit algorithms with running time $2^{O(k)}n$ when a tree decomposition of $G$ of width $k$ is provided as input (See, for example, \cite{CyganFKLMPPS15} chapters 7 and 11, as well as~\cite{BodlaenderCKN15}). 
Since $2^{O(c_t\log n)} = n^{O(c_t)}$ this immediately leads to polynomial time algorithms for these problems in $\mathcal{C}_t$.
\begin{theorem}\label{algfinal}
Let $t\geq 1$ be fixed and {\rm\texttt{P}} be a problem which admits an algorithm running in time $\mathcal{O}(2^{\mathcal{O}(k)}|V(G)|)$ on graphs $G$ with a given tree decomposition of width at most $k$. Then {\rm\texttt{P}} is solvable in time $n^{\mathcal{O}(t)}$ in $\mathcal{C}_{t}$. In particular, \textsc{Stable Set}, \textsc{Vertex Cover}, \textsc{Feedback Vertex Set}, \textsc{Dominating Set} and \textsc{$r$-Coloring} (with fixed $r$) are all polynomial-time solvable in $\mathcal{C}_{t}$.
\end{theorem}
This list of problems is far from exhaustive, it is worth mentioning the work of Pilipczuk~\cite{Pilipczuk11} who provides a relatively easy-to-check sufficient condition (expressibility in Existential Counting Modal Logic) for a problem to admit such an algorithm. 

Theorem~\ref{algfinal}   implies a polynomial time algorithm for \textsc{$r$-Coloring} (with fixed $r$) in $\mathcal{C}$ (without any assumptions on max clique size) because every graph that contains a clique of size $r+1$ can not be $r$-colored. Thus, to solve
\textsc{$r$-Coloring} we can first check for the existence of an $(r+1)$-clique in time $n^{r+\mathcal{O}(1)}$. If an $(r+1)$-clique is present, report that no $r$-coloring exists, otherwise invoke the algorithm of Theorem~\ref{algfinal} with $t=r+1$. This yields the following result.

\begin{theorem}\label{algcolor}
For every positive integer $r$, \textsc{$r$-Coloring} is polynomial-time solvable in $\mathcal{C}$.
\end{theorem}

Let us now discuss another important problem, and that is  \textsc{Coloring}.
It is well-known (and also follows immediately from
 Theorem \ref{logncollections}), that for every $t$ there exists a number
$d(t)$ such that  all graphs in $\mathcal{C}_t$ have chromatic number at most $d(t)$. Also, for each fixed $r$, by Theorem  \ref{algfinal}, $r$-\textsc{Coloring} is polynomial-time solvable in $\mathcal{C}_{t}$. Now by solving  $r$-\textsc{Coloring} for every $r \in \{1, \dots, d(t)\}$, we also obtain a polynomial-time algorithm for \textsc{Coloring} in $\mathcal{C}_{t}$.

Theorem~\ref{algfinal} quite directly leads to a polynomial-time approximation scheme (PTAS) for several problems on graphs in $\mathcal{C}$. We illustrate this using {\sc Vertex Cover} as an example. 

\begin{theorem}\label{vcPTAS}
There exists an algorithm that takes as input a graph $G \in \mathcal{C}$ 
and $0 < \epsilon \leq 1$, runs in time $n^{O(c_{2/\epsilon})}$ and outputs a vertex cover of size at most $(1+\epsilon)\textsf{vc}(G)$, where $\textsf{vc}(G)$ is the size of the minimum vertex cover of $G$.
\end{theorem}

\begin{proof}
Check in time $n^{O(1/\epsilon)}$ whether $G$ contains as input a clique $C$ of size at least $2/\epsilon$. If $G$ does not have a clique of size at least $2/\epsilon$ then compute an optimal vertex cover of $G$ in time $n^{O(c_{2/\epsilon})}$ using the algorithm of Theorem~\ref{algfinal}.
If $G$ contains such a clique $C$ then run the algorithm recursively on $G-C$ to obtain a vertex cover $S$ of $G-C$ of size at most $(1+\epsilon)\textsf{vc}(G-C)$. Return $S \cup C$; clearly $S \cup C$ is a vertex cover of $G$. Furthermore, every vertex cover of $G$ (and in particular a minimum one) contains at least $|C|-1$ vertices of $C$. It follows that $\textsf{vc}(G-C) \leq \textsf{vc}(G) - |C| + 1$ and that therefore
\begin{align*}
|S \cup C| & = |S| + |C| \\
& \leq (1+\epsilon)\textsf{vc}(G - C) + |C| \\
& \leq  (1+\epsilon)(\textsf{vc}(G) - |C| + 1 ) + |C| \\
& \leq (1+\epsilon)(\textsf{vc}(G))
\end{align*}
Here the last inequality follows because for $C \geq 2/\epsilon$ it holds that $(1+\epsilon)(|C|-1) \geq |C|$.
\end{proof}

The PTAS of Theorem~\ref{vcPTAS} generalizes without modification (except for the constant $2$ in the $2/\epsilon$) to \textsc{Feedback Vertex Set}, and more generally, to deletion to any graph class which is closed under vertex deletion, excludes some clique, and admits an algorithm (for the deletion problem) with running time $2^{O(k)}n^{O(1)}$ on graphs of treewidth $k$.
Despite the tight connection between {\sc Vertex Cover} and {\sc Stable Set}, Theorem~\ref{vcPTAS} does not directly lead to a PTAS for {\sc Stable Set} on graphs in $\mathcal{C}$. On the other hand, a QPTAS for {\sc Stable Set} in $\mathcal{C}$ follows from Theorem~\ref{balancedsep} together with an argument of Chudnovsky, Pilipczuk, Pilipczuk and Thomass\'e~\cite{QPTAS}, who gave a QPTAS for {\sc Stable Set} in $P_k$-free graphs (they refer to the problem as {\sc Independent Set}). For ease of reference, we repeat their argument here.

\begin{theorem}\label{ssQPTAS}
There exists an algorithm that takes as input a graph $G \in \mathcal{C}$ 
and $0 < \epsilon \leq 1$, runs in time $n^{O(\log^2 n/\epsilon)}$ and outputs a stable set $S$ in $G$ of size at least $(1-\epsilon)\alpha(G)$, where $\alpha(G)$ is the size of the maximum size stable set in $G$.
\end{theorem}

\begin{proof}
The algorithm is recursive, taking as input the graph $G$ and also an integer $N$. Initially, the algorithm is called with $N = |V(G)|$, in all subsequent recursive calls the value of $N$ remains the same. If $G$ is a single vertex the algorithm returns $V(G)$. If $G$ is disconnected, the algorithm runs itself recursively on each of the connected components and returns the union of the stable sets returned for each of them. 

Let $d$ be the constant of Theorem~\ref{balancedsep}. If $G$ is a connected graph on at least two vertices the algorithm iterates over all subsets $S \subseteq V(G)$ of size at most $\frac{2 d \log n \log N}{\epsilon}$ and all subsets $Y$ of size at most $d$ such that each connected component of $G - (N(S) \cup N[Y])$ has at most $|V(G)|/2$ vertices. For each such pair $(S, Y)$ the algorithm calls itself recursively on $G - (N(S) \cup N[Y])$ and returns the maximum size stable set found over all such recursive calls. 

Clearly, the set returned by the algorithm is a stable set. For each pair $(S, Y)$ which results in a recursive call of the algorithm, each connected component of the graph $G - (N(S) \cup N[Y])$ that the algorithm is called on has at most $|V(G)|/2$ vertices. Therefore the running time of the algorithm is governed by the recurrence $T(n) \leq n^{\frac{2 d \log n \log N}{\epsilon} + d} \cdot n \cdot T(n/2)$ which solves to $T(n) \leq n^{O(\frac{d \log^2 n \log N}{\epsilon})}$. Since the algorithm is initially called with $N=|V(G)|$ the running time of the algorithm is upper bounded by $|V(G)|^{O(\frac{d \log^3 |V(G)|}{\epsilon})}$.

It remains to argue that the size of the stable set output by the algorithm is at least $(1-\epsilon)\alpha(G)$. To that end, we show that the size of the independent set output by a recursive call on $(G, N)$ is at least $\alpha(G)(1 - \epsilon\frac{\log |V(G)|}{\log N})$.
Let $I$ be a maximum size stable set of $G$. By a standard coupon-collector argument (see e.g. Lemma~4.1 in \cite{QPTAS}) there exists a choice of $S \subseteq I$ of size at most
$\frac{d \log N}{\epsilon} 2\log n$ such that no vertex of $G - N(S)$ has at least $|I|\frac{\epsilon}{d \log N}$ neighbors in $I$.
Let $Y$ now be the set (of size at most $d$) given by Lemma~\ref{balancedsep} applied to $G-N[S]$.
Since no vertex in $Y$ has at least $\frac{\epsilon}{d \log N}$ neighbors in $I$ it follows that $|I - (N(S) \cup N[Y])| \geq |I|(1-\frac{\epsilon}{\log N})$.
Furthermore, each connected component of $G - (N(S) \cup N[Y])$ has at most $|V(G)|/2$ vertices. By the inductive hypothesis the recursive call on $(G - (N(S) \cup N[Y]), N)$ outputs a stable set of size at least 
$$|I|\left(1-\frac{\epsilon}{\log N}\right)\left(1- \frac{\epsilon (\log |V(G)| - 1)}{\log N}\right) \geq |I|\left(1 - \epsilon\frac{\log |V(G)|}{\log N}\right)\mbox{.}$$
Since $\alpha(G) = |I|$ the recursive call on $(G, N)$ outputs a stable set of size at least $\alpha(G)(1 - \epsilon\frac{\log |V(G)|}{\log N})$, as claimed. 
\end{proof}

\section{Acknowledgments}
We are grateful to Tara Abrishami and Bogdan Alecu for their  involvement
in early stages of this work. We also thank 
Kristina Vu\v{s}kovi\'c for many years of  conversations about the  structure of
even-hole-free graphs.


\begin{thebibliography} {99}
  
 \bibitem{AAKST}
 P. Aboulker, I. Adler, E. J. Kim, N. L. D. Sintiari, and N. Trotignon.
``On the tree-width of even-hole-free graphs.''

\bibitem{TWVIII} T. Abrishami, B. Alecu, M. Chudnovsky, S. Hajebi and S. Spirkl. 
  ``Induced subgraphs and tree decompositions VIII.
 Excluding  a forest in (theta, prism)-free graphs.''
{\em https://arxiv.org/abs/2301.02138} (2023)

    
 \bibitem{TWX}  T. Abrishami, B. Alecu, M. Chudnovsky, S. Hajebi and S. Spirkl.
  ``Induced subgraphs and tree decompositions X.
Towards logarithmic treewidth in even-hole-free graphs.''
{\em https://arxiv.org/abs/2307.13684} (2023)
 
\bibitem{wallpaper}
  T. Abrishami, M. Chudnovsky, C. Dibek, S. Hajebi, P. Rz\k{a}\.{z}ewski, S. Spirkl, and K. Vu\v{s}kovi\'c. ``Induced subgraphs and tree decompositions II. Toward walls and their line graphs in graphs of bounded degree.''
  \newblock {\em Journal of Combinatorial Theory. Series B 124}, 1 (2024),
  371--403.
  
\bibitem{TWIII} T. Abrishami, M. Chudnovsky, S. Hajebi and S. Spirkl.
  ``Induced subgraphs and tree decompositions III.
  Three paths configurations and logarithmic treewidth.''
  {\em Advances in Combinatorics}  (2022).


  \bibitem{TWI} T. Abrishami, M. Chudnovsky, and K. Vu\v{s}kovi\'c.
  ``Induced subgraphs and tree decompositions I.
  Even-hole-free graphs of bounded degree.''
  {\em J. Comb. Theory Ser. B}, {\bf 157} (2022), 144--175.


  

  
 



  
\bibitem{Addario-Berry2008BisimplicialGraphs}
L. Addario-Berry, M. Chudnovsky, F. Havet, B. Reed, and P. Seymour.
``Bisimplicial vertices in even-hole-free graphs.'' {\em Journal of Combinatorial Theory, Series B 98}, 6 (2008),
  1119--1164.


\bibitem{TWXI} B. Alecu, M. Chudnovsky, S. Hajebi and S. Spirkl.
  ``Induced subgraphs and tree decompositions XI.
   Local structure for even-hole-free graphs of large treewidth.''
{\em https://arxiv.org/abs/2309.04390} (2023)



\bibitem{BD} P. Bellenbaum and R. Diestel.
  ``Two short proofs concerning tree decompositions.''
  {\em Combinatorics, Probability and Computing}, {\bf  11}
  (2002) ,  541--547.

\bibitem{BodlaenderCKN15}
H.~L.~Bodlaender, M. Cygan, S. Kratsch and J. Nederlof.
``Deterministic single exponential time algorithms for connectivity problems parameterized by treewidth.''
{\em Information and Computation}, {\bf 243}, (2015), 86--111.

\bibitem{Bodlaender1988DynamicTreewidth}
 H.~L. Bodlaender.
\newblock {``Dynamic programming on graphs with bounded treewidth.''}
\newblock Springer, Berlin, Heidelberg, (1988), pp.~105--118.

\bibitem{CdSHV}
  K.  Cameron, M. V. G. da Silva, S. Huang, and K. Vu\v{s}kovi\'c,
  ``Structure and algorithms for (cap, even hole)-free graphs'',
  {\em Discrete Mathematics} {\bf 341} (2018), 463--473.

 


  


\bibitem{kblockpaper} J. Carmesin, R. Diestel, M. Hamann, and F. Hundertmark. ``$k$-Blocks: a connectivity invariant for graphs.'' \emph{SIAM Journal on Discrete Mathematics}, \textbf{28}(4), (2014) 1876--1891.





  \bibitem{QPTAS} M. Chudnovsky, Marcin Pillipczuk, Mihal Pillipczuk  and Stephan  Thomass\'e.
``Quasi-polynomial time approximation schemes for the Maximum Weight
  Independent Set Problem in $H$-free graphs.''
  {\em https://arxiv.org/abs/1907.04585} (2019).


\bibitem{CGP} M. Conforti, B. Gerards and K. Pashkovich,
  ``Stable sets and graphs with no even holes'',
  {\em Mathematical Programming: Series A and B}, {\bf 153} (2015), 13--39.

\bibitem{CyganFKLMPPS15} 
M. Cygan, F.~V. Fomin, L. Kowalik, D. Lokshtanov, D. Marx, M. Pilipczuk, M. Pilipczuk and S. Saurabh.
``Parameterized Algorithms.'' {\em Springer} (2015).


\bibitem{daSilva2013Decomposition2-joins}
{ M.~V. da~Silva and and K. Vu{\v{s}}kovi{\'{c}}.}
``Decomposition of even-hole-free graphs with star cutsets and
  2-joins.''
{\em J. Comb. Theory, Ser. B} {\bf 104}  (2013),
  144--183.
  
\bibitem{Davies}
 J. Davies. ``Vertex-minor-closed classes are $\chi$-bounded.'' 
 {\em Combinatorica} {\bf 42} (2022), 1049-1079.

 
\bibitem{diestel} 
R. Diestel. {\em Graph Theory}. Springer-Verlag, Electronic Edition, (2005). 



  
\bibitem{ES} P. Erd\H{o}s and G. Szekeres. ``A combinatorial problem in geometry.'' {\em Compositio Math} {\bf 2} (1935), 463--470.



\bibitem{GKL} P. Gartland, T. Korhonen, and D. Lokshtanov.
  ``On Induced Versions of Menger’s Theorem on Sparse Graphs.''
  https://arxiv.org/pdf/2309.08169.pdf (2023)
  
\bibitem{subtrees} F. Gavril. ``The intersection graphs of subtrees in trees are exactly the chordal graphs.''
\newblock {\em Journal of Combinatorial Theory. Series B} {\bf 16} (1974),
  47--56.

 \bibitem{Golumbic} M. C.  Golumbic. ``Algorithmic graph theory and perfect graphs.'' {\em In Annals of Discrete
Mathematics, second edition}, {\bf 16} Elsevier Science B.V., Amsterdam, (2004).
  
\bibitem{params-tied-to-tw}
D.J. Harvey and D.R. Wood. ``Parameters Tied to Treewidth.'' {\em J. Graph Theory} {\bf 84} (2017), 364–385.


\bibitem{HNST}
  K. Hendry, S. Norin, R. Steiner, and J. Turcotte.
  ``On an induced version of Menger's theorem.''
  {\em https://arxiv.org/abs/2309.07905} (2023)


\bibitem{Korhonen23}
T. Korhonen.
``Grid induced minor theorem for graphs of small degree.''
{\em J. Comb. Theory, Ser. {B}}, {\bf 160}, (2023), 206--214.

\bibitem{KuhnOsthus} D. K\"{u}hn and D. Osthus.
  ``Induced subgraphs in $K_{s,s}$-free graphs of large average degree.''
  {\em Combinatorica} {\bf 24} (2004), 287--304.
  


 \bibitem{Le}
N.K. Le,   
``Coloring even-hole-free graphs with no star cutset'',
\url{https://arxiv.org/abs/1805.01948} 


\bibitem{lozin}
 V. Lozin, I. Razgon. \newblock {``Tree-width dichotomy.''}
{\em  European Journal of Combinatorics} {\bf 103}, (2022), 103517.

    \bibitem{Menger}
      K. Menger, ``Zur allgemeinen Kurventheorie.''
      {\em Fund. Math.} {\bf  10} (1927)  96 -- 115.
      


    \bibitem{NSS} T. Ngueyn, A. Scott, P. Seymour,
      ``A counterexampe to the coarse Menger conjecture.''
      {\em preprint} (2023)


\bibitem{Pilipczuk11}
M. Pilipczuk,
``Problems Parameterized by Treewidth Tractable in Single Exponential Time: {A} Logical Approach.'' 
{\em Proceedings of Mathematical Foundations of Computer Science 2011}, {\bf 6907}, (2011), 520--531.
      
      
\bibitem{RS-GMII} N. Robertson and P.D.~Seymour. ``Graph minors. II. Algorithmic aspects of tree-width.'' \textit{Journal of Algorithms} \textbf{7}(3) (1986): 309--322.


\bibitem{RS-GMXIII} N. Robertson and P.D.~Seymour. ``Graph Minors. XIII. The Disjoint Paths Problem.'' \textit{J. Comb. Theory, Ser. {B}}, \textbf{63}(1) (1995): 65--110.

    
 
 


  \bibitem{RS-GMXVI} N. Robertson and P.D.~Seymour. ``Graph minors. XVI. Excluding a non-planar graph.''
{\em J. Combin. Theory, Ser. B}, {\bf 89} (2003), 43--76.

  
  
\bibitem{ST} N.L.D. Sintiari and N. Trotignon.
  ``(Theta, triangle)-free and (even-hole, $K_4$)-free graphs. Part 1: Layered wheels.'' {\em  J. Graph Theory} \textbf{97} (4) (2021), 475-509. 


\bibitem{Kristina}
  K. Vu{\v{s}}kovi{\'{c}}, 
  ``Even-hole-free graphs: A survey.''
  {\em Applicable Analysis and Discrete Mathematics}, (2010), 219--240.

   \bibitem{Weissauer} D. Weißauer, ``On the block number of graphs.'' \emph{SIAM J. Disc. Math.} {\bf 33}, 1 (2019): 346--357.

      
  \end{thebibliography}
\end{document}